\documentclass[11pt]{amsart}
\pdfoutput=1
\makeatletter
\let\origsection=\section \def\section{\@ifstar{\origsection*}{\mysection}}
\def\mysection{\@startsection{section}{1}\z@{.7\linespacing\@plus\linespacing}{.5\linespacing}{\normalfont\scshape\centering\S}}
\makeatother

\usepackage{amsmath,amssymb,amsthm}
\usepackage{mathrsfs}
\usepackage{mathabx}\changenotsign
\usepackage{dsfont}
\usepackage{bbm}

\usepackage{xcolor}
\usepackage[backref=section]{hyperref}
\usepackage[ocgcolorlinks]{ocgx2}
\hypersetup{
    colorlinks=true,
    linkcolor={red!60!black},
    citecolor={green!60!black},
    urlcolor={blue!60!black},
}

\usepackage[open,openlevel=2,atend]{bookmark}

\usepackage[abbrev,msc-links,backrefs]{amsrefs}
\usepackage{doi}

\renewcommand{\PrintDOI}[1]{\doi{#1}}

\usepackage[T1]{fontenc}
\usepackage{lmodern}
\usepackage[babel]{microtype}
\usepackage[english]{babel}

\usepackage{geometry}
\numberwithin{equation}{section}
\numberwithin{figure}{section}

\usepackage{enumitem}

\let\polishlcross=\l
\def\l{\ifmmode\ell\else\polishlcross\fi}

\def\qand{\quad\text{and}\quad}

\let\emptyset=\varnothing
\let\setminus=\smallsetminus

\makeatletter
\def\moverlay{\mathpalette\mov@rlay}
\def\mov@rlay#1#2{\leavevmode\vtop{   \baselineskip\z@skip \lineskiplimit-\maxdimen
   \ialign{\hfil$\m@th#1##$\hfil\cr#2\crcr}}}
\newcommand{\charfusion}[3][\mathord]{
    #1{\ifx#1\mathop\vphantom{#2}\fi
        \mathpalette\mov@rlay{#2\cr#3}
      }
    \ifx#1\mathop\expandafter\displaylimits\fi}
\makeatother

\DeclareFontFamily{U}  {MnSymbolC}{}
\DeclareSymbolFont{MnSyC}         {U}  {MnSymbolC}{m}{n}
\DeclareFontShape{U}{MnSymbolC}{m}{n}{
    <-6>  MnSymbolC5
   <6-7>  MnSymbolC6
   <7-8>  MnSymbolC7
   <8-9>  MnSymbolC8
   <9-10> MnSymbolC9
  <10-12> MnSymbolC10
  <12->   MnSymbolC12}{}
\DeclareMathSymbol{\powerset}{\mathord}{MnSyC}{180}

\usepackage{tikz}
\usetikzlibrary{math}
\usetikzlibrary{shapes}
\usetikzlibrary{decorations.pathreplacing}
\pgfdeclarelayer{layer-2}
\pgfdeclarelayer{layer-1}
\pgfsetlayers{layer-2,layer-1,main}

\newcommand{\qedge}[7]{

	\ifx\relax#4\relax
		\def\qoffs{0pt}
	\else
		\def\qoffs{#4}
	\fi

	\def\qhedge{
		($#1+#3!\qoffs!-90:#2-#3$) --
		($#2+#1!\qoffs!-90:#3-#1$) --
		($#3+#2!\qoffs!-90:#1-#2$) -- cycle}

	\coordinate (12) at ($#1!\qoffs!90:#2$);
	\coordinate (13) at ($#1!\qoffs!-90:#3$);
	\coordinate (23) at ($#2!\qoffs!90:#3$);
	\coordinate (21) at ($#2!\qoffs!-90:#1$);
	\coordinate (31) at ($#3!\qoffs!90:#1$);
	\coordinate (32) at ($#3!\qoffs!-90:#2$);
	
	\def\nqhedge{
		(13) let \p1=($(13)-#1$), \p2=($(12)-#1$) in
			arc[start angle={atan2(\y1,\x1)}, delta angle={atan2(\y2,\x2)-atan2(\y1,\x1)-360*(atan2(\y2,\x2)-atan2(\y1,\x1)>0)}, x radius=\qoffs, y radius=\qoffs] --
		(21) let \p1=($(21)-#2$), \p2=($(23)-#2$) in
			arc[start angle={atan2(\y1,\x1)}, delta angle={atan2(\y2,\x2)-atan2(\y1,\x1)-360*(atan2(\y2,\x2)-atan2(\y1,\x1)>0)}, x radius=\qoffs, y radius=\qoffs] --
		(32) let \p1=($(32)-#3$), \p2=($(31)-#3$) in
			arc[start angle={atan2(\y1,\x1)}, delta angle={atan2(\y2,\x2)-atan2(\y1,\x1)-360*(atan2(\y2,\x2)-atan2(\y1,\x1)>0)}, x radius=\qoffs, y radius=\qoffs] --
		cycle}

		\ifx\relax#5\relax
		\def\qlwidth{1pt}
	\else
		\def\qlwidth{#5}
	\fi
	
		\ifx\relax#7\relax
		\fill \nqhedge;
	\else
		\fill[#7]\nqhedge;
	\fi

		\ifx\relax#6\relax
		\draw[line width=\qlwidth,rounded corners=\qoffs]\nqhedge;
	\else
		\draw[line width=\qlwidth,#6]\nqhedge;
	\fi
}

\newcommand{\redge}[8]{

		\ifx\relax#5\relax
		\def\qoffs{0pt}
	\else
		\def\qoffs{#5}
	\fi

				\def\rhedge{
		($#1+#4!\qoffs!-90:#2-#4$) -- 
		($#2+#1!\qoffs!-90:#3-#1$) -- 
		($#3+#2!\qoffs!-90:#4-#2$) -- 
		($#4+#3!\qoffs!-90:#1-#3$) -- cycle}

	\coordinate (12) at ($#1!\qoffs!90:#2$);
	\coordinate (14) at ($#1!\qoffs!-90:#4$);
	\coordinate (23) at ($#2!\qoffs!90:#3$);
	\coordinate (21) at ($#2!\qoffs!-90:#1$);
	\coordinate (34) at ($#3!\qoffs!90:#4$);
	\coordinate (32) at ($#3!\qoffs!-90:#2$);
	\coordinate (41) at ($#4!\qoffs!90:#1$);
	\coordinate (43) at ($#4!\qoffs!-90:#3$);
	
	\def\nrhedge{
		(14) let \p1=($(14)-#1$), \p2=($(12)-#1$) in 
			arc[start angle={atan2(\y1,\x1)}, delta angle={atan2(\y2,\x2)-atan2(\y1,\x1)-360*(atan2(\y2,\x2)-atan2(\y1,\x1)>0)}, x radius=\qoffs, y radius=\qoffs] --
		(21) let \p1=($(21)-#2$), \p2=($(23)-#2$) in 
			arc[start angle={atan2(\y1,\x1)}, delta angle={atan2(\y2,\x2)-atan2(\y1,\x1)-360*(atan2(\y2,\x2)-atan2(\y1,\x1)>0)}, x radius=\qoffs, y radius=\qoffs] --
		(32) let \p1=($(32)-#3$), \p2=($(34)-#3$) in 
			arc[start angle={atan2(\y1,\x1)}, delta angle={atan2(\y2,\x2)-atan2(\y1,\x1)-360*(atan2(\y2,\x2)-atan2(\y1,\x1)>0)}, x radius=\qoffs, y radius=\qoffs] --
		(43) let \p1=($(43)-#4$), \p2=($(41)-#4$) in 
			arc[start angle={atan2(\y1,\x1)}, delta angle={atan2(\y2,\x2)-atan2(\y1,\x1)-360*(atan2(\y2,\x2)-atan2(\y1,\x1)>0)}, x radius=\qoffs, y radius=\qoffs] --
		cycle}

		\ifx\relax#6\relax
		\def\rlwidth{1pt}
	\else
		\def\rlwidth{#6}
	\fi
	
		\ifx\relax#8\relax
		\fill \nrhedge;
	\else
		\fill[#8]\nrhedge;
	\fi

		\ifx\relax#7\relax
		\draw[line width=\rlwidth,rounded corners=\qoffs]\nrhedge;
	\else
		\draw[line width=\rlwidth,#7]\nrhedge;
	\fi
	}

\let\epsilon=\varepsilon
\let\eps=\epsilon
\let\rho=\varrho
\let\theta=\vartheta

\def\EE{{\mathds E}}
\def\NN{{\mathds N}}

\def\PP{{\mathds P}}
\def\RR{{\mathds R}}

\newcommand{\cA}{\mathcal{A}}
\newcommand{\cB}{\mathcal{B}}
\newcommand{\cC}{\mathcal{C}}

\newcommand{\cE}{\mathcal{E}}
\newcommand{\cF}{\mathcal{F}}

\newcommand{\cH}{\mathcal{H}}
\newcommand{\cI}{\mathcal{I}}

\newcommand{\cP}{\mathcal{P}}
\newcommand{\cQ}{\mathcal{Q}}

\newcommand{\cS}{\mathcal{S}}
\newcommand{\cT}{\mathcal{T}}

\newcommand{\ccP}{\mathscr{P}}

\newtheoremstyle{note}  {4pt}  {4pt}  {\sl}  {}  {\bfseries}  {.}  {.5em}          {}
\newtheoremstyle{introthms}  {3pt}  {3pt}  {\itshape}  {}  {\bfseries}  {.}  {.5em}          {\thmnote{#3}}
\newtheoremstyle{remark}  {2pt}  {2pt}  {\rm}  {}  {\bfseries}  {.}  {.3em}          {}

\theoremstyle{plain}
\newtheorem{theorem}{Theorem}[section]
\newtheorem{lemma}[theorem]{Lemma}

\newtheorem{prop}[theorem]{Proposition}

\newtheorem{cor}[theorem]{Corollary}

\newtheorem{observation}[theorem]{Observation}

\theoremstyle{note}

\theoremstyle{remark}

\usepackage{lineno}
\newcommand*\patchAmsMathEnvironmentForLineno[1]{
\expandafter\let\csname old#1\expandafter\endcsname\csname #1\endcsname
\expandafter\let\csname oldend#1\expandafter\endcsname\csname end#1\endcsname
\renewenvironment{#1}
{\linenomath\csname old#1\endcsname}
{\csname oldend#1\endcsname\endlinenomath}}
\newcommand*\patchBothAmsMathEnvironmentsForLineno[1]{
\patchAmsMathEnvironmentForLineno{#1}
\patchAmsMathEnvironmentForLineno{#1*}}
\AtBeginDocument{
\patchBothAmsMathEnvironmentsForLineno{equation}
\patchBothAmsMathEnvironmentsForLineno{align}
\patchBothAmsMathEnvironmentsForLineno{flalign}
\patchBothAmsMathEnvironmentsForLineno{alignat}
\patchBothAmsMathEnvironmentsForLineno{gather}
\patchBothAmsMathEnvironmentsForLineno{multline}
}

\usepackage{scalerel}

\makeatletter
\newcommand{\overrighharpoonup}[1]{\ThisStyle{%
 \vbox {\m@th\ialign{##\crcr
 \rightharpoonupfill \crcr
 \noalign{\kern-\p@\nointerlineskip}
 $\hfil\SavedStyle#1\hfil$\crcr}}}}

\def\rightharpoonupfill{%
$\SavedStyle\m@th\mkern+0.8mu\cleaders\hbox{$\shortbar\mkern-4mu$}\hfill\rightharpoonuptip\mkern+0.8mu$}

\def\rightharpoonuptip{%
 \raisebox{\z@}[2pt][1pt]{\scalebox{0.55}{$\SavedStyle\rightharpoonup$}}}

\def\shortbar{%
 \smash{\scalebox{0.55}{$\SavedStyle\relbar$}}}
\makeatother

\def\MR#1{\relax}

\usepackage{adjustbox}
\usepackage{mathtools}
\newcommand{\defn}{\emph}

\newcommand{\bR}{\RR}
\newcommand{\cW}{\mathcal{W}}
\newcommand{\cM}{\mathcal{M}}
\newcommand{\cZ}{\mathcal{Z}}
\newcommand{\ordsubs}[2]{{#1}^{\underline{#2}}}
\newcommand{\unordsubs}[2]{\binom{#1}{#2}}
\newcommand{\intersecting}[1]{{#1}-intersecting}
\newcommand{\rob}{\intersecting{$\eta$}}

\newcommand{\rightharpoonupline}{\mathchoice%
	{\clipbox{{0.0\width} {0.3\height} {0.7\width} {-0.425\height}}{$\scriptstyle\rightharpoonup$}}
	{\clipbox{{0.0\width} {0.3\height} {0.7\width} {-0.425\height}}{$\scriptstyle\rightharpoonup$}}
	{\clipbox{{0.0\width} {0.4\height} {0.7\width} {-0.425\height}}{$\scriptscriptstyle\rightharpoonup$}}
	{\clipbox{{0.0\width} {0.4\height} {0.7\width} {-0.425\height}}{$\scriptscriptstyle\rightharpoonup$}}
}

\newcommand{\rightharpoonupend}{\mathchoice%
	{\clipbox{{.675\width} {0.3\height} 0pt {-0.425\height}}{$\scriptstyle\rightharpoonup$}}
	{\clipbox{{.675\width} {0.3\height} 0pt {-0.425\height}}{$\scriptstyle\rightharpoonup$}}
	{\clipbox{{.675\width} {0.4\height} 0pt {-0.425\height}}{$\scriptscriptstyle\rightharpoonup$}}
	{\clipbox{{.675\width} {0.4\height} 0pt {-0.425\height}}{$\scriptscriptstyle\rightharpoonup$}}
}
\makeatletter
\newcommand{\overrightharpoonup}[1]{\mathchoice%
	{\vbox{\m@th\ialign{##\cr$\displaystyle\hbox{$\displaystyle\rightharpoonupline$}\mkern-1mu\cleaders\hbox{$\displaystyle\mkern-2mu\rightharpoonupline$}\hfill\mkern-2mu\rightharpoonupend$\cr\noalign{\nointerlineskip\vspace{-0pt}}$\displaystyle #1$\cr}}}
	{\vbox{\m@th\ialign{##\cr$\textstyle\hbox{$\textstyle\rightharpoonupline$}\mkern-1mu\cleaders\hbox{$\textstyle\mkern-2mu\rightharpoonupline$}\hfill\mkern-2mu\rightharpoonupend$\cr\noalign{\nointerlineskip\vspace{-0pt}}$\textstyle #1$\cr}}}
	{\vbox{\m@th\ialign{##\cr$\scriptstyle\hbox{$\scriptstyle\rightharpoonupline$}\mkern-1mu\cleaders\hbox{$\scriptstyle\mkern-2mu\rightharpoonupline$}\hfill\mkern-2mu\rightharpoonupend$\cr\noalign{\nointerlineskip\vspace{-0pt}}$\scriptstyle #1$\cr}}}
	{\vbox{\m@th\ialign{##\cr$\scriptscriptstyle\hbox{$\scriptscriptstyle\rightharpoonupline$}\mkern-1mu\cleaders\hbox{$\scriptscriptstyle\mkern-2mu\rightharpoonupline$}\hfill\mkern-2mu\rightharpoonupend$\cr\noalign{\nointerlineskip\vspace{-0pt}}$\scriptscriptstyle #1$\cr}}}
}
\makeatother

\newcommand{\ordered}[1]{\overrightharpoonup{\mathbf{#1}}}
\newcommand{\orderededge}[1]{\overrightharpoonup{#1}}
\newcommand{\orde}[1]{\overrightharpoonup{#1}}
\newcommand{\ordedges}{\overrightharpoonup{E}}
\newcommand{\unord}[1]{\mathbf{#1}}
\DeclarePairedDelimiter{\abs}{\lvert}{\rvert}

\newcommand{\ex}{\mathds{E}}
\newcommand{\pr}{\mathds{P}}
\DeclareMathOperator{\Ima}{Im}
\DeclareMathOperator{\dom}{dom}

\DeclarePairedDelimiter{\ceil}{\lceil}{\rceil}
\def\COMMENT#1{}
\def\TASK#1{}
\let\TASK=\footnote             

\newcommand{\typelo}{\textup{lo}}
\newcommand{\typecon}[1][]{\if\relax\detokenize{#1}\relax\textup{con}\else\textup{#1-con}\fi}
\newcommand{\typeend}[1][]{\if\relax\detokenize{#1}\relax\textup{end}\else\textup{#1-end}\fi}
\newcommand{\defeq}{=}
\DeclareMathOperator{\reg}{reg}

\newcommand{\boundary}[1]{\partial #1}

\allowdisplaybreaks

\definecolor{darkgreen}{rgb}{0,0.5,0}

\begin{document}

\title[Decomposing hypergraphs into cycle factors]{Decomposing hypergraphs into cycle factors}

\author[F.~Joos]{Felix Joos}
\address{Institut f\"ur Informatik, Universit\"at Heidelberg, Germany}
\email{joos@informatik.uni-heidelberg.de}

\author[M.~K\"uhn]{Marcus K\"uhn}
\address{Institut f\"ur Informatik, Universit\"at Heidelberg, Germany}
\email{kuehn@informatik.uni-heidelberg.de}

\author[B. Sch\"ulke]{Bjarne Sch\"ulke}
\address{Fachbereich Mathematik, Universit\"at Hamburg, Hamburg, Germany}
\email{bjarne.schuelke@uni-hamburg.de}


\date{\today}

\thanks{The research leading to these results was partially supported by the Deutsche Forschungsgemeinschaft (DFG, German Research Foundation) -- 428212407 (F. Joos and M. K\"uhn).}

\begin{abstract}
A famous result by R\"odl, Ruci\'nski, and Szemer\'edi guarantees a (tight) Hamilton cycle in~$k$-uniform hypergraphs $H$ on $n$ vertices with minimum~$(k-1)$-degree $\delta_{k-1}(H)\geq (1/2+o(1))n$, thereby extending Dirac's result from graphs to hypergraphs.
For graphs, much more is known;
each graph on $n$ vertices with $\delta(G)\geq (1/2+o(1))n$ contains $(1-o(1))r$ edge-disjoint Hamilton cycles
where $r$ is the largest integer such that $G$ contains a spanning $2r$-regular subgraph,
which is clearly asymptotically optimal.
This was proved by Ferber, Krivelevich, and Sudakov answering a question raised by K\"uhn, Lapinskas, and Osthus.

We extend this result to hypergraphs;
every $k$-uniform hypergraph $H$ on $n$ vertices with  $\delta_{k-1}(H)\geq (1/2+o(1))n$
contains $(1-o(1))r$ edge-disjoint (tight) Hamilton cycles
where $r$ is the largest integer such that $H$ contains a spanning subgraph with each vertex belonging to~$kr$ edges.
In particular, this yields an asymptotic solution to a question of Glock, K\"uhn, and Osthus.

In fact, our main result applies to approximately vertex-regular $k$-uniform hypergraphs with a weak quasirandom property
and provides approximate decompositions into cycle factors without too short cycles.
\end{abstract}

\maketitle

\section{Introduction}

Decompositions are a very active branch of extremal combinatorics.
One of the earliest results regarding decompositions of graphs is Walecki's theorem, which states that a complete graph on an odd number of vertices has a decomposition into (edge-disjoint) Hamilton cycles.
In recent years, there have been many breakthroughs in the area of decompositions, 
such as the verification of the existence of designs~\cite{keevash:14,GKLO:16}, 
the resolution of the Oberwolfach problem~\cite{GJKKO:ta}, 
and the proof of Ringel's conjecture~\cite{MPS:20,KS:20b}.

A classic result by Dirac states that a graph on~$n\geq 3$ vertices with minimum degree at least~$n/2$ contains a Hamilton cycle.
It is natural to ask how many edge-disjoint Hamilton cycles exist in this setting.
Nash-Williams~\cite{NW:71} showed that there are $\lfloor 5n/224 \rfloor$ edge-disjoint Hamilton cycles.
As the union of edge-disjoint Hamilton cycles is an even-regular spanning subgraph,
there are at most $r/2$ edge-disjoint Hamilton cycles
where~$r$ is the largest even integer for which there exists an $r$-regular spanning subgraph.
Thus, for a graph~$G$, we define $\reg_{2}(G)$ to be the largest even integer $r$ such that~$G$ contains a spanning $r$-regular subgraph
and set $\reg_{2}(n,\delta)\defeq \min \{\reg_{2}(G)\colon |V(G)|=n,\delta(G)=\delta\}$.
Csaba, K\"uhn, Lo, Osthus, and Treglown~\cite{CKLOT:16} 
improved the result by Nash-Williams by showing that all large graphs $G$ on $n$ vertices contain at least $\reg_{2}(n,\delta(G))/2$ edge-disjoint Hamilton cycles.
K\"uhn, Lapinskas,  and Osthus~\cite{KLO:13} conjectured that this can be strengthened as each single~$G$ may have $\reg_{2}(G)/2$ edge-disjoint Hamilton cycles provided $\delta(G)\geq n/2$.
They also asked whether an approximate version is true; namely, that any graph $G$ on $n$ vertices with $\delta(G)\geq (1/2+o(1))n$ contains $(1/2-o(1))\reg_{2}(G)$ edge-disjoint Hamilton cycles.
Subsequently, this was proved by Ferber, Krivelevich, and Sudakov~\cite{FKS:17}.

The main result of this paper implies an analogous statement for $k$-uniform hypergraphs for~$k\geq 2$.
To state our results, we recall some terminology.
For an integer~$k\geq 2$, a hypergraph~$H$ is called~\emph{$k$-uniform hypergraph} or~\emph{$k$-graph} if all its edges have size~$k$.
A $k$-graph whose vertex set has a cyclic ordering such that its edge set consists of all sets of $k$ consecutive vertices in this ordering is called a \emph{(tight) cycle} (we only consider tight cycles in this article).
The \emph{length} of a cycle $C$ is defined as the number of edges in $C$.
As usual, a \emph{Hamilton cycle} in~$H$ is a cycle containing all vertices of~$H$.
Let $\delta_{k-1}(H)=\min\vert\{e\in E(H): \unord{x}\subseteq e\}\vert$ where the minimum is taken over all $(k-1)$-sets $\unord{x}\subseteq V(H)$.
In analogy to the above,
we define $\reg_k(H)$ as the largest integer $r$ divisible by $k$ such that $H$ contains a spanning subgraph $F$ in which each vertex of $F$ belongs to exactly $r$ edges of $F$.

Dirac's result was first generalised to hypergraphs by R\"odl, Ruci\'nski, and Szemer\'edi~\cite{RRS:06,RRS:11,RRS:08}. 
They showed that any $k$-graph~$H$ on~$n$ vertices with $\delta_{k-1}(H)\geq (\frac{1}{2}+o(1))n$ contains a Hamilton cycle.
Observe that, trivially, $H$ contains at most $\reg_k(H)/k$ edge-disjoint Hamilton cycles.
Our main result implies that~$H$ indeed asymptotically contains that many Hamilton cycles. 
More precisely, it yields the following strengthening of the result by R\"odl, Ruci\'nski, and Szemer\'edi.

\begin{theorem}\label{thm: simple min_deg}
	For all $k$, $\eps>0$, there exists $n_0$ such that
	every $k$-graph $H$ on $n\geq n_0$ vertices with $\delta_{k-1}(H)\geq (1/2+ \eps)n$ contains $(1-\eps)\reg_k(H)/k$ edge-disjoint Hamilton cycles.
\end{theorem}

This asymptotically solves (a much stronger version of) a conjecture due to Glock, K\"uhn, and Osthus~\cite[Conjecture 6.6]{GKO:20}
which states that if we in addition to the assumptions in Theorem~\ref{thm: simple min_deg} assume that each vertex is contained in the same number of edges and $k|n$, then $H$ has a decomposition into perfect matchings.
(Observe that in this case a Hamilton cycle contains $k$ edge-disjoint perfect matchings.)

In fact, there is no need to restrict our attention only to Hamilton cycles.
We call a $k$-graph~$C$ a \emph{cycle factor} (with respect to $H$) if $C$ is a union of vertex-disjoint cycles and has the same number of vertices as $H$.
The \emph{girth} of a cycle factor is the length of its shortest cycle.

\begin{theorem}\label{thm: simple2 min_deg}
	For all $k$, $\eps>0$,
	there exist $n_0$, $L$ such that every $k$-graph $H$ on $n\geq n_0$ vertices with $\delta_{k-1}(H)\geq (1/2+ \eps)n$ contains edge-disjoint copies of 
	the members of any family of $(1-\eps)\reg_k(H)/k$ cycle factors whose girth is at least $L$.
\end{theorem}

To the best of our knowledge, this is not even known for a single cycle factor.
Note that~$L$ has to grow as~$k$ grows and~$\eps$ shrinks, but we do not require any dependence on~$n$.

As it turns out, we can restrict our attention to essentially vertex-regular $k$-graphs.
To this end, 
call a~$k$-graph~$H$ on $n$ vertices \emph{$\rho$-almost $r$-regular} for some~$\rho, r \geq 0$ 
if~$d_H(v)=(1\pm\rho)r$ for all~$v\in V(H)$ and $\rho$-almost regular if it is~$\rho$-almost~$r$-regular for some~$r\geq0$.
Note that $\rho$-almost $r$-regular $k$-graphs may simply be the disjoint union of two cliques, say, and thus they may not even contain a single Hamilton cycle.
To avoid such scenarios, we work with the following fairly weak quasirandomness property.
For~$\unord{x}=\{x_1,\ldots,x_{k-1}\}\in\unordsubs{V(H)}{k-1}$, we write~\defn{$N_H(\unord{x})$} for the neighbourhood of~$\unord{x}$ in~$H$, 
that is, the set~$\{ v\in V(H): \{v,x_1,\ldots,x_{k-1}\}\in E(H) \}$.
Define~$H$ to be~\emph{\intersecting{$\eta$}} 
if for all~$\unord{x},\unord{y}\in \unordsubs{V(H)}{k-1}$, 
we have~$\abs{N(\unord{x})\cap N(\unord{y})}\geq \eta n$.
Considering a complete graph on $n$ vertices where we delete the edges of a clique on $(1-1/k+o(1))n$ vertices
implies that being~$(1/k-o(1))$-intersecting alone is not sufficient to ensure the existence of a Hamilton cycle either.

The following theorem is our main result and Theorems~\ref{thm: simple min_deg} and~\ref{thm: simple2 min_deg} follow from it.
\begin{theorem}\label{thm: main}
	For all $k$, $\eta,\varepsilon>0$,
	there exist $L,n_0,\rho>0$ such that
	every $\eta$-intersecting $\rho$-almost $r$-regular $k$-graph $H$ on $n\geq n_0$ vertices 
	contains edge-disjoint copies of the members of any family of $(1-\eps)r/k$ cycle factors whose girth is at least~$L$.
\end{theorem}

A result of Ferber, Krivelevich, and Sudakov~\cite{FKS:16} implies that any $k$-graph $H$ on~$n$ vertices with $\delta_{k-1}(H) \geq (1/2 + o(1))n$,
contains an $n^{-1/2}$-almost $r$-regular spanning subgraph $F$ for some $r>\frac{1}{8}\binom{n}{k-1}$.
This $F$ may not be $\eta$-intersecting for some $\eta>0$,
but the next result shows that this is not an obstacle for the application of Theorem~\ref{thm: main}.

\begin{lemma}\label{lemma: almost regular subgraph}
	For all~$k,\eps>0$, there exist~$n_0,\eta>0$ such that every~$k$-graph~$H$ on~$n\geq n_0$ vertices with~$\delta_{k-1}(H)\geq (\frac{1}{2}+\eps)n$ which contains a~$\rho$-almost~$r$-regular spanning subgraph for some~$\rho\in[0,1/2]$ and~$r\geq 0$ 
	also contains an~$\eta$-intersecting~$(\rho+n^{-1/3})$-almost~$r'$-regular spanning subgraph for some~$r'\geq \max\{(1-\eps)r,\frac{1}{8}\binom{n}{k-1}\}$.
\end{lemma}

In particular, Theorem~\ref{thm: main} together with Lemma~\ref{lemma: almost regular subgraph} implies Theorem~\ref{thm: simple2 min_deg} (and thereby Theorem~\ref{thm: simple min_deg}).
In addition,
for $k$-graphs $H$ on $n$ vertices with $\delta_{k-1}(H) \geq (1/2 + o(1))n$,
we have $\reg_k(H)= (1-o(1))r'$ where $r'$ is the largest integer such that $H$ contains an $o(1)$-almost $r'$-regular spanning subgraph.

Theorem~\ref{thm: main} for $k=2$ is an implication of a bandwidth theorem for approximate decompositions proved only recently by Condon, Kim, K\"uhn, and Osthus~\cite{CKKO:19},
explicitly mentioned as a statement in~\cite{GKO:20}.

For graphs, there are many results that iteratively made improvements on the maximal number of edge-disjoint Hamilton cycles in graphs $G$ on $n$ vertices with $\delta(G)\geq (1/2+o(1))n$, see~\cite{CKO:12,CKLOT:16,HS:13,KLO:13,KO:13,KO:14};
however, for hypergraphs, 
we are not aware of any result that provides many edge-disjoint tight Hamilton cycles under non-trivial minimum degree assumptions.
For results concerning decompositions into tight Hamilton cycles of $k$-graphs with fairly strong quasirandom properties, we refer the reader to~\cite{BF:12,EJ:20}.

We prove Theorem~\ref{thm: main} in Section~\ref{sec: comb} and Lemma~\ref{lemma: almost regular subgraph} in Section~\ref{sec: almost regular with intersecting}.

\section{Proof sketch of Theorem~\ref{thm: main}}
Suppose we are in the setting of Theorem~\ref{thm: main};
that is, suppose~$H$ is an $\eta$-intersecting $\rho$-almost $r$-regular $k$-graph in which we aim to find edge-disjoint copies of the cycle factors. 
Our argumentation is built on three stages which are described in Sections~\ref{sec:approx}--\ref{sec:approxtoreal}.

With some foresight, we set aside a thin randomly selected spanning subgraph~$F$ of~$H$; in particular,~$F$ is~$\eta'$-intersecting and~$\rho'$-almost regular for some~$\eta'<\eta$ and~$\rho'>\rho$.

In the first stage, we only consider the $k$-graph $H'\defeq H- F$.
For large $L$ (but which does not grow with $n$)
the $k$-graph $H'$ has a fractional decomposition into cycles of length $L$, by a recent result in~\cite{JK:21} (see Theorem~\ref{theorem: fractional cycle decomposition}).
Next, 
we exploit a result about hypergraph matchings with pseudorandom properties~\cite{EGJ:19a} (see Theorem~\ref{theorem: pseudorandom matching} and Corollary~\ref{cor: simple weighted pseudorandom matching}) 
to turn this fractional decomposition of $H'$ into 
edge-disjoint collections~$\cP_1,\dots,\cP_r$ of vertex-disjoint paths of length $L$ such that~$V(H)\setminus\bigcup_{P\in\cP_i}V(P)$ is very small for each~$i\in[r]$ and~$E(H')\setminus\bigcup_{i\in[r]}E(\cP_i)$ is very small as well (see Proposition~\ref{prop: simultaneous path cover}).
This completes the first stage.

The second stage in our approach deals with the question of how one can turn a single $\cP_i$ into a particular cycle factor $\cC_i$ (see Lemma~\ref{lem:layer}).
For this we use the edges in $F$.
One might hope that one can proceed similarly as R\"odl, Ruci\'nski, and Szemer\'edi in~\cite{RRS:06,RRS:11,RRS:08}
to join up paths and absorb the remaining vertices to obtain the desired cycle factor.
However, 
as the cycles in~$\cC_i$ may be very short, we cannot utilize an absorbing path of length $o(n)$ into which we could incorporate any small set of remaining vertices, simply because there may not exist a cycle in the desired cycle factor which is long enough to contain such a path.
If we split the absorbing path into subpaths and distributed these among numerous cycles in the cycle factor,
we would have too little control over how many vertices each cycle actually incorporates and,
thereby, over the resulting cycle lengths.

To overcome this, we prepare by grouping small absorbing elements (paths on $2k$ vertices) into more powerful absorbers, which we call \emph{blocks} (see Section~\ref{sec: abs}).
The crucial property is that in the end, regardless of the leftover vertices, each block absorbs exactly one vertex.
Thus, a cycle of length~$\ell-b$ obtained by connecting paths in~$\cP_i$ and~$b$ blocks turns into a cycle of length~$\ell$ during the absorption of the leftover vertices.
Hence, keeping this predictable change of lengths in mind, we can construct an almost spanning collection of vertex-disjoint cycles in such a way that the absorption of the leftover vertices engenders a cycle factor as desired.

The third stage in our argumentation deals with the task of repeating the second stage for every $i\in [r]$.
Proceeding in a greedy fashion, iteratively considering each $i\in [r]$, may quickly ruin the quasirandom properties of $F$.
Therefore, we actually provide all tools from the second stage as probabilistic constructions.
With this, we can ensure a fairly uniform use of the edges in $F$
when applying the arguments of the second stage iteratively for each $i\in [r]$.
By using Freedman's inequality, one observes that with positive probability this process terminates successfully before significantly spoiling the quasirandomness of $F$ (see Proposition~\ref{prop:cycdecomptoHCdecomp}).

\section{Preliminaries}\label{sec: prelim}

\subsection{Notation}

For~$n\in\NN_0$, we set~\defn{$[n]\defeq \{1,\ldots,n\}$} and~$[n]_0\defeq \{0,\ldots,n\}$.
For a set~$A$, we say that~$A$ is a~\defn{$k$-set} if~$\abs{A}=k$; we write~\defn{$\unordsubs{A}{k}$} for the set of~$k$-sets that are subsets of~$A$ and~\defn{$\ordsubs{A}{k}$} for the set of tuples~$(x_1,\ldots,x_k)\in A^k$ with~$x_i\neq x_j$ for all~$i\neq j$.
We often use~$\unord{x},\unord{y}$ to refer to sets and~$\ordered{x},\ordered{y}$ when considering tuples; however, if the tuple arises from ordering the vertices of an edge, then we often use~$\orderededge{e},\orderededge{f}$.
We may subsequently drop the arrow to denote the set of elements of a tuple, so that for a tuple~$\ordered{x}=(x_1,\ldots,x_k)$, we have~$\unord{x}=\{x_1,\ldots,x_k\}$.
An \defn{ordering} of a $k$-set $\unord{x}=\{x_1,\ldots,x_k\}$ is a sequence $x_1\ldots x_k$ without repetitions.

For non-negative reals~$\alpha,\beta,\delta,\delta'$, we write~\defn{$\alpha=(1\pm \delta)\beta$} to mean~$(1-\delta)\beta \leq \alpha \leq (1+\delta)\beta$ and we write~\defn{$(1\pm \delta)\alpha=(1\pm \delta')\beta$} to mean~$(1-\delta')\beta\leq (1-\delta)\alpha\leq (1+\delta)\alpha\leq (1+\delta')\beta$.
We write \defn{$\alpha\ll\beta$} to mean that there is a non-decreasing function~$\alpha_0\colon(0,1]\rightarrow(0,1]$ such that for any $\beta\in(0,1]$, the subsequent statement holds for all~$\alpha\in(0,\alpha_0(\beta)]$.
Hierarchies with more constants are defined similarly and should be read from the right to the left.
Constants in hierarchies will always be reals in~$(0,1]$.
Moreover, if~$1/n$ appears in a hierarchy, this implicitly means that~$n$ is a positive integer.
We ignore rounding issues when they do not affect the argument.

Whenever we use~$k$ to refer to the uniformity of a hypergraph, we tacitly assume that~$k\geq 2$.
Let~$H$ be a~$k$-graph on~$n$ vertices.
We write~\defn{$V(H)$} for the vertex set and~\defn{$E(H)$} for the edge set of~$H$ and we define~$\orderededge{E}(H)\defeq\{\orderededge{e}\in\ordsubs{V(H)}{k}\colon e\in E(H)\}$.
For~$j\in[k-1]$ and~$\unord{x}=\{x_1,\ldots,x_j\}\in\unordsubs{V(H)}{j}$, we write~\defn{$d_H(\unord{x})$} or~\defn{$d_H(x_1\ldots x_j)$} for the~$j$-degree~$\abs{\{ e\in E(H)\colon\unord{x}\subseteq e \}}$ of~$\unord{x}$,~\defn{$\delta_j(H)$} for the minimum~$j$-degree~$\min\bigl\{d_H(\unord{x})\colon\unord{x}\in\unordsubs{V(H)}{j}\bigr\}$ of~$H$ and~\defn{$\Delta_j(H)$} for the maximum~$j$-degree~$\max\bigl\{d_H(\unord{x})\colon\unord{x}\in\unordsubs{V(H)}{j}\bigr\}$ of~$H$.
We define~$\delta(H)\defeq \delta_{k-1}(H)$ and~$\Delta(H)\defeq \Delta_{k-1}(H)$.
The~$k$-graph~$H$ is \defn{vertex-regular} if there is an~$r\geq 0$ such that~$d_H(v)=r$ holds for all~$v\in V(H)$.

For~$U\subseteq V(H)$, we write~\defn{$H[U]$} for the induced subgraph~$(U,\{e\in E(H)\colon e\subseteq U \})$ and if~$X$ is a set, we define~\defn{$H-X\defeq H[V(H)\setminus X]$}.
For two~$k$-graphs~$H_i$ with vertex set~$V_i$ and edge set~$E_i$ for~$i\in[2]$, we define~\defn{$H_1-H_2\defeq (V_1,E_1\setminus E_2)$} and~\defn{$H_1\cap H_2\defeq (V_1\cap V_2, E_1\cap E_2)$} and we write~\defn{$H_1\subseteq H_2$} to indicate that~$H_1$ is a subgraph of~$H_2$.

A \defn{walk}~$W$ in~$H$ is a sequence~$w_1\ldots w_{\ell}$ of vertices of~$H$ such that~$\{w_{i},\ldots,w_{i+k-1}\}$ is an edge of~$H$ for all~$i\in[\ell-k+1]$;
we say that~$W$ is an~$\ell$-walk.
The \defn{length} of~$W$ is~$\ell-k+1$ and if~$\ell\geq k$, the walk~$W$ is a walk from~$(w_1,\ldots,w_{k})$ to~$(w_{\ell-k+1},\ldots,w_{\ell})$.
The walk~$W$ is \defn{self-avoiding} if no vertex of~$H$ appears twice in~$W$.

A~$k$-graph~$P$ on is called a \defn{path} if there is an ordering~$v_1\ldots v_{\ell}$ of its vertex set such that a~$k$-set forms an edge of~$P$ if and only if its elements appear consecutively in~$v_1\ldots v_{\ell}$.
We say that~$P$ is an~\defn{$\ell$-path} and the \defn{length} of~$P$ is~$\abs{E(P)}$.
A cycle of length~$\ell$ is also called an~\defn{$\ell$-cycle}.
Sometimes~$P$ is identified with the sequence~$v_1\ldots v_{\ell}$.
Further, we call the tuples $(v_1,\dots ,v_{i})$ and~$(v_{\ell-k+1},\dots ,v_{\ell})$ with~$i\in[\ell]$~\defn{end-tuples} of~$P$ and whenever~$\ordered{x}$ is an end-tuple of~$P$, the set~$\unord{x}$ is an \defn{end-set} of~$P$.
End-tuples of~$P$ that are~$k$-tuples are also called \defn{ordered end-edges} of~$P$ and end-sets of~$P$ that are~$k$-sets are also called \defn{end-edges} of~$P$.
For end-tuples~$\ordered{s}$ and~$\ordered{t}$ of~$P$ with~$\unord{s}\cap\unord{t}=\emptyset$, the graph~$P$ is also called an~$\ordered{s}$--$\ordered{t}$-path.
For ordered end-edges~$\orderededge{s}$ and~$\orderededge{t}$ of~$P$ with~$s\cap t=\emptyset$, we sometimes arbitrarily fix a direction of~$P$ by saying that~$\ordered{s}$ is the~\defn{ordered starting edge} and~$\ordered{t}$ the~\defn{ordered ending edge} of~$P$.


A \defn{matching}~$\cM$ in~$H$ is a set of disjoint edges of~$H$ and~$\cM$ is \defn{perfect} if all vertices of~$H$ belong to an edge in~$\cM$.
We also treat a perfect matching~$\cM$ in a bipartite graph~$G$ with bipartition~$\{U,V\}$ as a bijection~$\mu\colon A\rightarrow B$, this means that~$\mu$ is the bijection with~$\{u,\mu(u)\}\in\cM$ for all~$u\in U$.
A \defn{perfect fractional matching} in~$H$ is a function~$\omega\colon E(H)\rightarrow[0,1]$ with~$\sum_{e\in E(H)\colon v\in e} \omega(e)=1$ for all~$v\in V(H)$.

Sometimes we identify a set~$\cH$ of~$k$-graphs on disjoint vertex sets with the~$k$-graph with vertex set~$\bigcup_{H\in\cH}V(H)$ and edge set~$\bigcup_{H\in\cH} E(H)$; in this case we refer to~$\cH$ as a \defn{collection} of~$k$-graphs.
For a collection~$\cH$ of~$k$-graphs, we use~$\abs{\cH}$ to refer to the size of~$\cH$ as a set.

\subsection{Concentration inequalities}
We use the following versions of Chernoff's, McDiarmid's and Freedman's inequality.
\begin{lemma}[Chernoff's inequality]\label{lemma: chernoff}
	Suppose~$X_1,\ldots,X_n$ are independent Bernoulli random variables and let~$X\defeq \sum_{i\in[n]} X_i$.
	Then, for all~$\delta>0$,
	\begin{equation*}
		\pr[X\geq (1+\delta)\ex[X]]\leq\exp\biggl(-\frac{\delta^2}{2+\delta}\ex[X]\biggr),
	\end{equation*}
	and, if~$0<\delta<1$, then 
	\begin{equation*}
		\pr[X\leq (1-\delta)\ex[X]]\leq\exp\biggl(-\frac{\delta^2}{2}\ex[X]\biggr).
	\end{equation*}
\end{lemma}

\begin{lemma}[McDiarmid's inequality]\label{lemma: McDiarmid}
	Suppose~$X_1,\ldots,X_n$ are independent random variables and let~$f\colon \Ima(X_1)\times\ldots\times\Ima(X_n)\rightarrow\bR$.
	Assume for all~$i\in[n]$ that changing the~$i$-th coordinate of~$\ordered{x}\in\dom(f)$ changes~$f(\ordered{x})$ by at most~$c_i>0$.
	Then, for all~$\mu>0$,
	\begin{equation*}
	\pr[\abs{f(X_1,\ldots,X_n)-\ex[f(X_1,\ldots,X_n)]}\geq\mu]\leq 2\exp\left(-\frac{2\mu^2}{\sum_{i\in[n]} c_i^2}\right).
	\end{equation*}
\end{lemma}

\begin{lemma}[Freedman's inequality~\cite{F:75}]\label{lem:freedman}
	Suppose~$X_1,\ldots,X_n$ are Bernoulli random variables and let~$\mu>0$ with~$\sum_{i\in[n]}\ex[X_i\mid X_1,\ldots,X_{i-1}]\leq \mu$. Then,
	\begin{equation*}
		\pr\Bigl[\sum_{i\in[n]} X_i\geq 2\mu\Bigr] \leq \exp\biggl(-\frac{\mu}{6}\biggr).
	\end{equation*}
\end{lemma}

\subsection{Fractional cycle decompositions}
Suppose~$H$ and~$F$ are a~$k$-graphs and~$\cF$ is the set of copies of~$F$ in~$H$.
We say a function~$\omega\colon\cF\rightarrow[0,1]$ with~$\sum_{F'\in \cF\colon e\in E(F')} \omega(F')=1$ for all~$e\in E(H)$ is a \defn{fractional~$F$-decomposition} of~$H$.
We are only interested in the case where~$F$ is a cycle and need the following result from~\cite{JK:21}. 
\begin{theorem}[{\cite[Theorem 1.4]{JK:21}}]\label{theorem: fractional cycle decomposition}
	Suppose~$1/n\ll 1/L\ll \eta,1/k$.
	Suppose~$H$ is an~\intersecting{$\eta$}~$k$-graph on~$n$ vertices with edge set~$E$ and~$C_{L}$ is a~$k$-uniform cycle of length~$L$.
	Then, there is a fractional~$C_{L}$-decomposition~$\omega$ of~$H$ with
	\begin{equation*}
	\frac{\abs{E}}{\Delta(H)^{L}}\leq \omega(C)\leq\frac{3\abs{E}}{\delta(H)^{L}}
	\end{equation*}
	for all~$L$-cycles~$C$ in~$H$.
\end{theorem}

\subsection{Matchings in hypergraphs}
It is nowadays well-known that essentially vertex-regular hypergraphs admit almost perfect matchings provided each pair of vertices is only contained in few edges~\cite{rodl:85,PS:89}.
In fact, these matchings can be chosen in a way that they exhibit pseudorandom properties, which is very useful for applications.
The following result provides such matchings.

\begin{theorem}[{\cite[Theorem 1.2]{EGJ:19a}}]\label{theorem: pseudorandom matching}
	Suppose~$1/\Delta\ll \delta,1/k$.
	Let~$\eps\defeq \frac{\delta}{50k^2}$.
	Suppose~$H$ is a~$k$-graph with edge set~$E$,~$\Delta_1(H)\leq \Delta$, and~$\Delta_{2}(H)\leq \Delta^{1-\delta}$ as well as~$\abs{E}\leq \exp(\Delta^{\eps^2})$.
	Suppose that~$\cW$ is a set of at most~$\exp(\Delta^{\eps^2})$ functions from~$E$ to~$\bR_{\geq 0}$.
	Then, there exists a matching~$\cM$ in~$H$ such that~$\sum_{e\in\cM}\omega(e)=(1\pm\Delta^{-\eps})(\sum_{e\in E}\omega(e))/\Delta$ for all~$\omega\in\cW$ with~$\sum_{e\in E}\omega(e)\geq \max_{e\in E} \omega(e)\Delta^{1+\delta}$.
\end{theorem}

It is straightforward to turn Theorem~\ref{theorem: pseudorandom matching} into a result about edge sets in weighted hypergraphs
in the case where the functions in $\cW$ are~$\{0,1\}$-valued.
It can be obtained from Theorem~\ref{theorem: pseudorandom matching} by modelling the edge weights as corresponding numbers of edges in an auxiliary $(k+1)$-graph;
see, for example~\cite{KS:20b}, where a similar statement is deduced from Theorem~\ref{theorem: pseudorandom matching}.

\begin{cor}\label{cor: simple weighted pseudorandom matching}
	Suppose~$1/c\ll \delta,1/k$.
	Let~$\eps\defeq \frac{\delta}{200(k+1)^2}$.
	Suppose~$H$ is a~$k$-graph with vertex set~$V$ and edge set~$E$ and~$\omega\colon E\rightarrow [1/c,1]$ is a function with~$\sum_{e\in E\colon v\in e}\omega(e)\leq1$ and~$\sum_{e\in E\colon u,v\in e}\omega(e)\leq 1/c^{\delta}$ for all distinct~$u,v\in V$ as well as~$\sum_{e\in E}\omega(e)\leq\exp(c^{\eps^2})$.
	Suppose that~$\cE$ is a family of at most~$\exp(c^{\eps^2})$ subsets of~$E$ with~$\sum_{e\in E'} \omega(e)\geq c^{\delta}$ for all~$E'\in \cE$.
	Then, there exists a matching~$\cM$ in~$H$ with~$\abs{\cM\cap E'}=(1\pm c^{-\eps})\sum_{e\in E'} \omega(e)$
	for all~$E'\in\cE$.
\end{cor}
\begin{proof}
	We will apply Theorem~\ref{theorem: pseudorandom matching} to an auxiliary hypergraph obtained by replacing every edge in~$H$ with essentially~$\omega(e)M$ copies of itself, where~$M\defeq (c-1)c$ is a sufficiently large convenient multiplicity.
	As we want to avoid multihypergraphs, we simply increase the uniformity by~$1$ and add some dummy vertices.
	More precisely, consider the auxiliary~$(k+1)$-graph~$H_{M}$ with vertex set~$V\cup(E\times [M])$ whose edges are the sets~$\{ v_1,\ldots,v_k,(e,i) \}$ with~$\{v_1,\ldots,v_k\}=e$ and~$i\leq \ceil{M\omega(e)}$.
	Let~$E_M\defeq E(H_M)$.
	Observe that there is a correspondence of matchings in~$H_M$ and matchings in~$H$, namely, for every matching~$\cM_M$ in~$H_M$, those edges of~$H$ that are subsets of edges in~$\cM_M$ form a matching in~$H$.
	
	We will verify that the given properties of~$H$ and~$\omega$ translate to properties of~$H_M$ that allow an application of Theorem~\ref{theorem: pseudorandom matching} that in turn yields the desired matching in~$H$ via the aforementioned correspondence.
	Let~$\Delta_{M}\defeq c^{2}$. Since~$1/c\leq \omega(e)$ holds for all~$e\in E$, we obtain~$d_H(v)\leq c$ and~$d_H(uv)\leq c^{1-\delta}$ for all distinct~$u,v\in V$ as well as~$\abs{E}\leq c\exp(c^{\eps^2})$. This implies
	\begin{equation*}
	d_{H_M}(v)=\sum_{e\in E\colon v\in e}\ceil{M\omega(e)}\leq M+d_H(v)\leq \Delta_{M}
	\end{equation*}
	and
	\begin{equation*}
	d_{H_M}(uv)=\sum_{e\in E\colon u,v\in e}\ceil{M\omega(e)}\leq \frac{M}{c^{\delta}}+d_H(uv)\leq\Delta_M^{1-\delta/2}\leq\Delta_M^{1-\delta/4}
	\end{equation*}
	as well as
	\begin{equation*}
	\abs{E_M}=\sum_{e\in E}\ceil{M\omega(e)}\leq M\exp(c^{\eps^2})+\abs{E}\leq c^{2}\exp(c^{\eps^2})\leq \exp(\Delta_M^{\eps^2}).
	\end{equation*}
	Furthermore, observe that~$\abs{\cE}\leq\exp(c^{\eps^2})\leq \exp(\Delta_M^{\eps^2})$.
	For~$E'\in \cE$, let~$E'_M$ be the set of edges~$\{v_1,\ldots,v_k,(e,i)\}$ of~$H_M$ with~$e\in E'$.
	We have
	\begin{equation*}
	\abs{E'_M}=\sum_{e\in E'} \ceil{M\omega(e)}\geq Mc^{\delta}\geq c^{2}c^{\delta/2}= \Delta_M^{1+\delta/4}.
	\end{equation*}
	An application of Theorem~\ref{theorem: pseudorandom matching}, with~$\Delta_M$,~$\delta/4$,~$H_M$, the set of indicator functions of the sets~$E'_M$ playing the roles of~$\Delta$,~$\delta$,~$H$,~$\cW$, yields a matching~$\cM_M$ in~$H_M$ with
	\begin{equation*}
	\abs{\cM_M\cap E'_M}=(1\pm \Delta_M^{-\eps})\frac{\sum_{e\in E'} \ceil{M\omega(e)}}{\Delta_M}.
	\end{equation*}
	Since we have~$1/c\leq \omega(e)$ for all~$e\in E$ and thus~$\abs{E'}\leq c\sum_{e\in E'}\omega(e)$ for all~$E'\in\cE$, this implies
	\begin{equation*}
		\abs{\cM_M\cap E'_M}\leq (1+\Delta_M^{-\eps})\frac{\abs{E'}+M\sum_{e\in E'} \omega(e)}{\Delta_M}\leq(1+\Delta_M^{-\eps})\frac{c+M}{\Delta_M}\sum_{e\in E'} \omega(e)\leq (1+c^{-\eps})\sum_{e\in E'} \omega(e)
	\end{equation*}
	and
	\begin{equation*}
	\abs{\cM_M\cap E'_M}\geq (1-\Delta_M^{-\eps})\frac{M\sum_{e\in E'} \omega(e)}{\Delta_M}\geq (1-c^{-\eps})\sum_{e\in E'} \omega(e)
	\end{equation*}
	for all~$E'\in\cE$. 
	Thus, the edges in~$H$ which are subsets of edges in~$\cM_M$ form a matching in~$H$ with the desired properties.
\end{proof}

We also need the following result from \cite{glock2019counting} and another lemma concerning perfect fractional matchings in hypergraphs.
For a $k$-graph with edge set $E$ and~$C>0$, we say~$\omega\colon E\rightarrow \bR_{>0}$ is~\defn{$C$-balanced} if~$\frac{\max_{e\in E} \omega(e)}{\min_{e\in E} \omega(e)}\leq C$ and we say~$\omega$ is \defn{balanced} if it is~$2$-balanced.

\begin{lemma}[{\cite[Lemma 4.2]{glock2019counting}}]\label{lemma: fractional matching minimum degree}
	Let~$1/n\ll 1/C\ll\gamma,1/k$. Suppose~$H$ is a~$k$-graph on~$n$ vertices with~$\delta(H)\geq \bigl(\frac{1}{2}+\gamma\bigr)n$. Then there exists a~$C$-balanced perfect fractional matching in~$H$.
\end{lemma}

\begin{lemma}\label{lemma: perfect fractional matching}
	Suppose~$1/n\ll \rho\ll \eta, 1/k$.
	Suppose~$H$ is an~\intersecting{$\eta$}~$\rho$-almost regular~$k$-graph on~$n$ vertices. Then, there exists a balanced perfect fractional matching~$\omega$ in~$H$.
\end{lemma}
\begin{proof}
	The strategy of the proof is as follows.
	We start with a uniform wight distribution~$\omega_0\colon E(H)\rightarrow[0,1]$ and show that~$\omega_0$ can be turned into a perfect fractional matching as desired through a series of minor modifications.
	
	Let~$V\defeq V(H)$ and~$E\defeq E(H)$.
	For~$\omega\colon E\rightarrow [0,1]$ and~$v\in V$, we define~$\omega(v)\defeq\sum_{e\in E\colon v\in e} \omega(e)$.
	Let~$\omega_0\colon E\rightarrow[0,1]$ with~$\omega_0(e)=\frac{n}{k\abs{E}}$ for all~$e\in E$ and~$\xi\colon V\rightarrow\bR$ with~$\xi(v)=\omega_0(v)-1$ for all~$v\in V$.
	We have~$\sum_{v\in V}\omega_0(v)=n$ and thus
	\begin{equation}\label{equation: sum of deviations}
		\sum_{v\in V} \xi(v)=0.
	\end{equation}
	We wish to choose the series of modifications of~$\omega_0$ such that it mimics redistributing the deviations~$\xi(v)$ from the target weight~$1$ uniformly across all vertices. We achieve this by defining every modification as a manipulation of the weights on edges of suitable walks of length~$2$ as follows.
	For a self-avoiding walk~$W=v_1\ldots v_{k+1}$ in~$H$,~$\omega\colon E\rightarrow\bR$, and~$a\in\bR$, we say that~$\omega'\colon E\rightarrow\bR$ is the function \textit{obtained from~$\omega$ by using~$W$ with weight~$a$} if~$\omega'(e)=\omega(e)-a$ for~$e=\{v_1,\ldots,v_k\}$,~$\omega'(e)=\omega(e)+a$ for~$e=\{v_2,\ldots,v_{k+1}\}$, and~$\omega'(e)=\omega(e)$ otherwise.
	Hence~$\omega'(v_1)=\omega(v_1)-a$,~$\omega'(v_{k+1})=\omega(v_{k+1})+a$, and~$\omega'(v)=\omega(v)$ for all~$v\in V\setminus\{v_1,v_{k+1}\}$.
	
	Since~$H$ is~\intersecting{$\eta$}, the number of self-avoiding walks~$v_1\ldots v_{k+1}$ with~$v_1=s$ and~$v_{k+1}=t$ is at least~$\eta n^{k-1}/2$ for all distinct~$s,t\in V$. For distinct~$s,t\in V$, let~$W_{s,t}$ be a set of~$\eta n^{k-1}/2$ self-avoiding walks~$v_1\ldots v_{k+1}$ in~$H$ with~$v_1=s$ and~$v_{k+1}=t$. Let~$v^1_1\ldots v^1_{k+1},\ldots, v^m_1\ldots v^m_{k+1}$ be an ordering of~$\bigcup_{s,t\in V\colon s\neq t} W_{s,t}$ and for~$i\in[m]$, let~$\omega_i$ be the function obtained from~$\omega_{i-1}$ by using~$v^i_1\ldots v^i_{k+1}$ with weight~$\frac{2\xi(v^i_1)}{\eta n^{k}}$. Let~$\omega\defeq \omega_m$. From~\eqref{equation: sum of deviations} we conclude that
	\begin{align*}
		\omega(v)&=\omega_0(v)-(n-1)\cdot\frac{\eta n^{k-1}}{2}\cdot\frac{2\xi(v)}{\eta n^k}+\sum_{u\in V\setminus\{v\}} \frac{\eta n^{k-1}}{2}\cdot\frac{2\xi(u)}{\eta n^k}=1
	\end{align*}
	holds for all~$v\in V$. Thus, it suffices to show~$\omega(e)=(1\pm 1/3)\omega_0(e)$ for all~$e\in E$.
	
	Since for all~$e\in E$, there are at most~$2k!n$ self-avoiding walks~$v_1\ldots v_{k+1}$ in~$H$ with~$e\in \{ \{v_1,\ldots,v_k\},\{v_2,\ldots,v_{k+1}\} \}$, transitioning from~$\omega_0$ to~$\omega$ changed the weight on~$e$ by at most~$2k!n\cdot \frac{2\max_{v\in V}\abs{\xi(v)}}{\eta n^k}$.
	Observe that\COMMENT{%
		we have~$\frac{n}{k\abs{E}}=\frac{n}{\sum_{v\in V}d_H(v)}$ and~$\frac{1-\eps}{1+\eps}\geq 1-\frac{2\eps}{1+\eps}\geq 1-4\eps$ and~$\frac{1+\eps}{1-\eps}\leq 1+\frac{2\eps}{1-\eps}\leq 1+4\eps$ for~$\eps\leq 1/2$
	}
	\begin{equation*}
		\omega_0(v)=d_H(v)\frac{n}{k\abs{E}}=d_H(v)\frac{n}{\sum_{u\in V}d_H(u)}=1\pm 4\rho
	\end{equation*}
	and hence~$\abs{\xi(v)}\leq 4\rho$ holds for all~$v\in V$.
	Thus we obtain
	\begin{equation*}
		\omega(e)= \biggl(1\pm 2k!n\cdot \frac{8\rho}{\eta n^k}\cdot\frac{k\abs{E}}{n}\biggr)\omega_0(e)= \biggl(1\pm \frac{1}{3}\biggr)\omega_0(e)
	\end{equation*}
	for all~$e\in E$.
\end{proof}

\subsection{Almost regular spanning subgraphs with intersecting neighbourhoods}\label{sec: almost regular with intersecting}

In this subsection we prove Lemma~\ref{lemma: almost regular subgraph}.
We use the following statement which follows from~\cite[Theorem 1.2]{FKS:16} by considering the union of perfect matchings in an induced subgraph obtained after removing at most~$k-1$ vertices to make the number of vertices divisible by~$k$ and subsequently adding edges for the previously removed vertices.


\begin{lemma}[{\cite[Theorem 1.2]{FKS:16}}]\label{lemma: non trivial regular subgraph}
	Suppose~$1/n\ll \eps, 1/k$. Suppose~$H$ is a~$k$-graph on~$n$ vertices with~$\delta(H)\geq (\frac{1}{2}+\eps)$. Then~$H$ contains an~$n^{-1/2}$-almost~$r$-regular spanning subgraph for some~$r\geq(1+3\eps/2)\frac{1}{8}\binom{n}{k-1}$.
\end{lemma}

\begin{proof}[Proof of Lemma~\ref{lemma: almost regular subgraph}]
	Suppose~$1/n\ll \eta\ll\eps,1/k$. Suppose~$H$ is a~$k$-graph on~$n$ vertices with~$\delta(H)\geq (\frac{1}{2}+\eps)n$ that contains a~$\rho$-almost~$r$-regular spanning subgraph for some~$\rho\in[0,1/2]$ and~$r\geq 0$.
	If~$r\geq (1+3\eps/2)\frac{1}{8}\binom{n}{k-1}$, let~$\rho_F\defeq \rho$,~$r_F\defeq r$, otherwise let~$\rho_F\defeq n^{-1/2}$ and choose $r_F\geq (1+3\eps/2)\frac{1}{8}\binom{n}{k-1}$ such that there exists a~$\rho_F$-almost~$r_F$-regular spanning subgraph of~$H$, which is possible by Lemma~\ref{lemma: non trivial regular subgraph}.
	Let~$F$ denote a~$\rho_F$-almost~$r_F$-regular spanning subgraph of~$H$.
	In order to obtain a random spanning subgraph~$F'$ of~$H$ that has the desired properties with positive probability, we construct its edge set~$E(F')$ by essentially choosing the edges of~$F$ while ensuring that each edge of~$H$ is included with positive probability.
	
	By Lemma~\ref{lemma: fractional matching minimum degree} there is a~$1/\eta^{1/3}$-balanced perfect fractional matching~$\omega$ in~$H$.
	Let~$\omega_{\max}\defeq \max_{e\in E(H)}\omega(e)$.
	Construct the edge set of the random spanning subgraph~$F'$ of~$H$ as follows. For all~$e\in E(H)$, include~$e$ in~$E(F')$ independently at random with probability~$p_e$, where
	\begin{equation*}
	p_e\defeq \begin{cases}
	(1-\eps)+\frac{\eps\omega(e)}{\omega_{\max}} &\text{if~$e\in E(F)$}\\
	\frac{\eps\omega(e)}{\omega_{\max}}&\text{if~$e\notin E(F)$}.
	\end{cases}
	\end{equation*}
	Fix~$v\in V(H)$ and~$\unord{x},\unord{y}\in\unordsubs{V(H)}{k-1}$.
	Let~$r'\defeq (1-\eps)r_F+\frac{\eps}{\omega_{\max}}\geq n^{k-1}/(9k!)$.
	Clearly,~$\ex[d_{F'}(v)]=(1\pm\rho_F)r'$.
	Since~$p_e\geq \eps\eta^{1/3}$ and since~$H$ is~$2\eps$-intersecting, we obtain~$\ex[\abs{N_{F'}(\unord{x})\cap N_{F'}(\unord{y})}]\geq \eps^2\eta^{2/3}\cdot 2\eps n\geq 2\eta n$.
	Using Chernoff's inequality (Lemma~\ref{lemma: chernoff}) and the union bound shows that~$F'$ is as desired with positive probability.
\end{proof}

\subsection{Different types of degrees}

For a~$k$-graph with vertex set~$V$, the following lemma shows that whenever a set~$U\subseteq V$ meets the neighbourhood~$N_H(\unord{x})$ in roughly~$\theta d_H(\unord{x})$ vertices for all~$\unord{x}\in\unordsubs{V}{k-1}$, then all vertex-degrees decrease by about a factor of $\vartheta^{k-1}$ when transitioning to the subgraph induced by~$U$.

\begin{lemma}\label{lemma: k-1 sets control}
	Suppose~$H$ is a~$k$-graph with vertex set~$V$. Let~$\vartheta\in(0,1)$ and~$\eps\in\bigl[0,\frac{1-\vartheta}{8k^2}\bigr]$.
	Suppose $U\subseteq V$ is a set with~$d_{H[U\cup \unord{x}]}(\unord{x})=(1\pm \eps)\vartheta d_H(\unord{x})$ for all~$\unord{x}\in\unordsubs{V}{k-1}$. 
	Then~$d_{H[U\cup\{v\}]}(v)=(1\pm 8k^3\eps)\vartheta^{k-1}d_H(v)$ for all~$v\in V$.
\end{lemma}
\begin{proof}
	Fix~$v\in V$. 
	Let~$V'\defeq V\setminus\{v\}$ and $U'\defeq U\setminus\{v\}$.
	For~$i\in[k-1]_0$, let $m_i$ denote the number of edges~$e\in E(H)$ with~$v\in e$ and~$\abs{e\cap U'}=i$.
	Then our task is to estimate $m_{k-1}=d_{H[U\cup\{v\}]}(v)$.
	To this end, we inductively relate~$m_{i}$ and~$m_{i-1}$ for all~$i\in[k-1]$.
	
	For~$i\in[k-1]$, we have
	\begin{align*}
		im_i&=\sum_{\unord{x}\in\unordsubs{U'}{i-1},\unord{y}\in\unordsubs{V'\setminus U'}{k-i-1}}d_{H[U\cup\{v\}\cup\unord{x}\cup\unord{y}]}(\{v\}\cup\unord{x}\cup\unord{y})
		=(1\pm\eps)\vartheta\sum_{\unord{x}\in\unordsubs{U'}{i-1},\unord{y}\in\unordsubs{V'\setminus U'}{k-i-1}}d_{H}(\{v\}\cup\unord{x}\cup\unord{y})\\
		&=(1\pm\eps)\vartheta((k-i)m_{i-1}+im_i)
	\end{align*}
	and hence,\COMMENT{%
		Fix~$\vartheta'=(1\pm\eps)\vartheta$ with~$im_i=\vartheta'((k-i)m_{i-1}+im_i)$ and obtain~$\frac{1-\vartheta'}{\vartheta'}\frac{i}{k-i}m_{i}=m_{i-1}$. Then use		
		$\frac{1-\vartheta+\eps\vartheta}{\vartheta-\eps\vartheta}
		=\frac{1-\vartheta}{\vartheta}\frac{1+\frac{\eps\vartheta}{1-\vartheta}}{1-\eps}
		\leq\frac{1-\vartheta}{\vartheta}\frac{1+\frac{\eps}{1-\vartheta}}{1-\frac{\eps}{1-\vartheta}}
		=\frac{1-\vartheta}{\vartheta}\Bigl(1+\frac{2\frac{\eps}{1-\vartheta}}{1-\frac{\eps}{1-\vartheta}}\Bigr)
		\leq\frac{1-\vartheta}{\vartheta}\Bigl(1+4\frac{\eps}{1-\vartheta}\Bigr)$ and
		$\frac{1-\vartheta-\eps\vartheta}{\vartheta+\eps\vartheta}
		=\frac{1-\vartheta}{\vartheta}\frac{1-\frac{\eps\vartheta}{1-\vartheta}}{1+\eps}
		\geq\frac{1-\vartheta}{\vartheta}\frac{1-\frac{\eps}{1-\vartheta}}{1+\frac{\eps}{1-\vartheta}}
		=\frac{1-\vartheta}{\vartheta}\Bigl(1-\frac{2\frac{\eps}{1-\vartheta}}{1+\frac{\eps}{1-\vartheta}}\Bigr)
		\geq \frac{1-\vartheta}{\vartheta}\Bigl(1-2\frac{\eps}{1-\vartheta}\Bigr)$ to conclude
		$\Bigl(1\pm \frac{4\eps}{1-\vartheta} \Bigr)\frac{1-\vartheta}{\vartheta}\frac{k-i}{i}m_{i-1}=m_i$
	}
	\begin{equation*}
		m_{i-1}=\biggl(1\pm \frac{4\eps}{1-\vartheta}\biggr)\frac{1-\vartheta}{\vartheta} \frac{i}{k-i}m_{i}.
	\end{equation*}
	From this, we inductively conclude that\COMMENT{%
		(assume $ax\leq 1$) use~$(1+x)^a\leq 1+\sum_{a'\in[a]} (ax)^{a'}\leq 1+\sum_{a'\in[a]} ax= 1+a^2x$ and~$(1-x)^a\geq 1-\sum_{a'\in[a]} (ax)^{a'}\geq 1-\sum_{a'\in[a]} ax= 1-a^2x$ for~$x=\frac{4\eps}{1-\vartheta}$ and~$a=k-1$.
	}
	\begin{equation*}
		m_i=\biggl(1\pm \frac{4k^2\eps}{1-\vartheta}\biggr)\frac{(1-\vartheta)^{k-1-i}}{\vartheta^{k-1-i}} \binom{k-1}{i}m_{k-1}.
	\end{equation*}
	Recall that $\sum_{i=0}^{k-1} \binom{k-1}{i}(1-\vartheta)^{k-1-i}\vartheta^i=((1-\vartheta)+\vartheta)^{k-1}=1$. 
	Therefore,
	\begin{align*}
	d_H(v)-m_{k-1}&=\sum_{i=0}^{k-2}m_i=\biggl(1\pm \frac{4k^2\eps}{1-\vartheta}\biggr)\frac{m_{k-1}}{\vartheta^{k-1}}\sum_{i=0}^{k-2}\binom{k-1}{i}(1-\vartheta)^{k-1-i}\vartheta^{i}\\&=\biggl(1\pm \frac{4k^2\eps}{1-\vartheta}\biggr)\frac{1-\vartheta^{k-1}}{\vartheta^{k-1}}m_{k-1}
	\end{align*}
	and thus,\COMMENT{%
		First obtain~$d_H(v)=(1\pm \frac{4k^2(1-\vartheta^{k-1})\eps}{1-\vartheta})\frac{1}{\vartheta^{k-1}}m_{k-1}$.
		Then use: If~$x=(1\pm \delta)y$ and~$\delta\leq 1/2$, then~$\frac{1}{1-\delta}x\geq y$ and~$\frac{1}{1+\delta}x\leq y$ and thus~$(1-2\delta)x\leq \Bigl(1-\frac{\delta}{1+\delta}\Bigr)x=\frac{1}{1+\delta}x\leq y\leq \frac{1}{1-\delta}x= \Bigl(1+\frac{\delta}{1-\delta}\Bigr)x\leq (1+2\delta)x$.
		Finally use~$1-\vartheta^{k-1}=1-(1-(1-\vartheta))^{k-1}\leq 1-(1-k(1-\vartheta))=k(1-\vartheta)$
	}
	\begin{equation*}
	m_{k-1}=\biggl(1\pm \frac{8k^2(1-\vartheta^{k-1})\eps}{1-\vartheta}\biggr)\vartheta^{k-1}d_H(v)=(1\pm 8k^3\eps)\vartheta^{k-1}d_H(v)
	\end{equation*}
	which completes the proof.
\end{proof}

\subsection{Many paths in intersecting~$k$-graphs}

We say that a walk~$W=w_{1}\ldots w_{\ell}$ in a~$k$-graph is \defn{internally self-avoiding} if the walks~$w_{1}\ldots w_{\ell-k}$ and~$w_{k+1}\ldots w_{\ell}$ are self-avoiding.
We use the following result from~\cite{JK:21}.
\begin{lemma}[{\cite[Lemma 2.3]{JK:21}}]\label{lemma: connecting}
	Suppose~$1/n\ll \alpha \ll 1/\ell\ll\eta,1/k$.
	Suppose~$H$ is an~\intersecting{$\eta$}~$k$-graph on~$n$ vertices with vertex set~$V$.
	Then, for all~$\orderededge{s},\orderededge{t}\in\orderededge{E}(H)$, the number of internally self-avoiding~$\ell$-walks from~$\orderededge{s}$ to~$\orderededge{t}$ in~$H$ is at least~$\alpha n^{\ell-2k}$.
\end{lemma}

\section{Approximate decomposition into path coverings}\label{sec:approx}
In this section we use Corollary~\ref{cor: simple weighted pseudorandom matching} to turn fractional decompositions provided by Theorem~\ref{theorem: fractional cycle decomposition} into approximate decompositions of almost vertex-regular~$k$-graphs into almost spanning collections of paths.
These collections of paths form the basis for the construction of cycle factors in Section~\ref{sec:approxtoreal}.

We need to keep track of how~$k$-sets that may be edges in a larger graph are distributed with respect to the paths constructed below.
Suppose~$H$ is a~$k$-graph on~$n$ vertices with vertex set~$V$ and~$\cP$ is a collection of paths in~$H$.
Let~$\unord{e}\in \binom{V}{k}$.
We use the following terminology to classify~$\unord{e}$ with respect to~$\cP$, where we often tacitly assume that the collection of paths with respect to which we classify a~$k$-set is obvious from the context.
We say that~$\unord{e}$ is
\begin{enumerate}[label=\textup{(\roman*)}]
	\item \defn{$j$-ending} (or of type \defn{$\typeend[$j$]$}) if~$j$ is the maximal integer for which there is a~$j$-set~$\unord{x}\subseteq \unord{e}$ and a path~$P\in\cP$ such that~$\unord{x}$ is an end-set of~$P$;
	\item \defn{ending} if it is~$j$-ending for some~$j\in [k]$;
	\item \defn{leftover} (or of type \defn{$\typelo$}) if it is neither ending nor a subset of~$V(\cP)$;
	\item \defn{$j$-concentrated} (or of type \defn{$\typecon[$j$]$}) if it is neither ending nor leftover and~$j$ is the maximal integer for which there is a~$j$-set~$\unord{x}\subseteq \unord{e}$ and a path~$P\in\cP$ such that~$\unord{x}\subseteq V(P)$;
	\item \defn{concentrated} if it is~$j$-concentrated for some~$j\in[k]$.
\end{enumerate}

We denote the set of types by~$\cT\defeq \cT_{\typeend}\cup\{\typelo\}\cup\cT_{\typecon}$ where~$\cT_{\typeend}\defeq \{\typeend[$1$],\dots ,\typeend[$k$]\}$ and~$\cT_{\typecon}\defeq \{\typecon[$1$],\ldots,\typecon[$k$]\}$.
Note that given~$\cP$ and~$\unord{e}$, the~$k$-set~$\unord{e}$ has a unique type with respect to~$\cP$.
Given multiple collections~$\cP_1,\ldots,\cP_r$ of paths, for~$i\in[r]$, we use~$\tau(\unord{e},i)$ to denote the type of~$\unord{e}$ with respect to~$\cP_i$, and for~$\tau\in\cT$, we set~$\cI_{\tau}(\unord{e})\defeq \{i\in[r]\colon \tau(\unord{e},i)=\tau\}$, where we tacitly assume that the index set and the collection of paths that belongs to a given index are obvious from the context.
\begin{prop}\label{prop: simultaneous path cover}
	Suppose~$1/n\ll \rho,1/L\ll \eta,\mu,1/k$.
	Let~$H$ be an~\intersecting{$\eta$}~$k$-graph on~$n$ vertices with vertex set~$V$ such that there is an integer~$r$ with~$kr\leq d_H(v)\leq (1+\rho)kr$ for all~$v\in V$.
	Then, there exist edge-disjoint collections~$\cP_1,\ldots,\cP_{r}$ of~$L$-paths in~$H$ with~$\abs{V(\cP_i)}\geq (1-\mu)n$ for all~$i\in[r]$ such that the following holds for all~$\unord{e}\in\binom{V}{k}$.
	\begin{enumerate}[label=\textup{(\roman*)}]
		\item\label{lemma: simultaneous path cover 1} $\abs{\cI_{\typelo}(\unord{e})}\leq\mu r$;
		\item\label{lemma: simultaneous path cover 2} $\abs{\cI_{\typecon[$j$]}(\unord{e})}\leq n^{k-j}/\eta^{2L}$ for all~$j\in[k-1]$ and~$\abs{\cI_{\typecon[$k$]}(\unord{e})}\leq n/\eta^{2L}$;
		\item\label{lemma: simultaneous path cover 3} $\abs{\cI_{\typeend[$j$]}(\unord{e})}\leq n^{k-j}/L^{1/2}$ for all~$j\in[k-1]$.
	\end{enumerate}
\end{prop}
Note that the proof also yields edge-disjoint collections~$\cC_1,\ldots,\cC_{r}$ of~$L$-cycles instead of the paths with~$\abs{V(\cC_i)}\geq (1-\mu)n$ for all~$i\in[r]$ (without the properties~\ref{lemma: simultaneous path cover 1}--\ref{lemma: simultaneous path cover 3});
we use this in the proof of Lemma~\ref{lem:layer}.

\begin{proof}[Proof of Proposition~\ref{prop: simultaneous path cover}]
	First we argue that it suffices to find collections~$\cC_1,\ldots,\cC_{r}$ of~$L$-cycles in~$H$ with properties similar to those of collections of paths in the statement; then we obtain such collections of cycles by applying Corollary~\ref{cor: simple weighted pseudorandom matching} in an auxiliary hypergraph that represents a fractional cycle decomposition given by Theorem~\ref{theorem: fractional cycle decomposition}.
	
	Suppose~$\cC_1,\ldots,\cC_{r}$ are edge-disjoint collections of~$L$-cycles in~$H$.
	Note that for all~$j\in[k-1]$ and~$\unord{x}\subseteq \unordsubs{V}{j}$, there are at most~$n^{k-j}$ collections~$\cC_i$ with~$i\in[r]$ where the elements of~$\unord{x}$ appear consecutively in a cycle in~$\cC_i$ and for all~$i\in[r]$, there is at most one cycle~$C\in\cC_i$ such that the elements of~$\unord{x}$ appear consecutively in~$C$.
	Now, for all~$i\in[r]$ and every cycle in~$C\in\cC_i$, delete~$k-1$ consecutive edges uniformly at random and independently of the edges deleted in the other cycles to obtain collections of paths~$\cP_1,\dots,\cP_r$.
	Then the expected value of the random variable counting the number of collections~$\cP_i$ with~$i\in[r]$ where~$\unord{x}$ is an end-set of a path in~$\cP_i$ is at most~$2n^{k-j}/L$.
	Hence Chernoff's inequality (Lemma~\ref{lemma: chernoff}) entails that it is possible to delete~$k-1$ consecutive edges of every cycle in~$\cC_i$ for all~$i\in[r]$ to obtain collections of paths~$\cP_1,\dots,\cP_r$ such that for all~$j\in[k-1]$ and~$\unord{x}\in\unordsubs{V}{j}$, there are at most~$3n^{k-j}/L$ collections~$\cP_i$ with~$i\in[r]$ where~$\unord{x}$ is an end-set of a path in~$\cP_i$.
	For such collections~$\cP_1,\dots,\cP_r$,~$j\in[k-1]$ and~$\unord{e}\in \binom{V}{k}$, we have~$\cI_{\typeend[$j$]}(\unord{e})\leq \binom{k}{j}\cdot 3n^{k-j}/L\leq n^{k-j}/L^{1/2}$ and thus it suffices to obtain edge-disjoint collections~$\cC_1,\ldots,\cC_{r}$ of~$L$-cycles in~$H$ with~$\abs{V(\cC_i)}\geq (1-\mu)n$ for all~$i\in[r]$ such that the following holds.
	\begin{itemize}
		\item $\abs{\{ i\in[r]\colon v\notin V(\cC_i) \}}\leq \mu r/k$ for all~$v\in V$;
		\item $\abs{\{ i\in[r]\colon \exists C\in\cC_i\colon \unord{x}\subseteq V(C) \}}\leq n^{k-j}/(\eta^{2L}\binom{k}{j})$ for all~$j\in[k-1]$ and~$\unord{x}\in\unordsubs{V}{j}$.
	\end{itemize}
	
	Let~$E\defeq E(H)$ and let~$\cC_L(H)$ denote the set of~$L$-cycles in~$H$.
	From Theorem~\ref{theorem: fractional cycle decomposition} we obtain a fractional~$L$-cycle decomposition~$\omega$ of~$H$ with
	\begin{equation}\label{equation: weights are close}
		\frac{\abs{E}}{n^L}\leq \omega(C)\leq \frac{3\abs{E}}{\eta^L n^L}
	\end{equation}
	for all~$C\in\cC_L(H)$.
	Consider the~$2L$-graph~$H^{\star}$ with vertex set~$(V\times[r])\cup E$ where we add for each~$C\in\cC_L(H)$ and all~$i\in[r]$, the edge~$e_{C,i}^{\star}\defeq \{(v,i)\colon v\in V(C)\}\cup E(C)$ (note that the edge~$e_{C,i}^{\star}$ uniquely identifies~$C$ and~$i$).
	Let~$E^{\star}\defeq E(H^{\star})$ and~$\Gamma\defeq (1+\rho)r\geq \Delta(H)/k$ and let~$\omega^{\star}\colon E^{\star}\rightarrow [0,1]$ be the edge weight function with~$\omega^{\star}(e_{C,i}^{\star})=\omega(C)/\Gamma$ for all~$e^{\star}\in E^{\star}$. Here,~$\omega^{\star}$ is a representation of~$\omega$ that is normalized such that for all~$v\in V$ and~$i\in [r]$, we have
	\begin{equation}\label{equation: 1-degree 1}
		\sum_{\substack{e^{\star}\in E^{\star}\colon\\(v,i)\in e^{\star}}} \omega^{\star}(e^{\star})=\frac{1}{k}\cdot\sum_{\substack{e\in E(H)\colon\\ v\in e}}\sum_{\substack{C\in \cC_L(H)\colon\\ e\in E(C)}}\frac{\omega(C)}{\Gamma}=\frac{1}{k}\cdot \frac{d_H(v)}{\Gamma}\in \biggl[\frac{1}{1+\rho},1\biggr].
	\end{equation}
	Note that there is a correspondence of matchings in~$H^{\star}$ to edge-disjoint collections of~$L$-cycles in~$H$, namely a matching~$\cM^{\star}$ in~$H^{\star}$ corresponds to the collections~$\cC_1,\ldots,\cC_r$, where~$\cC_i=\{C\in\cC_L(H)\colon e_{C,i}^{\star}\in\cM^{\star}\}$ for all~$i\in[r]$.
	This will allow us to obtain collections of cycles with the desired properties from Corollary~\ref{cor: simple weighted pseudorandom matching}.
	We now introduce appropriate parameters including suitable subsets of~$E^{\star}$ and check that the conditions necessary for a suitable application of Corollary~\ref{cor: simple weighted pseudorandom matching} hold in this setting.
	
	Let~$\delta\defeq \frac{1}{3L}$,~$\eps\defeq \frac{\delta}{900L^2}$, and~$c\defeq n^{L}$.
	From \eqref{equation: weights are close} we obtain (with some room to spare)
	\begin{equation}\label{equation: weight bounds}
		\frac{1}{c}\leq \frac{\eta}{2k!\Gamma n^{L-k}}\leq \frac{\abs{E}}{\Gamma n^L}\leq\frac{\omega(C)}{\Gamma}\leq \frac{3 \abs{E}}{\eta^{L}\Gamma n^L}\leq  \frac{3}{\eta^{L}\Gamma n^{L-k}}\leq \frac{1}{n^{L-3/2}}c^{-\delta}
	\end{equation}
	for all~$C\in\cC_L(H)$. To complete our analysis of the $1$-degrees in~$H^{\star}$, we observe that in addition to~\eqref{equation: 1-degree 1}, for all~$e\in E$, we have
	\begin{equation}\label{equation: 1-degree 2}
		\sum_{\substack{e^{\star}\in E^{\star}\colon\\e\in e^{\star}}} \omega^{\star}(e^{\star})=r\cdot\sum_{\substack{C\in \cC_L(H)\colon\\ e\in E(C)}}\frac{\omega(C)}{\Gamma}=r\cdot\frac{1}{\Gamma}\leq 1.
	\end{equation}
	Let us now consider the~$2$-degrees in~$H^{\star}$.
	For distinct~$u,v\in V$ and distinct~$e,f\in E$, we have
	\begin{gather*}
		\abs{\{ C\in\cC_L(H)\colon u,v\in V(C) \}}\leq L n^{L-2},\quad \abs{\{ C\in\cC_L(H)\colon v\in V(C)\wedge e\in E(C) \}}\leq k! n^{L-k}\\
		\text{and}\quad\abs{\{ C\in\cC_L(H)\colon e,f\in E(C) \}}\leq 2k! n^{L-k-1}.
	\end{gather*}
	Using \eqref{equation: weight bounds}, this yields for all~$i\in[r]$, that
	\begin{align}
		\sum_{\substack{e^{\star}\in E^{\star}:\\(u,i),(v,i)\in e^{\star}}} \omega^{\star}(e^{\star})&=\sum_{\substack{C\in \cC_L(H):\\u,v\in V(C)}} \frac{\omega(C)}{\Gamma}\leq \frac{L}{n^{1/2}}c^{-\delta}\leq c^{-\delta},\label{equation: 2-degree 1}\\
		\sum_{\substack{e^{\star}\in E^{\star}:\\(v,i),e\in e^{\star}}} \omega^{\star}(e^{\star})&=\sum_{\substack{C\in \cC_L(H):\\v\in V(C)\wedge e\in E(C)}} \frac{\omega(C)}{\Gamma}\leq \frac{k!}{n^{k-3/2}}c^{-\delta}\leq c^{-\delta},\label{equation: 2-degree 2}\quad\text{and}\quad\\
		\sum_{\substack{e^{\star}\in E^{\star}:e,f\in e^{\star}}} \omega^{\star}(e^{\star})&=r\cdot\sum_{\substack{C\in \cC_L(H):\\e,f\in E(C)}} \frac{\omega(C)}{\Gamma}\leq \frac{2k!r}{n^{k-1/2}}c^{-\delta}\leq c^{-\delta}.\label{equation: 2-degree 3}
	\end{align}
	Furthermore, observe that
	\begin{equation}\label{equation: total weight}
		\sum_{e^{\star}\in E(H^{\star})}\omega^{\star}(e^{\star})=r\cdot\sum_{C\in\cC_L(H)}\frac{\omega(C)}{\Gamma}=r\cdot\frac{\abs{E}}{\Gamma L}\leq \exp(c^{\eps^2})
	\end{equation}
	gives an upper bound for the total weight on the edges of~$H^{\star}$.
	For~$i\in[r]$,~$v\in V$,~$j\in[k-1]$, and~$\unord{x}\in\unordsubs{V}{j}$, we define edge-sets as follows.
	\begin{gather*}
		E_i^{\star}\defeq \{ e_{C,i'}^{\star}\in E^{\star}: i'=i \},\quad E_v^{\star}\defeq \{ e_{C,i'}^{\star}\in E^{\star}: v\in V(C) \}
		\quad\text{and}\quad E_{\unord{x}}^{\star}\defeq \{ e_{C,i'}^{\star}\in E^{\star}: \unord{x}\subseteq  V(C)\}.
	\end{gather*}
	We have
	\begin{equation}\label{equation: test set size 1}
		\sum_{e^{\star}\in E_i^{\star}}\omega^{\star}(e^{\star})=\sum_{C\in\cC_L(H)} \frac{\omega(C)}{\Gamma}=\frac{\abs{E}}{\Gamma L}\geq\frac{r n}{\Gamma L}\geq c^{\delta}
	\end{equation}
	and
	\begin{equation}\label{equation: test set size 2}
		\sum_{e^{\star}\in E_v^{\star}}\omega^{\star}(e^{\star})=r\cdot\frac{1}{k}\cdot\sum_{\substack{e\in E(H)\colon\\ v\in e}}\sum_{\substack{C\in \cC_L(H)\colon\\ e\in E(C)}}\frac{\omega(C)}{\Gamma}=r\cdot\frac{1}{k}\cdot \frac{d_H(v)}{\Gamma}\geq \frac{r}{1+\rho}\geq c^{\delta}.
	\end{equation}
	Since Lemma~\ref{lemma: connecting} implies~$\abs{\{ C\in\cC_L(H)\colon d_C(\unord{x})\geq 1 \}}\geq n^{L-j-1/3}$, using \eqref{equation: weight bounds} we obtain
	\begin{equation}\label{equation: test set size 3}
		\sum_{e^{\star}\in E_{\unord{x}}^{\star}}\omega^{\star}(e^{\star})\geq r\cdot\sum_{\substack{C\in\cC_L(H)\colon\\ d_C(\unord{x})\geq 1}}\frac{\omega(C)}{\Gamma}\geq r\cdot n^{L-j-1/3}\cdot\frac{\eta}{2k!\Gamma n^{L-k}}\geq n^{k-j-1/2}\geq c^{\delta}.
	\end{equation}
	Thus, since~$r+n+\sum_{j\in [k-1]}\binom{n}{j}\leq\exp(c^{\eps^2})$ holds and by~\eqref{equation: 1-degree 1}--\eqref{equation: test set size 3} we may apply Corollary~\ref{cor: simple weighted pseudorandom matching} to obtain a matching~$\cM^{\star}$ in~$H^{\star}$ with
	\begin{equation*}
		\abs{\cM^{\star}\cap E_i^{\star}}\geq (1-c^{-\eps})\frac{\abs{E}}{\Gamma L}\geq (1-c^{-\eps})\frac{rn}{\Gamma L}\geq (1-\mu)\frac{n}{L},
	\end{equation*}
	\begin{equation*}
		\abs{\cM^{\star}\cap E_v^{\star}}\geq (1-c^{-\eps})\frac{r}{k}\cdot\frac{d_H(v)}{\Gamma}\geq (1-c^{-\eps})\frac{r^2}{\Gamma}\geq \biggl(1-\frac{\mu}{k}\biggr) r,
	\end{equation*}
	and
	\begin{equation*}
		\abs{\cM^{\star}\cap E_{\unord{x}}^{\star}}\leq (1+c^{-\eps})\sum_{e^{\star}\in E_{\unord{x}}^{\star}}\omega^{\star}(e^{\star})
	\end{equation*}
	for all~$i\in[r]$,~$v\in V$,~$j\in[k-1]$, and~$\unord{x}\in\unordsubs{V}{j}$. Since we have~$\abs{\{ C\in\cC_L(H)\colon \unord{x}\subseteq V(C)\}}\leq L^{j}n^{L-j}$, \eqref{equation: weight bounds} implies
	\begin{equation*}
		\abs{\cM^{\star}\cap E_{\unord{x}}^{\star}}\leq (1+c^{-\eps})r\cdot \sum_{\substack{C\in\cC_L(H)\colon\\ \unord{x}\subseteq V(C)}}\frac{\omega(C)}{\Gamma}\leq (1+c^{-\eps})r\cdot L^{j}n^{L-j}\cdot\frac{3}{\eta^{L} \Gamma n^{L-k}}\leq \frac{1}{\eta^{2L}\binom{k}{j}}n^{k-j}
	\end{equation*}
	and thus, choosing~$\cC_i$ as the collection of~$L$-cycles~$C$ in~$H$ with~$e_{C,i}^{\star}\in \cM^{\star}$ yields collections of cycles in~$H$ with the desired properties.
\end{proof}

\section{Ingredients for absorption}\label{sec: abs}

Suppose we are given a~$k$-graph~$H$ on~$n$ vertices.
In this section we construct a collection of paths~$\cP$ in~$H$ whose lengths do not grow with~$n$
such that any small set of vertices~$X$ can be absorbed into these paths; that is, for every path~$P\in\cP$, there is a new path~$P'$ with the same end-tuples as~$P$ and~$V(P')\subseteq V(P)\cup X$ such that the new paths form a collection of paths~$\cP'$ with~$V(\cP')= V(P)\cup X$.

There are two main novelties in our setting.
Firstly, we choose~$\cP$ randomly in an extremely uniform way such that $V(\cP)$ behaves like a uniformly chosen vertex set of size $|V(\cP)|$.
Secondly, we can control how many vertices each path in~$\cP$ will absorb if an adversary determines a set of vertices to absorb. 
Here the main difficulty is that the lengths of the paths in~$\cP$ do not grow with~$n$.
To the best of our knowledge, this problem has not been dealt with in the literature so far.

\subsection{Random walks and vertex absorbers}

Let~$H$ be a~$k$-graph on~$n$ vertices with vertex set~$V$ and edge set~$E$.
For~$x\in V$, a path~$A=a_1\ldots a_{2k}$ in~$H$ is called an \defn{$x$-absorber}\footnote{The reader familiar with the topic knows that usually absorbers consist of~$2k-2$ vertices; in fact, we only deviate from this to enhance readability and simplify some arguments.} if~$a_1\ldots a_{k}xa_{k+1}\ldots a_{2k}$ is also a path in~$H$.
In what follows, we describe how certain random walks contain many vertex-disjoint~$x$-absorbers for all~$x\in V$ simultaneously
and at the same time ensure that the set of vertices visited by these random walks behaves like a vertex set chosen uniformly at random among all vertex sets of the same size.

In what follows,~$a$,~$i$,~$\ell$,~$L$,~$t_{\star}$ are always positive integers.
Let~$W=w_1\ldots w_{L}$ be a walk.
For~$i\in[L/(2k+\ell)]$, we define
\begin{equation*}
	A^{\ell}_{i}(W)\defeq w_{(2k+\ell)(i-1)+1}\ldots w_{(2k+\ell)(i-1)+2k}\quad\text{and}\quad\cA^{\ell}(W)\defeq \biggl\{ A^{\ell}_i(W)\colon i\in\biggl[\frac{L}{2k+\ell}\biggr]\biggr\}.
\end{equation*}
We may think of~$A^{\ell}_i$ as the~$i$-th potential absorber in~$W$ when requiring~$\ell$ vertices between absorbers whereas~$\cA^{\ell}(W)$ is the set of all these potential absorbers in~$W$.
To gain control over where absorbed vertices will be placed, we consider absorbers in groups which we call \defn{blocks}.
Towards the end of the section, our construction yields a set of blocks~$\cB$ and during the absorption, every block~$B\in\cB$ will absorb exactly one vertex.
More precisely, we say that a walk~$B=b_1\ldots b_{a(2k+\ell)}$ in~$H$ is an~\defn{$(a,\ell)$-block} in~$H$; that is,~$B$ can be split into~$a$ consecutive walks that consist of a~$2k$-walk followed by an~$\ell$-walk.
For~$i\in[L/(a(2k+\ell))]$, we define
\begin{equation*}
	B^{a,\ell}_{i}(W)\defeq w_{a(2k+\ell)(i-1)+1}\ldots w_{a(2k+\ell)i}\quad\text{and}\quad\cB^{a,\ell}(W)\defeq \biggl\{B^{a,\ell}_{i}(W)\colon i\in\biggl[\frac{L}{a(2k+\ell)}\biggr]\biggr\}
\end{equation*}
where~$B^{a,\ell}_{i}(W)$ can be considered as the~$i$-th~$(a,\ell)$-block in~$W$ whereas~$\cB^{a,\ell}(W)$ is the set of all these~$(a,\ell)$-blocks in~$W$.
Later, we will choose absorbers and blocks randomly.
In contrast to other approaches existing in the literature, 
we build our absorbing structure via random walks whose distributions are given by perfect fractional matchings to ensure a very uniform distribution of the vertices visited by these random walks.

We introduce the convention that sequences~$s_m\ldots s_n$ with~$m>n$ are considered as the empty sequence which we identify with~$\emptyset$.
For~$\omega\colon E\rightarrow\bR_{\geq 0}$,~$j\in[k]$, and~$v_{1},\ldots,v_j\in V$, 
we define~$\omega(v_1\ldots v_{j})\defeq \sum_{e\in E\colon v_1,\ldots, v_j\in e}\omega(e)$ and we set~$\omega(\emptyset)\defeq \sum_{e\in E}\omega(e)$.

Let~$\omega\colon E\rightarrow\bR_{>0}$.
We say a sequence of~$V$-valued random variables~$X_1\ldots X_{t_{\star}}$ is a random walk in~$H$ \defn{with parameters~$(L,\omega)$}, or simply an~\defn{$(L,\omega)$-random walk} in~$H$, if its distribution is given by
\begin{equation}\label{equation: distribution of random walk}
\begin{aligned}
	\pr[X_t=v_t\mid X_1\ldots X_{t-1}=v_1\ldots v_{t-1}]
	&=\pr[X_t=v_t\mid X_{t-m}\ldots X_{t-1}=v_{t-m}\ldots v_{t-1}]\\
	&=\begin{cases}
	\frac{\omega(v_{t-m}\ldots v_t)}{(k-m)\omega(v_{t-m}\ldots v_{t-1})}
	&\text{if~$v_t\notin \{v_{t-m},\ldots,v_{t-1}\}$}\\
	0&\text{if~$v_t\in \{v_{t-m},\ldots,v_{t-1}\}$}
	\end{cases}
\end{aligned}
\end{equation}
for all~$t\in[t_{\star}]$,~$m\defeq \min\{k-1,t-1 \mod L\}$, and~$v_1,\ldots,v_{t}\in V$ with~$\pr[X_1\ldots X_{t-1}=v_1\ldots v_{t-1}]>0$. This is indeed a probability distribution because for all~$j\in[k-1]$ and~$v_1,\ldots,v_j\in V$, we have
\begin{equation*}
	\sum_{v\in V\setminus\{v_1,\ldots,v_j\}} \omega(v_{1}\ldots v_j v)=(k-j)\omega(v_{1}\ldots v_{j}).
\end{equation*}
Whenever we consider an~$(L,\omega)$-random walk in a~$k$-graph~$H$, we assume that~$L$ is a positive integer and~$\omega\colon E(H)\rightarrow\bR_{>0}$.
Observe that if~$X_1\ldots X_{t_{\star}}$ is an~$(L,\omega)$-random walk in a~$k$-graph~$H$, then also~$X_{L(s-1)+1}\ldots X_{Ls}$ is an~$(L,\omega)$-random walk in~$H$ for all~$s\in[t_{\star}/L]$.
\begin{observation}\label{observation: independent}
	Suppose~$H$ is a~$k$-graph.
	Suppose~$X_1\ldots X_{t_{\star}}$ is an~$(L,\omega)$-random walk in~$H$.
	Then, the random walks~$X_1\ldots X_{Ls}$ and~$X_{Ls+1}\ldots X_{t_{\star}}$ are independent for all positive integers $s$.
\end{observation}

Recall that~$\omega\colon E\rightarrow \bR_{>0}$ is \defn{balanced} if~$\frac{\max_{e\in E} \omega(e)}{\min_{e\in E} \omega(e)}\leq 2$.
For an~$(L,\omega)$-random walk~$X_1\ldots X_{t_{\star}}$ in~$H$ for some balanced~$\omega$, the balancedness of~$\omega$ allows us to bound the probability of the event that~$X_t=X_{t'}$ for some distinct~$t,t'\in[t_{\star}]$; the union bound yields the following observation.
\begin{observation}\label{observation: pairwise disjoint}
	Let~$1/n\ll \eta, 1/k$. 
	Suppose~$H$ is a~$k$-graph on~$n$ vertices with~$\delta(H)\geq \eta n$.
	Suppose~$X_1\ldots X_{t_{\star}}$ with~$t_{\star}\leq 2n^{1/3}$ is an~$(L,\omega)$-random walk in~$H$ for some balanced~$\omega$.
	Then, we have
	\begin{equation*}
		\pr[\text{$X_1\ldots X_{t_{\star}}$ is self-avoiding}]\geq 1-n^{-1/4}.
	\end{equation*}
\end{observation}

The following lemma shows that 
for an~$(L,\omega)$-random walk~$X_1\ldots X_{t_{\star}}$ in~$H$ not only the transition probabilities are determined by~$\omega$, 
but also the probability of~$X_1\ldots X_{t_{\star}}$ consecutively visiting a sequence of vertices can be easily computed in terms of~$\omega$.

\begin{lemma}\label{lemma: probability for tuple}
	Suppose~$H$ is a~$k$-graph on~$n$ vertices with vertex set~$V$. 
	Suppose~$X_1\ldots X_{L}$ is an~$(L,\omega)$-random walk in~$H$.
	Let~$t\in[L]$ and~$j\in[\min\{k,t\}]$. Then, we have
	\begin{equation}\label{equation: probability for tuple}
		\pr[X_{t-j+1}\ldots X_t=v_{-j+1}\ldots v_0]=\frac{(k-j)!\omega(v_{-j+1}\ldots v_{0})}{k!\omega(\emptyset)}
	\end{equation}
	for all~$v_{-j+1},\ldots,v_{0}\in V$.
\end{lemma}
\begin{proof}
	We prove the statement by induction on~$t$.
	If~$t=1$, then the statement is true by choice of~$X_1$.
	Next assume that~\eqref{equation: probability for tuple} is true for a~$t\in [L-1]$ and all~$j\in[\min\{k,t\}]$.
	
	Given such a~$t\in[L-1]$, let~$j\in[\min\{k,t+1\}]$ and~$m\defeq \min\{k-1,t\}$ (hence~$j\leq m+1$) as well as~$v_{-j+1},\ldots,v_{0}\in V$ be given. Furthermore, let~$U\defeq V\setminus\{v_{-j+1},\ldots,v_{0}\}$ and~$h\defeq m-j+1\geq 0$.
	Now we establish~\eqref{equation: probability for tuple} with~$t+1$ instead of~$t$. We compute
	\begin{align}
		\pr[X_{t-j+2}\ldots X_{t+1}=v_{-j+1}\ldots v_0]
		&=\sum_{(v_{-m},\ldots, v_{-j})\in \ordsubs{U}{h}}\pr[X_{t-m+1}\ldots X_{t+1}=v_{-m}\ldots v_0]\nonumber\\
		&=\sum_{(v_{-m},\ldots, v_{-j})\in \ordsubs{U}{h}}\frac{\omega(v_{-m}\ldots v_0)\pr[X_{t-m+1}\ldots X_t=v_{-m}\ldots v_{-1}]}{(k-m)\omega(v_{-m}\ldots v_{-1})}\label{equation: use of distribution}\\
		&=\sum_{(v_{-m},\ldots, v_{-j})\in \ordsubs{U}{h}} \frac{\omega(v_{-m}\ldots v_0)\cdot (k-m)!\omega(v_{-m}\ldots v_{-1})}{(k-m)\omega(v_{-m}\ldots v_{-1})\cdot k!\omega(\emptyset)}\label{equation: use of induction hypothesis}\\
		&=\frac{(k-m-1)!}{k!\omega(\emptyset)}\sum_{(v_{-m},\ldots, v_{-j})\in \ordsubs{U}{h}} \omega(v_{-m}\ldots v_0)\nonumber\\
		&=\frac{(k-m-1)!}{k!\omega(\emptyset)}\sum_{\substack{e\in E\colon\\ v_{-j+1},\ldots,v_{0}\in e}} \binom{k-j}{h}\cdot h!\cdot\omega(e)\nonumber\\
		&=\frac{(k-j)!\omega(v_{-j+1}\ldots v_{0})}{k!\omega(\emptyset)},\nonumber
	\end{align}
	where we used~\eqref{equation: distribution of random walk} for~\eqref{equation: use of distribution} and the induction hypothesis for~\eqref{equation: use of induction hypothesis}.
\end{proof}

The next lemma shows that whenever~$X_1\ldots X_{t_{\star}}$ is an~$(L,\omega)$-random walk in~$H$ for some balanced~$\omega$ and~$x$ a vertex in a slightly larger~$k$-graph, then after a few steps~$X_1\ldots X_{t_{\star}}$ has a decent chance of producing an~$x$-absorber in~$2k$ consecutive steps.
This follows easily, because as an~$\eta$-intersecting~$k$-graph,~$H$ contains sufficiently many suitable~$x$-absorbers, Lemma~\ref{lemma: connecting} guarantees that there are sufficiently many ways for the random walk to arrive at such an~$x$-absorber independent of the starting conditions and the balancedness of~$\omega$ entails that every walk extending an already chosen initial segment of the random walk occurs with sufficiently large probability.

\begin{observation}\label{observation: absorber occurs}
	Let~$1/n\ll \alpha\ll \nu_+,\ell \ll \eta, 1/k$ and~$L\geq 2k$. 
	Suppose~$H_+$ is an~$\eta$-intersecting~$k$-graph on at most~$(1+\nu_+)n$ vertices with vertex set~$V_+$ and~$H$ is an induced subgraph of~$H_+$ on~$n$ vertices. 
	Suppose~$X_1\ldots X_L$ is an~$(L,\omega)$-random walk in~$H$ for some balanced~$\omega$. Then, for all~$t\in\{-\ell,\ldots,L-\ell-2k\}$ and~$x\in V_+$, we have
	\begin{equation*}
		\pr[\text{$X_{t+\ell+1}\ldots X_{t+\ell+2k}$ is an~$x$-absorber in~$H_+$}\mid X_1\ldots X_{t}]\geq \alpha.
	\end{equation*}
\end{observation}

Measuring the impact of removing the vertices of an~$(L,\omega)$-random walk in~$H$ from~$H$ is one of the core objectives in this subsection.
The following lemma shows that for each~$(k-1)$-set of vertices of~$H$, its neighbourhood is essentially visited as often as expected.
Via Lemma~\ref{lemma: k-1 sets control}, this transfers to the vertex degrees appropriately.

\begin{lemma}\label{lemma: same fraction}
	Let~$1/n\ll \eta, 1/k, 1/L$.
	Suppose~$H$ is a~$k$-graph on~$n$ vertices with vertex set~$V$ and~$\delta(H)\geq \eta n$.
	Suppose~$X_1\ldots X_{t_{\star}}$ with~$n^{1/3}/2\leq t_{\star}\leq 2n^{1/3}$ is an~$(L,\omega)$-random walk in~$H$ for some balanced perfect fractional matching~$\omega$ in~$H$. Let~$U\defeq V\setminus \{ X_t\colon t\in[t_{\star}] \}$.
	Then, for all~$\unord{x}\in\unordsubs{V}{k-1}$, we have
	\begin{equation}\label{equation: remaining degree}
		\pr\biggl[d_{H[U\cup\unord{x}]}(\unord{x})=(1\pm n^{-31/40})\biggl(1-\frac{t_{\star}}{n}\biggr)d_H(\unord{x})\biggr]\geq 1-\exp(-n^{1/14}).
	\end{equation}
\end{lemma}
\begin{proof}
	Since Lemma~\ref{lemma: probability for tuple} yields~$\pr[X_t=v]=1/n$ for all~$t\in[t_{\star}]$ and~$v\in V$ and since for all~$\ell\in[L]$, the random variables~$X_{L(s-1)+\ell}$ with~$s\in[t_{\star}/L]$ are mutually independent by Observation~\ref{observation: independent}, the statement is a consequence of Chernoff's inequality (Lemma~\ref{lemma: chernoff}).
	
	Let us turn to the details.
	Fix~$\unord{x}\in\unordsubs{V}{k-1}$.
	To see that~\eqref{equation: remaining degree} holds, we will show that\COMMENT{%
		If~$\abs{N_H(\unord{x})\cap \{ X_t\colon t\in[t_{\star}] \}}=(1\pm n^{-1/8})\frac{t_{\star}}{n}\abs{N_H(\unord{x})}$, then~$d_{H[U\cup\unord{x}]}(\unord{x})=\biggl(1\pm n^{-1/8}\frac{\frac{t_{\star}}{n}}{1-\frac{t_{\star}}{n}}\biggr)\biggl(1-\frac{t_{\star}}{n}\biggr)d_H(\unord{x})$. Observe that~$n^{-1/8}\frac{\frac{t_{\star}}{n}}{1-\frac{t_{\star}}{n}}\leq 3n^{-\frac{2}{3}-\frac{1}{8}}=3n^{-\frac{19}{24}}=3n^{-\frac{3}{4}-\frac{1}{24}}\leq n^{-31/40}$.
	}
	\begin{equation}\label{equation: quatifiaction of removal}
		\pr\biggl[\abs{N_H(\unord{x})\cap \{ X_t\colon t\in[t_{\star}] \}}=(1\pm n^{-1/8})\frac{t_{\star}}{n}\abs{N_H(\unord{x})}\biggr]\geq 1-\exp(-n^{1/14}).
	\end{equation}
	To this end, for~$t\in[L]$, let~$Y_t$ denote the indicator random variable of the event that~$X_t\in N_H(\unord{x})$.
	Note that Lemma~\ref{lemma: probability for tuple} implies~$\pr[Y_t=1]=\abs{N_H(\unord{x})}/n$.
	For~$\ell\in[L]$, let~$N_{\ell}\defeq \sum_{s\in[t_{\star}/L]} Y_{L(s-1)+\ell}$.
	Crucially, note that~$N_{\ell}$ is the sum of the independent random variables~$Y_{L(s-1)+\ell}$ with~$s\in[t_{\star}/L]$.
	Furthermore, for~$t\in[t_{\star}]$, let~$Z_t$ denote the indicator random variable of the event that~$X_t\in \{X_{t'}\colon t'\in[t-1]\}\cap N_H(\unord{x})$ and let~$Z\defeq \sum_{t\in[t_{\star}]} Z_t$.
	Observe that
	\begin{equation*}
		\abs{N_H(\unord{x})\cap \{ X_t\colon t\in[t_{\star}] \}}=\sum_{\ell\in[L]}N_{\ell}-Z.
	\end{equation*}
	Let us estimate~$N_{\ell}$ and~$Z$.
	For all~$\ell\in[L]$, Chernoff's inequality (Lemma~\ref{lemma: chernoff}) entails\COMMENT{%
		$\pr\biggl[N_{\ell}=\biggl(1\pm \frac{n^{-1/8}}{2}\biggr)\frac{\abs{N_H(\unord{x})}}{n}\frac{t_{\max\star}}{L}\biggr]\geq 1-2\exp(-\frac{n^{-1/4}}{12}\frac{\abs{N_H(\unord{x})}}{n}\frac{t_{\star}}{L})\geq 1-2\exp(-\frac{n^{-1/4}}{12}\eta\frac{n^{1/3}}{2L})$
	}
	\begin{equation*}
		\pr\biggl[N_{\ell}=\biggl(1\pm \frac{n^{-1/8}}{2}\biggr)\frac{\abs{N_H(\unord{x})}}{n}\frac{t_{\star}}{L}\biggr]\geq 1-\exp(-n^{1/13}).
	\end{equation*}
	Furthermore, from~$t_{\star}\leq 2n^{1/3}$,~$\delta(H)\geq \eta n$,~\eqref{equation: distribution of random walk}, and the balancedness of~$\omega$, we obtain\COMMENT{%
		Let~$m\defeq \min\{k-1,t-1\mod L\}$. Use~$\pr[Z_t=1\mid Z_1,\ldots, Z_{t-1}]\leq \frac{\frac{n^{k-m-1}}{(k-m-1)!}\max_{e\in E} \omega(e)}{\frac{\eta}{2} \frac{n^{k-m}}{(k-m)!}\min_{e\in E}\omega(e)}\cdot 2n^{1/3}\leq \frac{2kn^{k-m-1}\max_{e\in E} \omega(e)}{\eta n^{k-m}\min_{e\in E}\omega(e)}\cdot 2n^{1/3}\leq \frac{8k}{\eta}n^{-2/3}$
	}
	\begin{equation*}
		\pr[Z_t=1\mid Z_1,\ldots, Z_{t-1}]\leq n^{-1/2}
	\end{equation*}
	for all~$t\in[t_{\star}]$. This shows that~$Z$ is stochastically dominated by a binomial random variable with parameters~$t_{\star}$ and~$n^{-1/2}$ and thus Chernoff's inequality (Lemma~\ref{lemma: chernoff}) implies\COMMENT{%
		$\pr\biggl[Z\geq \biggl(1+\frac{\eta}{3}n^{3/8}\biggr)n^{-1/2}t_{\star}\biggr]
		\leq \exp(-\frac{(\frac{\eta}{3}n^{3/8})^2}{2+\frac{\eta}{3}n^{3/8}}n^{-1/2}t_{\star})
		\leq \exp(-\frac{(\frac{\eta}{3}n^{3/8})^2}{2\frac{\eta}{3}n^{3/8}}n^{-1/2}\frac{n^{1/3}}{2})
		= \exp(-\frac{\eta}{12}n^{3/8}n^{-1/2}n^{1/3})$
	}
	\begin{align*}
		\pr\biggl[Z\geq \frac{n^{-1/8}}{2}\frac{t_{\star}}{n}\abs{N_H(\unord{x})}\biggr]
		&\leq \pr\biggl[Z\geq \biggl(1+\frac{\eta}{3}n^{3/8}\biggr)n^{-1/2}t_{\star}\biggr]\\
		&\leq \exp(-n^{1/6}).
	\end{align*}
	The union bound yields~\eqref{equation: quatifiaction of removal}.
\end{proof}

\subsection{Building the absorbing structure}

The following result verifies the existence of few vertex-disjoint paths whose union contains many vertex absorbers.
Moreover, there is in fact a ``uniform'' probabilistic construction of these paths.
To state the next result, we introduce the following terminology concerning paths.
For a path~$P=v_1\ldots v_{\ell}$, we define the boundary~\defn{$\boundary{P}$} as~$\boundary{P}\defeq P[\{v_1,\ldots,v_k,v_{\ell-k+1},\ldots,v_{\ell}\}]$ if~$\ell\geq 2k+1$ and~$\boundary{P}\defeq P$ otherwise. 
Furthermore, for a collection of paths~$\cP$, we define~$\boundary{\cP}\defeq \bigcup_{P\in\cP} \boundary{P}$.

\begin{lemma}\label{lem: absstruct}
	Let~$1/n\ll \rho \ll 1/L \ll 1/a\ll \nu_+, 1/\ell \ll \eta,1/k$.
	Suppose~$H_+$ is an~$\eta$-intersecting~$k$-graph on at most~$(1+\nu_+)n$ vertices with vertex set~$V_+$ and~$H$ is an induced~$\rho$-almost regular subgraph of~$H_+$ on~$n$ vertices with edge set~$E$.
	Let~$\vartheta\defeq 1/a$ and for all~$x\in V_+$, denote the set of~$x$-absorbers in~$H_+$ by~$\cA_x$.
	Then, there is a probabilistic construction of a collection~$\cP$ of~$L$-paths in~$H$ together with a set~$\cB\subseteq \bigcup_{P\in\cP} \cB^{a,\ell}(P)$ such that the following holds.
	\begin{enumerate}[label=\textup{(\roman*)}]
		\item\label{item: number of paths} $\abs{\cP}\leq \vartheta^2 n/L$;
		\item\label{item: almost regularity persists} $H-V(\cP)$ is~$2\rho$-almost regular;
		\item\label{item: sufficiently many absorbing blocks} $\abs{\{ B\in\cB\colon \cA^{\ell}(B)\cap \cA_x\neq\emptyset \}}\geq 3\vartheta^4n$ for all~$x\in V_+$ (in particular,~$\abs{\cB}\geq 3\vartheta^4 n$);
		\item\label{item: no bad blocks} $\abs{\{ x\in V_+\colon \cA^{\ell}(B)\cap\cA_x =\emptyset \}}\leq \vartheta^4n$ for all~$B\in\cB$;
		\item\label{item: probability for occurence} $\pr[e\in E(\cP)]\leq \vartheta\frac{1}{n^{k-1}}$ and~$\pr[e\in E(\boundary{P})]\leq \frac{1}{L}\frac{1}{n^{k-1}}$ for all~$e\in E$.
	\end{enumerate}
\end{lemma}
\begin{proof}
	The key idea of the proof is as follows. Constructing~$\cP$ by starting with~$\cP=\emptyset$ and iteratively adding suitable random walks in~$H$ to~$\cP$ yields a~$\cP$ as desired apart from~\ref{item: probability for occurence} with probability at least~$1/5$. This can be used to obtain an appropriate probability distribution on such collections~$\cP$.
	
	More precisely, for~$t_{\star}\defeq n^{1/3}$, we will construct~$\cP$ in~$s_{\star}\defeq \vartheta^2 n/t_{\star}$ \defn{stages}, where in stage~$s\in[s_{\star}]$ we potentially add the~$L$-walks generated by an~$(L,\omega^{s-1})$-random walk~$X^s=X^s_1\ldots X^s_{t_{\star}}$ in a random subgraph~$H^{s-1}$ of~$H$ for a balanced perfect fractional matching~$\omega^{s-1}$ in~$H^{s-1}$.
	To this end, for every subgraph~$S\subseteq H$ with a balanced perfect fractional matching, fix one such perfect fractional matching in~$S$ to which we refer as the perfect fractional matching \defn{assigned} to~$S$.
	We inductively define the random~$k$-graphs~$H=H^0\supseteq \ldots\supseteq H^{s_{\star}}$ and random walks~$X^1,\ldots,X^{s_{\star}}$ as follows.
	\begin{enumerate}[label=\textup{(\arabic*)}]
		\item Let~$H^0\defeq H$;
		\item for~$s\in[s_{\star}]$, define the random walk~$X^s=X^s_1\ldots X^s_{t_{\star}}$ in~$H^{s-1}$ and~$H^s\subseteq H^{s-1}$ as follows.
		\begin{enumerate}[label=\textup{(\arabic{enumi}\alph*)}]
			\item If there is a balanced perfect fractional matching in~$H^{s-1}$, let~$X^s=X^s_1\ldots X^s_{t_{\star}}$ be an~$(L,\omega^{s-1})$-random walk in~$H^{s-1}$, where~$\omega^{s-1}$ denotes the perfect fractional matching assigned to~$H^{s-1}$;
			otherwise let~$X^s\defeq X^{s-1}$;
			\item if~$X^s$ is self-avoiding, let~$H^s\defeq H^{s-1}-\{ X^s_t\colon t\in[t_{\star}] \}$;
			otherwise let~$H^s\defeq H^{s-1}$.
		\end{enumerate}
	\end{enumerate}
	Note that Lemma~\ref{lemma: perfect fractional matching} guarantees a balanced perfect fractional matching in~$H_0=H$ and that for~$s\in[s_{\star}]$, if~$X^s=X^{s-1}$, then~$H^s=H^{s-1}$ by construction of~$H^{s-1}$ even if~$X^s$ is self-avoiding.
	
	We may think of stages~$s\in[s_{\star}]$ where there is no balanced perfect fractional matching in~$H^{s-1}$ and stages~$s\in[s_{\star}]$ where~$X^s$ is not self-avoiding as \defn{failed} stages as these stages fail to generate~$L$-paths disjoint from each other and all~$L$-paths previously added to~$\cP$. While stages that fail due to the absence of a balanced perfect fractional matching in~$H^{s-1}$ are \defn{fatal} in the sense that they entail failure in all subsequent stages, we can otherwise recover from failure in the sense that subsequent stages may still be successful. We repeat previously generated paths in the case of fatal failure simply for convenience.
	
	Later we show that with probability at least~$4/5$ no fatal failure occurs; that is, there is a balanced perfect fractional matching in~$H^{s-1}$ for all~$s\in[s_{\star}]$. We also show that with probability at least~$4/5$, there are at most~$n^{1/2}$ non-fatal failures, that is, there are at most~$n^{1/2}$ stages~$s\in[s_{\star}]$ such that~$X^s$ is not self-avoiding.
	
	For~$s\in[s_{\star}]_0$, let~$V^s\defeq V(H^s)$ and let~$V\defeq V^0=V(H)$.
	We use~$p_{\star}\defeq t_{\star}/L$ to refer to the number of~$L$-walks generated in every stage. For~$s\in[s_{\star}]$ and~$p\in[p_{\star}]$, let~$P^s_p\defeq X^s_{L(p-1)+1}\ldots X^s_{Lp}$ denote the~$p$-th walk generated in stage~$s$ and let~$\cP^s\defeq \{ P^s_p\colon p\in[p_{\star}] \}$ denote the set of all walks generated in stage~$s$. Let
	\begin{gather*}
		\cP'\defeq \bigcup_{s\in[s_{\star}]} \cP^s,\qquad \cB(\cP')\defeq \bigcup_{P\in\cP'} \cB^{a,\ell}(P),\\
		\cP\defeq \bigcup_{\substack{s\in[s_{\star}]\colon \text{$X^s$ is self-avoiding}}}\cP^s\qquad\text{ and}\qquad\cB(\cP)\defeq \bigcup_{P\in\cP} \cB^{a,\ell}(P).
	\end{gather*}
	
	In accordance with~\ref{item: no bad blocks}, we say that an~$(a,\ell)$-block~$B$ is \defn{good} if there are at most~$\vartheta^4 n$ vertices~$x\in V_+$ such that~$\cA^{\ell}(B)$ does not contain an~$x$-absorber in~$H_+$ and we call~$B$ \defn{bad} otherwise.
	We define events~$\cE_1$,~$\cE_2$,~$\cE_3$, and~$\cE_4$ as follows.
	\begin{enumerate}[label=$\cE_{\arabic*}$:]
		\item For all~$x\in V_+$, there are at least~$5\vartheta^4 n$ triples~$(s,p,i)$ with~$s\in [s_{\star}]$,~$p\in[p_{\star}]$,~$i\in [L/(a(2k+\ell))]$, and~$\cA^{\ell}(B^{a,\ell}_i(P^s_p))\cap \cA_x\neq\emptyset$;
		\item there are at most~$\vartheta^4 n$ bad blocks in~$\cB(\cP')$;
		\item there are at most~$n^{1/2}$ stages~$s\in[s_{\star}]$ such that~$X^s$ is not self-avoiding;
		\item $H^{s_{\star}}$ is~$(\rho+n^{-1/12})$-almost regular.
	\end{enumerate}
	We claim the following.
	\begin{align}\label{equation: event claim}
		\begin{minipage}[c]{0.875\textwidth}\em
			If~$\cE\defeq \cE_1\cap\ldots\cap \cE_4$ occurs, then~$\cP$ with~$\cB\defeq \{ B\in \cB(\cP)\colon \text{$B$ is good} \}$ satisfies~\ref{item: number of paths}--\ref{item: no bad blocks}.
		\end{minipage}\ignorespacesafterend
	\end{align}
	To see that this is true, we argue as follows.
	First observe that since~$H-X$ is~$(\eta/2)$-intersecting for all~$X\subseteq V$ of size at most~$\eta n/2$, the random~$k$-graphs~$H^s$ with~$s\in[s_{\star}]$ are~$(\eta/2)$-intersecting. Thus if~$\cE_4$ occurs, there exists a balanced perfect fractional matching in~$H^{s_{\star}}$ by Lemma~\ref{lemma: perfect fractional matching}. This implies that there was a balanced perfect fractional matching in~$H^{s}$ for all~$s\in[s_{\star}]_0$ and thus no fatal failure occured; otherwise~$H^{s_{\star}}$ would be equal to the first such~$k$-graph without a perfect fractional matching by construction.
	
	Next note that if~$\cE_1\cap\cE_3$ occurs, the number of triples~$(s,p,i)$ with~$s\in [s_{\star}]$,~$p\in[p_{\star}]$,~$i\in [L/(a(2k+\ell))]$, and~$\cA^{\ell}(B^{a,\ell}_i(P^s_p))\cap \cA_x\neq\emptyset$ such that~$X^s$ is self-avoiding is at least~$4\vartheta^4 n$ for all~$x\in V_+$. This together with the previous observation shows that whenever~$\cE_1\cap \cE_3\cap \cE_4$ occurs we have that for all~$x\in V_+$, there are at least~$4\vartheta^4 n$ blocks~$B\in\cB(\cP)$ with~$\cA^{\ell}(B)\cap \cA_x\neq\emptyset$. If~$\cE_2$ now also occurs in addition to~$\cE_1\cap \cE_3\cap \cE_4$, we lose at most~$\vartheta^4 n$ blocks by dropping bad blocks which shows that~\ref{item: sufficiently many absorbing blocks} holds.
	
	Finally, since we only consider good blocks,~\ref{item: no bad blocks} holds,~$s_{\star}\cdot p_{\star}=\vartheta^2 n/L$ implies that~\ref{item: number of paths} holds by construction of~$\cP$, and if~$\cE_4$ occurs,~\ref{item: almost regularity persists} holds because~$H-V(\cP)=H^{s_{\star}}$.
	This proves~\eqref{equation: event claim}.
	
	Let us finish the proof assuming~$\pr[\cE]\geq 1/5$. By~\eqref{equation: event claim}, this implies that choosing~$\hat{\cP}$ from all possible realisations of~$\cP$ such that~$\pr[\hat{\cP}=\cQ]=\pr[\cP=\cQ\mid\cE]$ for all possible realisations~$\cQ$ of~$\cP$ is a probabilistic construction as desired because for all~$e\in E$ and~$i\in [L-k+1]$, Lemma~\ref{lemma: probability for tuple} entails\COMMENT{%
		Use~$\frac{\omega^{s-1}(x_1\ldots x_j)}{\omega^{s-1}(\emptyset)}\leq \frac{\frac{n^{k-j}}{(k-j)!}\max_{e\in E}\omega^{s-1}(e)}{\frac{\eta}{2}\frac{n^k}{k!}\min_{e\in E}\omega^{s-1}(e)}\leq 2\frac{\frac{n^{k-j}}{(k-j)!}}{\frac{\eta}{2}\frac{n^k}{k!}}$
	}
	\begin{align*}
		\pr[\exists (X_t)_{t\in[L]}\in \hat{\cP}\colon \{X_{i},\ldots,X_{i+k-1}\}=e]&=\pr[\exists (X_t)_{t\in[L]}\in \cP\colon \{X_{i},\ldots,X_{i+k-1}\}=e\mid\cE]\\
		&\leq 5\sum_{s\in[s_{\star}]}\sum_{(X_t)_{t\in[L]}\in\cP^s}\pr[\{X_{i},\ldots,X_{i+k-1}\}=e]\\
		&= 5\sum_{s\in[s_{\star}]}\sum_{(X_t)_{t\in[L]}\in\cP^s}k!\frac{\omega^{s-1}(e)}{k!\omega^{s-1}(\emptyset)}\\
		&\leq 5\cdot s_{\star}\cdot p_{\star}\cdot k!\cdot 2\frac{1}{\frac{\eta}{2} n^k}
		\leq \frac{\vartheta}{L}\frac{1}{n^{k-1}},
	\end{align*}
	which implies~\ref{item: probability for occurence}.
	
	It remains to prove~$\pr[\cE]\geq 1/5$. This easily follows if~$\pr[\cE_i]\geq 4/5$ for all~$i\in[4]$, which is what we prove next. For~$x\in V_+$ and~$s\in[s_{\star}]$, let~$Y_{x,s}$ denote the random variable counting the pairs~$(p,i)$ with~$p\in [p_{\star}]$,~$i\in[L/(2k+\ell)]$, and~$A^{\ell}_i(P^s_p)\in \cA_x$. Note that~$\cE_1$ occurs if~$Y_{x,s}\geq 5\vartheta^3 n/s_{\star}$ for all~$x\in V_+$ and~$s\in[s_{\star}]$.
	Observation~\ref{observation: independent} in conjunction with Observation~\ref{observation: absorber occurs} shows that~$Y_{x,s}$ stochastically dominates a binomial random variable with parameters~$p_{\star}\frac{L}{2k+\ell}$ and~$\vartheta^{1/2}$. Hence Chernoff's inequality (Lemma~\ref{lemma: chernoff}) implies that
	\begin{align*}
		\pr\biggl[Y_{x,s}<\frac{5\vartheta^3 n}{s_{\star}}\biggr]\leq \pr\biggl[Y_{x,s}\leq\frac{1}{2}\cdot\vartheta^{1/2}\cdot p_{\star} \frac{L}{2k+\ell}\biggr]\leq \exp(-n^{1/4})
	\end{align*}
	and the union bound yields~$\pr[\cE_1]\geq 4/5$.
	
	For all~$x\in V_+$ and~$B\in\cB(\cP')$, Observation~\ref{observation: absorber occurs} implies
	\begin{equation*}
		\pr[\cA^{\ell}(B)\cap\cA_x =\emptyset]\leq (1-\vartheta^{1/2})^a\leq \vartheta^5/2.
	\end{equation*}
	This entails
	\begin{equation*}
		\ex[\abs{\{x\in V_+\colon \cA^{\ell}(B)\cap\cA_x =\emptyset\}}]\leq \vartheta^5n.
	\end{equation*}
	Using Markov's inequality we obtain~$\pr[\text{$B$ is bad}]\leq\vartheta$ for all~$B\in\cB(\cP')$. This shows that the expected value of the random variable counting the number of bad blocks in~$\cB(\cP')$ is at most
	\begin{equation*}
		\vartheta\cdot s_{\star}\cdot p_{\star}\cdot \frac{L}{a(2k+\ell)}= \frac{\vartheta^4}{2k+\ell}n\leq \vartheta^4 n/5.
	\end{equation*}
	Again using Markov's inequality, this yields~$\pr[\cE_2]\geq 4/5$.
	
	For all~$s\in[s_{\star}]$, Observation~\ref{observation: pairwise disjoint} implies that~$X^s$ is self-avoiding with probability at least~$1-2n^{-1/4}$. From this, we obtain that the expected value of the random variable counting the stages~$s\in[s_{\star}]$ where~$X^s$ is not self-avoiding is at most~$2n^{-1/4}\cdot s_{\star}= 2\vartheta^2 n^{5/12}\leq n^{1/2}/5$. Using Markov's inequality we conclude that~$\pr[\cE_3]\geq 4/5$.
	
	To see that~$\pr[\cE_4]\geq 4/5$ holds, we show that all random~$k$-graphs~$H^s$ with~$s\in[s_{\star}]_0$ are almost regular with high probability.
	More precisely, for~$s\in[s_{\star}]_0$, let~$\cE^s_4$ denote the event that~$H^s$ is~$(\rho+sn^{-3/4})$-almost regular. Our goal is to show
	\begin{equation}\label{equation: almost regular up to s}
		\pr\Bigl[\bigcap_{s'=0}^{s}\cE^{s'}_4\Bigr]\geq 1-s\exp(n^{-1/15})
	\end{equation}
	for all~$s\in[s_{\star}]_0$. This suffices because~$\cE^{s_{\star}}_4\subseteq \cE_4$. We proceed by induction on~$s$. First note that~$\pr[\cE^0_4]=1$.
	Next, assume that~\eqref{equation: almost regular up to s} holds for some~$s\in[s_{\star}-1]_0$.
	Then, Lemma~\ref{lemma: perfect fractional matching} guarantees that there is no fatal failure in stage~$s+1$.
	For~$U\defeq V^s\setminus \{ X^{s+1}_t\colon t\in[t_{\star}] \}$ and~$\unord{x}\in\unordsubs{V^{s}}{k-1}$, Lemma~\ref{lemma: same fraction} entails
	\begin{equation*}
		\pr\biggl[d_{H^{s}[U\cup \unord{x}]}(\unord{x})=(1\pm n^{-31/40})\biggl(1-\frac{t_{\star}}{\abs{V^{s}}}\biggr)d_{H^{s}}(\unord{x})\biggm|\bigcap_{s'=0}^{s}\cE^{s'}_4\biggr]\geq 1-\exp(-n^{1/14}).
	\end{equation*}
	Lemma~\ref{lemma: k-1 sets control}\COMMENT{%
	If~$X^{s+1}_1\ldots X^{s+1}_{t_{\star}}$ is not self-avoiding there is nothing to show by construction of~$H^{s+1}$. Assume that~$X^{s+1}_1\ldots X^{s+1}_{t_{\star}}$ is self-avoiding. If~$d_{H^{s}[U\cup \unord{x}]}(\unord{x})=(1\pm n^{-31/40})\biggl(1-\frac{t_{\star}}{\abs{V^{s}}}\biggr)d_{H^{s}}(\unord{x})$ holds for all~$\unord{x}\in\unordsubs{V^s}{k-1}$ and~$H^s$ is~$(\rho+sn^{-3/4})$-almost~$r$-regular for some~$r$, then Lemma~\ref{lemma: k-1 sets control} implies that~$H^{s+1}$ is~$(\rho+sn^{-3/4}+16k^3n^{-31/40})$-almost~$\biggl(1-\frac{t_{\star}}{\abs{V^{s}}}\biggr)^{k-1}r$-regular. Observe that~$\rho+sn^{-3/4}+16k^3n^{-31/40}\leq \rho+(s+1)n^{-3/4}$.} and the union bound yield~$\pr[\cE^{s+1}_4\mid\bigcap_{s'=0}^{s}\cE^{s'}_4]\geq 1-\exp(-n^{1/15})$ and thus by induction hypothesis~$\pr[\bigcap_{s'=0}^{s+1}\cE^{s'}_4]\geq 1-(s+1)\exp(-n^{1/15})$.
\end{proof}

In the following Proposition~\ref{prop:randabsorption} we employ the absorbers provided by Lemma~\ref{lem: absstruct} to absorb sets of vertices~$X$ in the sense that for all suitable~$X$, we find a selection of these absorbers such that the selected absorbers can simultaneously absorb all vertices in~$X$.
As this requires matching all vertices~$x\in X$ to an~$x$-absorber, we make use of the following observation which follows easily by iteratively using Hall's theorem and removing perfect matchings.
\begin{observation}\label{observation: bijaux}
	Suppose~$G$ is a bipartite graph on~$2n$ vertices with bipartition~$\{V_1,V_2\}$ with~$\abs{V_1}=\abs{V_2}=n$.
	For~$i\in[2]$, let~$\delta_i\defeq \min\{d_G(v)\colon v\in V_i¸\}$.
	Then, there are at least~$(\delta_1+\delta_2-n)/2$ disjoint perfect matchings in~$G$.
\end{observation}
\begin{prop}\label{prop:randabsorption}
	Let~$1/n\ll \rho \ll 1/L\ll \vartheta\ll \nu_+ \ll \eta,1/k$.
	Suppose~$H_+$ is an~$\eta$-intersecting~$k$-graph on at most~$(1+\nu_+)n$ vertices with vertex set~$V_+$ and edge set~$E_+$ and~$H$ is a~$\rho$-almost regular induced subgraph of~$H_+$ on~$n$ vertices.
	Then, there is a probabilistic construction of a collection~$\cP$ of~$L$-paths in~$H$ together with a function~$\sigma\colon \cP\rightarrow[L]_0$ such that the following holds.
	\begin{enumerate}[label=\textup{(\roman*)}]
		\item\label{item: random absorption number of paths} $\abs{\cP}\leq \vartheta^2 n/L$;
		\item\label{item: random absorption capacity}$c\defeq \sum_{P\in\cP} \sigma(P)\geq \vartheta^4n$;
		\item\label{item: random absorption regularity} $H- V(\cP)$ is~$2\rho$-almost regular;
		\item\label{item: random absorption absorption} for all probabilistic constructions of a set~$X\subseteq V_+$ with~$X\cap V(\cP)=\emptyset$ and~$\abs{X}=c$, there is a probabilistic construction of a collection~$\cP'$ of paths in~$H_+$ such that the following holds.
		\begin{enumerate}[label=\textup{(\roman{enumi}.\roman*)}]
			\item\label{item: random absorption absorption structure} There is a bijection~$\varphi:\cP\to\cP'$ such that for all~$P\in\cP$, the path~$\varphi(P)$ is an~$(L+\sigma(P))$-path with~$V(\varphi(P))\subseteq V(P)\cup X$ that has the same ordered end-edges as~$P$;
			\item\label{item: random absorption absorption probabilities 1} for all~$e\in E_+$, we have
			\begin{equation*}
				\pr[e\in E(\cP')]\leq \frac{1}{\vartheta^4 n^{k-1}}\qand\pr[e\in E(\boundary{\cP'})]\leq \frac{1}{Ln^{k-1}}.
			\end{equation*}
		\end{enumerate}
	\end{enumerate}
\end{prop}
\begin{proof}
	For the probabilistic construction of the collection~$\cP$ together with a function~$\sigma$, we will employ Lemma~\ref{lem: absstruct}. Then, for the probabilistic construction of~$\cP'$ we will randomly absorb all vertices~$x\in X$ into the paths in~$\cP$; that is, we will randomly place every vertex~$x\in X$ in the center of an~$x$-absorber in~$H_+$ that is a subgraph of a path in~$\cP$.
	
	In more detail, we argue as follows.
	Let~$E\defeq E(H)$.
	For~$x\in V_+$, let~$\cA_x$ denote the set of~$x$-absorbers in~$H_+$.
	Choose~$\ell$ such that~$\vartheta\ll \nu_+,1/\ell \ll \eta,1/k$.
	Let~$a\defeq 1/\vartheta$.
	Lemma~\ref{lem: absstruct} provides a probabilistic construction of a collection~$\cP$ of~$L$-paths in~$H$ together with a set~$\cB\subseteq \bigcup_{P\in\cP} \cB^{a,\ell}(P)$ such that the following holds.
	\begin{enumerate}[label=\textup{(\arabic*)}]
		\item\label{item: app number of paths} $\abs{\cP}\leq \vartheta^2 n/L$;
		\item\label{item: app almost regularity persists} $H-V(\cP)$ is~$2\rho$-almost regular;
		\item\label{item: app sufficiently many absorbing blocks} $\abs{\{ B\in\cB\colon \cA^{\ell}(B)\cap \cA_x\neq\emptyset \}}\geq 3\vartheta^4n$ for all~$x\in V_+$;
		\item\label{item: app no bad blocks} $\abs{\{ x\in V_+\colon \cA^{\ell}(B)\cap\cA_x =\emptyset \}}\leq \vartheta^4n$ for all~$B\in\cB$;
		\item\label{item: app probability for occurence} $\pr[e\in E(\cP)]\leq \vartheta\frac{1}{n^{k-1}}$ and~$\pr[e\in E(\boundary{P})]\leq \frac{1}{L}\frac{1}{n^{k-1}}$ for all~$e\in E$.
	\end{enumerate}
	Let~$\sigma\colon \cP\rightarrow [L]_0$ with~$\sigma(P)=\abs{\cB\cap \cB^{a,\ell}(P)}$ for all~$P\in\cP$ and let~$c\defeq \abs{\cB}$. Then~\ref{item: random absorption number of paths}--\ref{item: random absorption regularity} hold.
	
	For all~$(a,\ell)$-blocks~$B$ in~$H$ and all~$x\in V_+$ with~$\cA_x\cap \cA^{\ell}(B)\neq\emptyset$, we fix one~$x$-absorber~$A_x(B)\in \cA^{\ell}(B)$ in~$H_+$ for later use.
	For~$\cP$ as above and a probabilistic construction of a set~$X\subseteq V_+$ with~$X\cap V(\cP)=\emptyset$ and~$\abs{X}=c$, we obtain~$\cP'$ through random absorption of all elements of~$X$ into the paths in~$\cP$ as follows.
	Consider the auxiliary bipartite graph~$G$ with bipartition~$\{X,\cB\}$ and an edge between~$x\in X$ and~$B\in \cB$ if and only if~$\cA_x\cap \cA^{\ell}(B)\neq\emptyset$. Intuitively, edges in this graph represent possible absorptions of elements of~$X$ into the adjacent blocks.
	Thus, we may think of perfect matchings in~$G$ as representations of valid stategies for the absorption of all elements of~$X$.
	Due to~\ref{item: app sufficiently many absorbing blocks} and~\ref{item: app no bad blocks}, we have~$d_G(x)\geq 3\vartheta^4n$ for all~$x\in X$ and~$d_G(B)\geq c-\vartheta^4n$ for all~$B\in\cB$.
	Therefore, Observation~\ref{observation: bijaux} guarantees the existence of a set of~$\vartheta^4 n$ edge-disjoint perfect matchings in~$G$.
	Pick one matching uniformly at random from this set and interpret this matching as a bijection~$\mu\colon X\rightarrow\cB$.
	Let~$\varphi$ denote the bijection defined on~$\cP$ that maps~$P\in\cP$ to the path obtained from~$P$ when placing~$\mu^{-1}(B)$ in the center of~$A_{\mu^{-1}(B)}(B)$ for all~$B\in\cB\cap\cB^{a,\ell}(P)$. Let~$\cP'\defeq \Ima(\varphi)$.
	
	It remains to check that this defines a probabilistic construction of a collection~$\cP'$ of paths satisfying~\ref{item: random absorption absorption probabilities 1}.
 	For all~$\unord{x}\in \unordsubs{V_+}{k-1}$ and all possible realisations~$\cQ$ of~$\cP$, there is at most one block~$B\in \bigcup_{P\in \cQ}\cB^{a,\ell}(P)$ with~$d_B(\unord{x})\geq 1$.
	Hence, for all~$x\in V_+$ and~$\unord{x}\in \unordsubs{V_+}{k-1}$, we have
	\begin{equation}\label{equation: probability matching x to rest}
		\pr[x\in X\wedge d_{\mu(x)}(\unord{x})\geq 1\mid \cP,X]\leq \frac{1}{\vartheta^4 n}
	\end{equation}
	by construction of~$\mu$.
	Thus, for all~$e\in E_+$,~\ref{item: app probability for occurence} yields
	\begin{align*}
		\pr[e\in E(\cP')]&\leq\pr[\unord{x}\cap X=\emptyset\wedge e\in E(\cP)]
		+\sum_{x\in e}\pr[x\in X\wedge  d_{\mu(x)}(e\setminus\{x\})\geq 1]\\
		&\overset{\mathclap{\ref{item: app probability for occurence}}}{\leq}\frac{\vartheta}{n^{k-1}}+\sum_{x\in e}\pr[x\in X\wedge d_{\mu(x)}(e\setminus\{x\})\geq 1\wedge d_{\cP}(e\setminus\{x\})\geq 1]\\
		&\overset{\mathclap{\eqref{equation: probability matching x to rest}}}{\leq}\frac{\vartheta}{n^{k-1}}+\sum_{x\in e}\frac{1}{\vartheta^4n}\pr[d_{\cP}(e\setminus\{x\})\geq 1]\\
		&\leq \frac{\vartheta}{n^{k-1}}+\sum_{x\in e}\frac{1}{\vartheta^4n}\sum_{v\in V_+}\pr[(e\setminus\{x\})\cup\{v\}\in E(\cP)]\\
		&\overset{\mathclap{\ref{item: app probability for occurence}}}{\leq} \frac{\vartheta}{n^{k-1}}+j\cdot \frac{1}{\vartheta^4 n}\cdot (1+\nu_+)n\cdot\frac{\vartheta}{n^{k-1}}
		\leq \frac{1}{\vartheta^4 n^{k-1}}.
	\end{align*}
	Furthermore, we obtain
	\begin{equation*}
		\pr[e\in E(\boundary{\cP'})]=\pr[e\in E(\boundary{\cP})]\leq \frac{1}{L}\frac{1}{n^{k-1}},
	\end{equation*}
	which completes the proof.
\end{proof}

\section{From paths to cycles}\label{sec:approxtoreal}
In this section, we perform the step from the yield of Proposition~\ref{prop: simultaneous path cover} to the decomposition into cycle factors as in our main theorem.
We do this by describing a random process which converts one almost spanning path collection into a cycle factor and subsequently using another random process which repeatedly applies the first one to transform path collection after path collection.
The first step is done in Lemma~\ref{lem:layer} and the second in Proposition~\ref{prop:cycdecomptoHCdecomp}.

We begin with the following somewhat standard lemma, which states that in any~$\eta$-intersecting~$k$-graph, there is a small (reservoir) set such that between any two ordered edges, there are many short paths with all ``inner'' vertices in this set.
Complementing the notation~$\boundary{P}$, for a path~$P=v_1,\dots,v_{\ell}$, we define the \emph{interior} of~$P$ as (the subpath)~$P^{\circ}=v_{k+1},\dots,v_{\ell-k}$ if~$\ell\geq 2k+1$ and~\defn{$P^{\circ}\defeq P[\emptyset]$} otherwise.
The vertex set of~$P^{\circ}$ is the set of \defn{inner vertices} of~$P$.
For a collection~$\cP$ of paths, we define~$\cP^{\circ}=\bigcup_{P\in\cP}P^{\circ}$.
\begin{lemma}\label{lem:reservoir}
	Suppose~$1/n \ll \beta,\rho \ll1/\ell_0, 1/\ell_1\ll \eta, 1/k$, where~$\ell_0\leq\ell_1$.
	Suppose~$H$ is an~$\eta$-intersecting~$\rho$-almost regular~$k$-graph on~$n$ vertices with vertex set~$V$ and edge set~$E$.
	Then there is a set~$R\subseteq V$ such that the following holds.
	\begin{enumerate}[label=\textup{(\roman*)}]
		\item\label{item: reservoir size} $\beta n/2 \leq\abs{R}\leq \beta n$;
		\item\label{item: reservoir connectivity}
		for all~$\orde{s},\orde{t}\in\ordedges$ with~$s\cap t=\emptyset$ and all integers~$\ell\in[\ell_0,\ell_1]$, the number of~$\orde{s}$--$\orde{t}$-paths~$P$ in~$H$
		with~$\vert V(P^{\circ})\vert=\ell$ and~$V(P^{\circ})\subseteq R$ is at least~$\beta \vert R\vert ^{\ell}$;
		\item\label{item: reservoir regularity} $H-R$ is~$2\rho$-almost regular.
	\end{enumerate}
\end{lemma}

This follows easily from Lemma~\ref{lemma: connecting} by considering a random set of vertices in which each vertex is included independently at random (for instance, with probability~$3\beta/4$) and using Chernoff's inequality, McDiarmid's inequality, and the union bound to show that the random set has the desired properties with high probability.
We omit the calculations, which are standard by now.

The next lemma ensures that under the right conditions, many tuples can be connected in a probabilistically well behaved way.
It will later be useful when building the cycle factor in Lemma~\ref{lem:layer}.
\begin{lemma}\label{lem:probconn}
	Suppose~$1/n\ll\zeta\ll\beta,1/\ell_1,1/\ell_0\ll1/k$, where~$\ell_0\leq\ell_1$. 
	Suppose~$H$ is a~$k$-graph on~$n$ vertices with vertex set~$V$ and edge set~$E$.
	Suppose that~$\mathcal{Q}=\{\orde{s}_1,\orde{t}_1,\dots,\orde{s}_{m},\orde{t}_m\}\subseteq \ordedges$ is a random set with~$m\leq \zeta n$,~$e\cap e'=\emptyset$ for all distinct~$\orde{e},\orde{e}'\in\cQ$, and~$\mathds{P}[\orde{e}\in\cQ]\leq\frac{\zeta}{n^{k-1}}$ for every~$\orde{e}\in\ordedges$.
	Further, for all~$i\in[m]$, let~$\lambda_i\in[\ell_0,\ell_1]$ and suppose that~$\cP_{i}$ is a set of at least~$\beta n^{\lambda_i}$~$\orde{s}_i$--$\orde{t}_i$-paths with~$\lambda_i$ inner vertices.
	Then, there is a probabilistic construction of a collection~$\cW\subseteq\bigcup_{i\in[m]}\cP_i$ of paths with~$\vert\cW\cap\cP_i\vert=1$ for all~$i\in[m]$ and~$\PP[e\in E(\cW)]\leq\frac{\zeta^{1/2}}{n^{k-1}}$ for every~$e\in E$.
\end{lemma}

\begin{proof}
	We aim to connect the ordered edges~$\orde{s}_i$ and~$\orde{t}_{i}$ by choosing one of the paths in~$\cP_i$ uniformly at random.
	However, since the paths shall be vertex-disjoint, we employ an iterative procedure where in each step we only consider those paths in~$\cP_i$ which are vertex-disjoint from all previously chosen ones and all ordered edges in~$\cQ$.
	
	Suppose we have already chosen paths~$W_1,\dots,W_{i-1}$ for some~$i\in[m-1]$, where~$W_j$ is
	an~$\orde{s}_{j}$--$\orde{t}_{j}$-path with~$\lambda_j$ inner vertices for all~$j\in[i-1]$.
	Then $$\Big\vert\bigcup_{j\in[i-1]} V(W_j)\cup\bigcup_{\orde{e}\in\cQ}e\Big\vert\leq\ell_1\cdot m+2(k-1)\cdot m\leq \frac{\beta}{2}n\,.$$
	Thus, since~$\vert\cP_i\vert\geq\beta n^{\lambda_i}$, there are still at least~$
	\beta n^{\lambda_i}-\frac{\beta}{2}n^{\lambda_i}=\frac{\beta}{2} n^{\lambda_i}$ paths in~$P\in\cP_{i}$ with~$V(P)\cap(\bigcup_{j\in[i-1]} V(W_j)\cup\bigcup_{\orde{e}\in\cQ}e)=\emptyset$.
	Choose one of these uniformly at random as~$W_i$.
	
	Now, set~$\cW=\{W_1,\dots,W_m\}$ and let us analyse the probabilities.
	For all~$j\in[k]_0$,~$\unord{x}\in\unordsubs{V}{j}$ and~$i\in[m]$, the number of~$\orde{s}_i$--$\orde{t}_i$-paths~$W$ in~$H$ with~$\abs{V(W^{\circ})}=\lambda_i$ and~$\unord{x}\subseteq V(W^{\circ})$ is at most~$\ell_1^k \cdot n^{\lambda_i-j}$. Therefore, by the choice of~$\cW$, this implies~$\pr[\unord{x}\subseteq V(W_i^{\circ})\mid \orde{s}_1,\ldots,\orde{s}_m,\orde{t}_1,\ldots,\orde{t}_m]\leq \frac{2\ell_1^k}{\beta n^{j}}$.
	Furthermore, for all~$j\in[k]$ and~$\unord{x}\in\unordsubs{V}{j}$, the number of ordered edges~$\orde{e}\in\ordedges$ with~$\unord{x}\subseteq e$ is at most~$k^k\cdot n^{k-j}$ and hence~$\pr[\exists \orde{e}\in\cQ\colon \unord{x}\subseteq e]\leq \frac{\zeta k^k}{n^{j-1}}$ holds.
	Thus, for all~$e\in E$, we obtain
	\begin{align*}
		\pr[e\in E(\cW)]&\leq \sum_{i\in[m]}\pr[e\subseteq s_i\cup V(W_i^{\circ})\vee e\subseteq V(W_i^{\circ})\cup t_i]\\
		&\leq \sum_{i\in[m]}\sum_{\unord{y}\subseteq e} \sum_{\orde{e}'\in\{\orde{s}_i,\orde{t}_i\}}\pr[e\setminus\unord{y}\subseteq V(W_i^{\circ})\mid \unord{y}\subseteq e']\pr[ \unord{y}\subseteq e']\\
		&\leq \sum_{\unord{y}\subseteq e}\frac{2\ell_1^{k}}{\beta n^{\abs{e\setminus\unord{y}}}}\sum_{i\in[m]}\sum_{\orde{e}'\in\{\orde{s}_i,\orde{t}_i\}}\pr[\unord{y}\subseteq e']\\
		&=\frac{4\ell_1^{k}m}{\beta n^k}+\sum_{\unord{y}\subseteq e\colon\unord{y}\neq\emptyset}\frac{2\ell_1^{k}}{\beta n^{\abs{e\setminus\unord{y}}}}\pr[\exists \orde{e}'\in\cQ\colon \unord{y}\subseteq e']\\
		&\leq \frac{4\zeta\ell_1^{k}}{\beta n^{k-1}}+2^{k}\cdot \frac{\zeta k^k\cdot 2\ell_1^k}{\beta n^{k-1}}\leq\frac{\zeta^{1/2}}{n^{k-1}}\,,
	\end{align*}
	which completes the proof.
\end{proof}

In the following lemma, we transform a path collection as yielded by Proposition~\ref{prop: simultaneous path cover} into a cycle factor.
It resembles the usual final step in a proof by absorption, in particular, this is where we will use Proposition~\ref{prop:randabsorption}.
However, in order to subsequently apply this construction process multiple times in the proof of Proposition~\ref{prop:cycdecomptoHCdecomp} without ruining certain quasirandom properties, we need to construct the cycle factor probabilistically and take care that the probability for any edge to occur in the constructed cycle factor is small enough.
Figure~\ref{pic: transformation} illustrates this process.

\begin{lemma}\label{lem:layer}
	Suppose~$1/n\ll \mu,1/L\ll\delta\ll\rho\ll \eta,1/k$. 
	Suppose~$H$ is a~$k$-graph on~$n$ vertices with vertex set~$V$ and~$F$ is an~\rob ~$\rho$-almost regular spanning subgraph of~$H$.
	Let~$\mathcal{P}\subseteq H-F$ be a collection of~$L$-paths with~$\vert V(\mathcal{P}) \vert\geq (1-\mu) n$ and let~$\cC$ be a cycle factor on~$n$ vertices of girth at least~$L^3$.
	
	Then, there is a probabilistic construction of a copy~$\cC'$ of~$\cC$ in~$H$ such that~$\cC' \subseteq \cP\cup F$ and for every~$e\in E(F)$ of type~$\tau\in\mathcal{T}\setminus\{\typeend[k]\}$, we have~$\mathds{P}[e\in E(F\cap \cC')]\leq p_{\tau}$, where~$p_{1\textup{-con}}=\frac{\delta^{1/2}}{n^{k-1}}$,~$p_{j\textup{-end}}=\frac{1}{\delta^{k} n^{k-j}}$ for all~$j\in[k-1]$, and~$p_{j
		\textup{-con}}=p_{\textup{lo}}=\frac{1}{\delta ^{k}n^{k-1}}$ for all~$j\in[k]\setminus\{1\}$.
\end{lemma}

\begin{proof}
	We use the absorption method to transform the almost spanning collection of paths~$\cP$ into the desired cycle factor~$\cC'$.
	To this end, we perform the usual steps of a proof via absorption in a~$k$-graph similar to~$F[V\setminus V(\cP)]$.
	That is, we set aside the absorbing structure (via Proposition~\ref{prop:randabsorption}), cover almost everything (via Proposition~\ref{prop: simultaneous path cover}), connect all the paths in the approximate covering, in the absorbing structure, and, in this case, paths in~$\cP$ (via Lemma~\ref{lemma: connecting}), and finally, we absorb the remaining vertices.
	Note however, that here we actually want to perform all these steps as a probabilistic construction and analyse the probabilities for edges of~$F$ to occur in~$F\cap\cC'$.
	
	Suppose
	\begin{align}
		1/n\ll \mu,1/L\ll\delta\ll\rho\ll 1/L'\ll\mu'\ll\beta\ll\vartheta\ll\ell_1\ll\ell_0\ll\eta,1/k\,.
	\end{align} 
	We say a set~$\hat{\mathcal{P}}\subseteq \mathcal{P}$ is \emph{good} if for~$\hat{V}=V\setminus V(\hat{\mathcal{P}})$, we have~$\vert \hat{V}\vert\in [\frac{\delta}{2}n,\frac{3}{2}\delta n]$, for all~$\unord{x},\unord{y}\in \unordsubs{V}{k-1}$, we have~$\vert N_{F}(\unord{x})\cap N_{F}(\unord{y})\cap\hat{V}\vert\geq \frac{\eta}{2} \vert \hat{V}\vert$, and~$F[\hat{V}]$ is~$2\rho$-almost regular.
	Consider a random selection~$\mathcal{P}_{\text{rand}}$ of paths in~$\mathcal{P}$ which includes every path in~$\mathcal{P}$ independently with probability~$1-\delta$.
	Then McDiarmid's inequality (Lemma~\ref{lemma: McDiarmid}) guarantees that~$\mathds{P}[\mathcal{P}_{\text{rand}}\text{ is good}]\geq\frac{99}{100}$, say.
	Denote the set of good sets by~$\ccP$ and pick a set~$\mathcal{P}'=\{P_1,\dots ,P_m\}$ at random from~$\ccP$ such that $$\PP[\cP'=\cQ]=\PP[\cP_{\text{rand}}=\cQ\mid\cP_{\text{rand}}\in\ccP]$$
	for all~$\cQ\in\ccP$.
	Then~$\mathcal{P}'$ is good and we set~$V_1=V\setminus V(\mathcal{P}')$.
	For~$P\in\mathcal{P}$, we have~$V(P)\subseteq V_1$ if and only if~$P\notin\mathcal{P}'$, and each path in~$\cP$ is included in~$\cP'$ independently with probability~$1-\delta$.
	Thus, for distinct~$\widetilde{P}_1,\dots,\widetilde{P}_k\in\cP$, we have 
	\begin{align}\label{eq:layerprobpathcont}
		\mathds{P}\Big[\bigcup_{i\in[k]}V(\widetilde{P}_i)\subseteq V_1\Big]=&\frac{\mathds{P}[\cP_{\text{rand}}\in\ccP\land \forall i\in[k]: \widetilde{P}_i\notin \mathcal{P}_{\text{rand}}]}{\mathds{P}[\cP_{\textup{rand}}\in\ccP]}\nonumber\\
		\leq&\PP[\forall i\in[k]: \widetilde{P}_i\notin \mathcal{P}_{\text{rand}}]\cdot\frac{100}{99}\leq \frac{100\delta^k}{99}\,.
	\end{align}
	
	Next, we describe the construction of a copy~$\cC'$ of~$\cC$ and analyse $\mathds{P}[e\in E(F\cap\cC')\mid \cP']$ for all~$e\in E$ (note that fixing~$\cP'$ in particular fixes~$V_1$).
	In the end, we use this to deduce the upper bound on the probabilities for different types of edges of~$F$ to lie in~$E(F\cap\cC')$.
	Since we perform several probabilistic constructions sequentially, in principle, we could obtain all the probabilities conditioned on all previous steps.
	However, for our analysis it suffices to condition only on the choice of~$\cP'$.
	
	\begin{figure}[t]
		\centering
		\begin{tikzpicture}
		\tikzmath{
			\standardLineWidth=0.08;
			\pNumber=5;
			\pRadius=0.2;
			\pLength=5;
			\pSpacing=0.5;
			\pLineWidth=2*\standardLineWidth;
			\pBraceOffset=1.7;
			\pBraceLineWidth=\standardLineWidth;
			\pPrimeNumber=3;
			\pPrimeLineWidth=\pLineWidth;
			\pPrimeBraceOffset=0.6;
			\pPrimeBraceLineWidth=\standardLineWidth;
			\pIIndex=2;
			\wTopSpacing=0.7;
			\wNumber=4;
			\wRadius=0.36;
			\wLength=1.6;
			\wSpacing=\wRadius;
			\wLineWidth=0.08;
			\wBraceOffset=0.25;
			\wBraceLineWidth=\standardLineWidth;
			\sNumber=3;
			\sRadius=\wRadius;
			\sLength=0.75;
			\sSpacing=\sRadius;
			\sLineWidth=\wLineWidth;
			\sBraceOffset=\wBraceOffset;
			\sBraceLineWidth=\standardLineWidth;
			\uvLength=0.4;
			\uvLineWidth=\wLineWidth;
			\xBottomSpacing=0.7;
			\xWidth=2;
			\xHeight=1;
			\connectionLineWidth=0.04;
			\pwConnectionOffsetFactor=0.6;
			\rWidth=1;
			\rHeight=2;
			\rLabelBottomWeight=0.6;
			\rLabelOffset=1;
			\arrowLineWidth=0.1;
			\arrowTipOffset=0.2;
			\vOffset=0.12;
			\vLineWidth=\connectionLineWidth;
			\vLabelOffset=0.6;
			\vLabelBottomWeight=0.6;
			\braceAmplitude=0.2;
		}
		
		\definecolor{pColor}{rgb}{0,0.6,0}
		\definecolor{wColor}{rgb}{0,0,1}
		\definecolor{sColor}{rgb}{0.75,0,0.75}
		\definecolor{rColor}{rgb}{1,0.5,0}
		\definecolor{uvColor}{rgb}{0,0.75,1}
		\definecolor{vColor}{rgb}{0.9,0.9,0.9}
		\tikzset{
			pPath/.style={draw=pColor, opacity=0.25},
			pLabel/.style={text=pColor, opacity=0.25},
			pPrimePath/.style={draw=pColor},
			pPrimeLabel/.style={text=pColor},
			wPath/.style={draw=wColor},
			wLabel/.style={text=wColor},
			sPath/.style={draw=sColor},
			sLabel/.style={text=sColor},
			xSet/.style={draw=sColor, fill=sColor, opacity=0.4},
			xLabel/.style={text=sColor},
			connection/.style={draw=rColor},
			uvPath/.style={draw=uvColor},
			uvLabel/.style={text=uvColor},
			rSet/.style={inner color=rColor, outer color=vColor, opacity=0.45},
			rLabel/.style={text=rColor},
			vArea/.style={draw=black, fill=vColor},
			vLabel/.style={text=black},
			helper/.style={inner sep=0, outer sep=0}
		}
		
		\foreach \i in {1,...,\pNumber} {
			\draw[pPath, line width=\pLineWidth*1cm] ++({\pLength+\pRadius+0.5*\pLineWidth},{-(2*\pRadius+\pSpacing)*(\i-1)-0.5*\pLineWidth}) --++ (-\pLength,0) node[helper, pos=0](p\i-start){} arc (90:270:\pRadius) node[helper, pos=0](p\i-arc-start){} node[helper, pos=0.5](p\i-center){} node[helper, pos=1](p\i-arc-end){} --++ (\pLength,0) node[helper, pos=1](p\i-end){};
		}
		
		\foreach \i in {1,...,\pPrimeNumber} {
			\draw[pPrimePath, line width=\pPrimeLineWidth*1cm] ++({\pLength+\pRadius+0.5*\pLineWidth},{-(2*\pRadius+\pSpacing)*(\i-1)-0.5*\pLineWidth}) --++ (-\pLength,0) arc (90:270:\pRadius) --++ (\pLength,0);
		}
		
		\foreach \i in {1,...,\wNumber} {
			\draw[wPath, line width=\wLineWidth*1cm] ++({0.5*\wLineWidth+(2*\wRadius+\wSpacing)*(\i-1)},{-(\pPrimeNumber*(2*\pRadius+\pSpacing)-\pSpacing+\pLineWidth+\wTopSpacing)}) --++ (0,-\wLength) node[helper, pos=0](w\i-start){} arc (180:360:\wRadius) node[helper, pos=0](w\i-arc-start){} node[helper, pos=0.5](w\i-center){} node[helper, pos=1](w\i-arc-end){} --++ (0,\wLength) node[helper, pos=1](w\i-end){};
		}
		
		\foreach \i in {1,...,\sNumber} {
			\draw[sPath, line width=\sLineWidth*1cm] ++({0.5*\sLineWidth+(2*\wRadius+\wSpacing)*\wNumber+(2*\sRadius+\sSpacing)*(\i-1)},{-(\pPrimeNumber*(2*\pRadius+\pSpacing)-\pSpacing+\pLineWidth+\wTopSpacing+\wLength+\wRadius-\sLength-\sRadius)}) --++ (0,-\sLength) node[helper, pos=0](s\i-start){} arc (180:360:\sRadius) node[helper, pos=0](s\i-arc-start){} node[helper, pos=0.5](s\i-center){} node[helper, pos=1](s\i-arc-end){} --++ (0,\sLength) node[helper, pos=1](s\i-end){};
		}
		
		\foreach \i in {1,...,\pPrimeNumber} {
			\draw[uvPath, line width=\uvLineWidth*1cm] (p\i-start.center) --++ (\uvLength,0) node[helper, pos=0.5](u\i-center){} node[helper, pos=1](u\i-end){};
			\draw[uvPath, line width=\uvLineWidth*1cm] (p\i-end.center) --++ (\uvLength,0) node[helper, pos=0.5](v\i-center){} node[helper, pos=1](v\i-end){};
		}
		
		\pgfmathsetmacro{\endIndex}{\pPrimeNumber-1}
		\foreach \i in {1,...,\endIndex} {
			\pgfmathsetmacro{\next}{int(\i+1)}
			\draw[connection, line width=\connectionLineWidth*1cm] (v\i-end.center) to[out=0, in=0] (u\next-end.center);
		}
		
		\pgfmathsetmacro{\endIndex}{\wNumber-1}
		\foreach \i in {1,...,\endIndex} {
			\pgfmathsetmacro{\next}{int(\i+1)}
			\draw[connection, line width=\connectionLineWidth*1cm] (w\i-end.center) to[out=90, in=90] (w\next-start.center);
		}
		
		\pgfmathsetmacro{\endIndex}{\sNumber-1}
		\foreach \i in {1,...,\endIndex} {
			\pgfmathsetmacro{\next}{int(\i+1)}
			\draw[connection, line width=\connectionLineWidth*1cm] (s\i-end.center) to[out=90, in=90] (s\next-start.center);
		}
		
		\draw[connection, line width=\connectionLineWidth*1cm] (v\pPrimeNumber-end.center) to[out=0, in=0] ([yshift=-\pwConnectionOffsetFactor*\pSpacing*1cm]v\pPrimeNumber-end.center) --++ (-\pLength,0) to[out=180,in=90] (w1-start.center);
		
		\draw[connection, line width=\connectionLineWidth*1cm] (w\wNumber-end.center) to[out=90,in=90] (s1-start.center|-w\wNumber-end.center) -- (s1-start.center);
		
		\draw[connection, line width=\connectionLineWidth*1cm] (s\sNumber-end.center) -- (s\sNumber-end.center|-v1-end.center) to[out=90,in=0,out looseness=0.7] (u1-end.center);
		
		\path (s1-start.center)--(s\sNumber-end.center) node[helper, pos=0.5](s-center){};
		\path (s-center.center) --++ (0,\xBottomSpacing) node[helper, pos=1](x-center){};
		\begin{scope}[transparency group, xSet]
		\node[draw=none, fill=sColor, ellipse, minimum width=\xWidth*1cm, minimum height=\xHeight*1cm] (X) at (x-center.center){};
		\foreach \i in {1,...,\sNumber} {
			\path (s\i-start.center) -- (s\i-end.center) node[helper, pos=0.5](s\i-control){};
			\draw[-{latex}, line width=\arrowLineWidth*1cm, line cap=round] (X.center) .. controls (s\i-control.center) .. ([yshift=\arrowTipOffset*1cm]s\i-center.center);
		}
		\end{scope}
		
		\path (u1-end.center) -- (v\pPrimeNumber-end) node[pos=\rLabelBottomWeight, helper](uv-end-center){};
		\begin{pgfonlayer}{layer-1}
		\node[rSet, minimum width=\rWidth*1cm, minimum height=\rHeight*1cm, ellipse] at ([xshift=\rLabelOffset*1cm]uv-end-center.center) {};
		\end{pgfonlayer}
		
		\begin{pgfonlayer}{layer-2}
		\filldraw[line width=\vLineWidth*1cm, vArea] (p1-start.center) ++ (0,{(0.5*\pLineWidth+\vOffset)*1cm})  -- (p\pPrimeNumber-end.center) node[pos=0, helper](v-start){} --++ (0,{(-\pSpacing+(1-\pwConnectionOffsetFactor)*\pSpacing+\vOffset)*1cm}) --++ ({(-\pLength-\pRadius-0.5*\pLineWidth-\vOffset)*1cm},0) --++ (0,{(-\wTopSpacing-\wLength-\wRadius-0.5*\wLineWidth+(1-\pwConnectionOffsetFactor)*\pSpacing-2*\vOffset)*1cm}) --++ ({(\wNumber*(2*\wRadius+\wSpacing)+\sNumber*(2*\sRadius+\sSpacing)-\sSpacing+\sLineWidth+2*\vOffset)*1cm},0)node[pos=1, helper](v-bottom-right){} |- cycle ;
		\end{pgfonlayer}
		\node[helper] at (v-start.center -| v-bottom-right.center) (v-top-right){};
		
		\draw[decorate, decoration={brace, amplitude=\braceAmplitude*1cm, mirror}, pPath, line width=\pBraceLineWidth*1cm, line cap=round] ([xshift={(-\pRadius-0.5*\pLineWidth-\pBraceOffset)*1cm}]p1-arc-start.center) -- ([xshift={(-\pRadius-0.5*\pLineWidth-\pBraceOffset)*1cm}]p\pNumber-arc-end.center) node[pos=0.5, left, pLabel, xshift=-\braceAmplitude*1cm]{$\cP$};
		
		\draw[decorate, decoration={brace, amplitude=\braceAmplitude*1cm, mirror}, pPrimePath, line width=\pPrimeBraceLineWidth*1cm, line cap=round] ([xshift={(-\pRadius-0.5*\pPrimeLineWidth-\pPrimeBraceOffset)*1cm}]p1-arc-start.center) -- ([xshift={(-\pRadius-0.5*\pPrimeLineWidth-\pPrimeBraceOffset)*1cm}]p\pPrimeNumber-arc-end.center) node[pos=0.5, left, pPrimeLabel, xshift=-\braceAmplitude*1cm]{$\cP'$};
		\node[left, pPrimeLabel] at (p\pIIndex-center){$P_i$};
		
		\draw[decorate, decoration={brace, amplitude=\braceAmplitude*1cm, mirror}, wPath, line width=\wBraceLineWidth*1cm, line cap=round] ([yshift={(-\wRadius-0.5*\wLineWidth-\wBraceOffset)*1cm}]w1-arc-start.center) -- ([yshift={(-\wRadius-0.5*\wLineWidth-\wBraceOffset)*1cm}]w\wNumber-arc-end.center) node[pos=0.5, below, wLabel, yshift=-\braceAmplitude*1cm]{$\widehat{\cW}$};
		
		\draw[decorate, decoration={brace, amplitude=\braceAmplitude*1cm, mirror}, sPath, line width=\sBraceLineWidth*1cm, line cap=round] ([yshift={(-\sRadius-0.5*\sLineWidth-\sBraceOffset)*1cm}]s1-arc-start.center) -- ([yshift={(-\sRadius-0.5*\sLineWidth-\sBraceOffset)*1cm}]s\sNumber-arc-end.center) node[pos=0.5, below, sLabel, yshift=-\braceAmplitude*1cm]{$\cS$};
		
		\node[right, uvLabel] at (u\pIIndex-end.center) {$\ordered{u}_i$};
		\node[right, uvLabel] at (v\pIIndex-end.center) {$\ordered{v}_i$};
		
		\node[xLabel] at (x-center.center) {$X$};
		
		\node[rLabel] at ([xshift=\rLabelOffset*1cm]uv-end-center.center){$R$};
		
		\path (v-top-right.center) -- (v-bottom-right.center) node[pos=\vLabelBottomWeight, right, vLabel]{$V_1$};
		
		\end{tikzpicture}
		\caption{Transforming a path collection into a Hamilton cycle.}\label{pic: transformation}
	\end{figure}
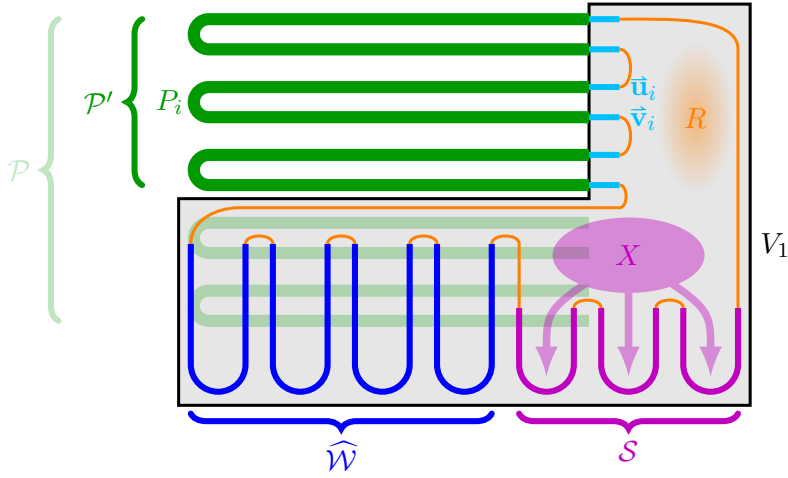
	
	Set aside a set~$R\subseteq V_1$ provided by applying Lemma~\ref{lem:reservoir} with~$F[V_1]$,~$\beta$,~$\ell_0$,~$\ell_1$,~$2\rho$,~$\eta/2$ here taking the roles of~$H$,~$\beta$,~$\ell_0$,~$\ell_1$,~$\rho$,~$\eta$ in Lemma~\ref{lem:reservoir}.
	Then~$\frac{\beta}{2}\vert V_1\vert\leq\vert R\vert\leq \beta\vert V_1\vert$ and for all~$\orde{s},\orde{t}\in\ordedges(F)$ with~$s\cap t=\emptyset$ and all integers~$\ell\in[\ell_0,\ell_1]$, the number of~$\orde{s}$--$\orde{t}$-paths~$P$ in~$H$ with~$\vert V(P^{\circ})\vert=\ell$ and~$V(P^{\circ})\subseteq R$ is at least~$\beta \vert R\vert ^{\ell}$.
	Further, Lemma~\ref{lem:reservoir} guarantees that~$F[V_1\setminus R]$ is~$4\rho$-almost regular.
	
	Since we later only want to deal with edges of~$F[V_1]$, we extend all paths in~$\cP'$ by an edge of~$F$ at each end.
	More precisely, for~$i\in[m]$, let~$\orde{a}_i$ be the ordered starting edge and~$\orde{b}_i$ be the ordered ending edge of~$P_i\in\cP'$ and inductively choose~$\orde{u}_i=(u_i^1,\dots,u_i^{k})\in\ordedges(F)$ and~$\orde{v}_i=(v_i^1,\dots,v_i^{k})\in\ordedges(F)$ for each~$i\in[m]$ as follows.
	
	Suppose that for some~$i\in[m]$, we already have defined ordered edges~$\orde{u}_{i'},\orde{v}_{i'}\in \ordedges(F)$ for all~$i'\in[i-1]$ such that~$\orde{u}_{i'}P_{i'}\orde{v}_{i'}$ is a path in~$H$.
	Set $$U_i=V_1\setminus \Big(R\cup\bigcup_{i'\in[i-1]}u_{i'}\cup v_{i'}\Big)$$
	and note that since~$m\leq\frac{n}{L}$,  we know that~$\vert U_i\vert\geq (1-\frac{3}{2}\beta)\vert V_1\vert$ holds.
	Moreover, since~$\mathcal{P}'$ is good, we know that for all~$\unord{x},\unord{y}\in\unordsubs{V}{k-1}$ and~$U\subseteq V_1$ with~$\vert U\vert\geq (1-2\beta)\vert V_1\vert$, we have
	\begin{align}\label{eq:neighinV1-}
		\vert N_{F}(\unord{x})\cap N_{F}(\unord{y})\cap U\vert\geq \frac{\eta}{3}\vert V_1\vert\qand \vert N_F(\unord{x})\cap U\vert\geq \frac{\eta}{3}\vert V_1\vert\,.
	\end{align}
	
	Suppose for some~$j\in[k]$, vertices~$u_i^{j'}$ are given for all~$j'\in\{j+1,\dots,k\}$.
	Then choose~$u_i^j\in N_F(u_i^{j+1},\dots,u_i^{k},a_i^{1},\dots,a_i^{j-1})\cap U_i$ uniformly at random (note that due to~\eqref{eq:neighinV1-}, this means that~$u_i^j$ is chosen uniformly at random from a set of size at least~$\frac{\eta}{3}\vert V_1\vert$).
	Subsequently, suppose that for some~$j\in[k]$, vertices~$v_i^{j'}$ are given for all~$j'\in[j-1]$.
	Then choose~$$v_i^j\in N_F(b_i^{j+1},\dots,b_i^{k},v_i^1,\dots,v_i^{j-1})\cap (U_i\setminus u_i)$$ uniformly at random.
	By this definition, the edges~$u_1,v_1,\dots,u_m,v_m$ are pairwise disjoint.
	Furthermore, observe that the definition of~$\orde{u}_i$ and~$\orde{v}_i$ together with~\eqref{eq:neighinV1-} yield the following.
	For all~$j\in[k]$,~$i\in[m]$, and~$\unord{x}\in\unordsubs{V}{j}$,
	\begin{align}\label{eq:probinoneuv}
		\mathds{P}[\unord{x}\subseteq u_i\lor\unord{x}\subseteq v_i\mid \cP']\leq 2k^j\left(\frac{3}{\eta \vert V_1\vert}\right)^{j}\leq\frac{\ell_0}{\vert V_1\vert^{j}}\,.
	\end{align}
	
	Now set~$V_2= V_1\setminus (R\cup\bigcup_{i\in[m]} u_{i}\cup v_{i})$ and note that by~\eqref{eq:neighinV1-},~$F[V_2]$ is~$\eta/3$-intersecting.
	Furthermore, since~$F[V_1\setminus R]$ is~$\eta/3$-intersecting and~$4\rho$-almost regular and~$\vert\bigcup_{i\in[m]} u_i\cup v_i\vert\leq \frac{2kn}{L}$, we know that~$F[V_2]$ is~$5\rho$-almost regular.
	
	Now choose~$\nu_+$ such that~$\vartheta\ll\nu_+\ll\eta,1/k$ and apply Proposition~\ref{prop:randabsorption} with~$F[V_1]$,~$F[V_2]$, $5\rho$, $L'$, $\vartheta$, $\nu_+$,~$\eta/2$ here in place of~$H_+$,~$H$,~$\rho$,~$L$,~$\vartheta$,~$\nu_+$,~$\eta$ there.
	
	This engenders a probabilistic construction of a pair~$(\cS,\sigma)$ where~$\cS$ is a collection of~$L'$-paths in~$F[V_2]$ and~$\sigma\colon \cS\rightarrow [L']_0$ is a function such that the following holds.
	\begin{enumerate}[label=\textup{(S\arabic*)}]
		\item\label{item: app random absorption number of paths} $\abs{\cS}\leq \vartheta^2 \vert V_2\vert/L'$;
		\item\label{item: app random absorption capacity}$c=\sum_{S\in\cS} \sigma(S)\geq \vartheta^4\vert V_2\vert$;
		\item\label{item: app random absorption regularity} setting~$V_3=V_2\setminus V(\mathcal{S})$, we have that~$F[V_3]$ is~$10\rho$-almost regular;
		\item\label{item: app random absorption absorption} for all probabilistic constructions of a set~$X\subseteq V_1$ with~$X\cap V(\cS)=\emptyset$ and~$\abs{X}=c$, there is a probabilistic construction of a collection~$\cS'$ of paths in~$F[V_1]$ such that the following holds.
		\begin{enumerate}[label=\textup{(S\arabic{enumi}.\arabic*)}]
			\item\label{item: app random absorption absorption structure} There is a bijection~$\varphi:\cS\to\cS'$ such that for all~$S\in\cS$, the path~$\varphi(S)$ is an~$(L'+\sigma(S))$-path with~$V(\varphi(S))\subseteq V(S)\cup X$ that has the same ordered end-edges as~$S$;
			\item\label{item: app random absorption absorption probabilities} for all~$e\in E(F)$, we have
			\begin{equation*}
			\pr[e\in E(\cS')\mid \cP']\leq \frac{1}{\vartheta^4 \vert V_2\vert^{k-1}}\qand\pr[e\in E(\boundary{\cS'})\mid \cP']\leq \frac{1}{L'\vert V_2\vert^{k-1}}.
			\end{equation*}
		\end{enumerate}
	\end{enumerate}
	
	The next step is to cover almost all vertices of~$F[V_3]$ by long paths.
	As in the other steps, we need to do this in a probabilistic way.
	This will be achieved by utilising a weak version of Proposition~\ref{prop: simultaneous path cover} followed by some random selections.
	
	By~\ref{item: app random absorption number of paths},~$F[V_3]$ is~$\eta/4$-intersecting and due to property~\ref{item: app random absorption regularity}, it is~$10\rho$-almost~$r$-regular for some~$r\geq \frac{\eta}{5}\binom{\vert V_3\vert}{k-1}$. 
	Thus, setting~$r'=(1-10\rho)\frac{r}{k}$ yields~$kr'\leq d_{F[V_3]}(v)\leq(1+21\rho)kr'$ for all~$v\in V_3$.
	Hence, we can apply Proposition~\ref{prop: simultaneous path cover} with~$F[V_3]$,~$21\rho$,~$L'$,~$\eta/4$,~$\mu'$,~$r'$ here instead of~$H$,~$\rho$,~$L$,~$\eta$,~$\mu$,~$r$ there.
	As remarked after the statement of Proposition~\ref{prop: simultaneous path cover}, the proof also provides edge-disjoint collections~$\cW_1,\ldots,\cW_{r'}$ of~$L'$-cycles in~$F[V_3]$ with~$V(\cW_i)\geq (1-\mu')\vert V_3\vert$ for all~$i\in[r']$.
	Let~$\widehat{\cW}$ be a random collection of~$L'$-paths obtained by choosing one of~$\mathcal{W}_1,\dots ,\mathcal{W}_{r'}$ uniformly at random and then independently deleting~$k-1$ consecutive edges in each cycle.
	Note that~$\vert V(\widehat{\cW})\vert\geq(1-\mu')\vert V_3\vert$ and that for each~$e\in E(F)$, we have
	\begin{align}\label{eq:layerprobcover}
		\mathds{P}[e\in E(\widehat{\cW})\mid \cP']\leq\frac{1}{r'}\leq\frac{\ell_0}{\vert V_1\vert^{k-1}}\,
	\end{align}
	and
	\begin{align}\label{eq:layerprobcoverend}
		\mathds{P}[e\in E(\boundary{\widehat{\cW}})\mid \cP']\leq\frac{2}{L'r'}\leq\frac{1}{L'^{2/3}\vert V_1\vert^{k-1}}\,.
	\end{align}
	
	Next, we will utilise Lemma~\ref{lem:probconn} applied to the ordered end-edges of the paths in~$\widehat{\cW}$, the ordered end-edges of the paths in~$\cS$, and the ordered edges~$\orde{u}_1,\orde{v}_1,\dots,\orde{u}_m,\orde{v}_m$ to connect the respective paths to cycles.
	First, we say which paths we aim to put into one cycle and afterwards we prepare for the application of Lemma~\ref{lem:probconn}. 
	
	Let~$C_1,\dots,C_h$ be the cycles in~$\cC$ and for~$i\in[h]$, set~$L_i=\vert V(C_i)\vert\geq L^3$.
	We now inductively define collections of paths~$\cZ_1,\dots,\cZ_h$ which we use to construct copies of the cycles~$C_1,\dots,C_h$.
	Suppose that for~$i\in[h]$, we already have chosen collections of paths~$\cZ_1,\dots,\cZ_{i-1}$.
	Next, we pick a set~$\cZ_i$ of previously unused paths.
	More precisely, for~$i'\in[m]$, write~$P_{i'}'=\orde{u}_{i'}P_{i'}\orde{v}_{i'}$ and choose $$\mathcal{Z}_i\subseteq\big(\{P_1',\dots,P_m'\}\cup\widehat{\cW}\cup\cS\big)\setminus\bigcup_{j\in[i-1]}\cZ_j$$ such that~$\vert\cZ_i\cap\cS\vert$ is maximal with~$\vert\mathcal{Z}_i\vert\cdot\ell_0+\vert V(\mathcal{Z}_i)\vert +\sum_{S\in\cS\cap\cZ_i}\sigma(S)\leq L_i$ and such that subject to this~$\vert \cZ_i\vert$ is maximal.
	Observe that since~$\vert V(\{P_1',\dots,P_m'\}\cup\widehat{\cW}\cup\cS)\vert\geq n-\vartheta^4\vert V_2\vert$, since each path in~$\{P_1',\dots,P_m'\}\cup\widehat{\cW}\cup\cS$ has at most~$2L$ vertices (and for a path~$S\in\cS$, we in fact have~$\vert V(S)\vert+\sigma(S)\leq 2L$), and by~\ref{item: app random absorption capacity}, this definition implies 
	\begin{align}\label{eq:numvertinpath}
		L_i-2L-\ell_0<\vert\mathcal{Z}_i\vert\cdot\ell_0+\vert V(\mathcal{Z}_i)\vert+\sum_{S\in\cS\cap\cZ_i}\sigma(S)\leq L_i\,.
	\end{align}
	Therefore,~$L_i\geq L^3$ entails~$\vert\cZ_i\vert\geq L$ and so~$\vert\cZ_i\vert (\ell_1-\ell_0)\geq 2L-\ell_0$.
	Together with~\eqref{eq:numvertinpath}, this implies\COMMENT{%
		Note that~\eqref{eq:numvertinpath} gives~$\vert\cZ_i\vert\cdot\ell_0\leq L_i-\Big(\vert V(\mathcal{Z}_i)\vert+\sum_{S\in\cS\cap\cZ_i}\sigma(S)\Big)$ and adding~$\vert\cZ_i\vert (\ell_1-\ell_0)\geq 2L-\ell_0$ to the first inequality of~$\eqref{eq:numvertinpath}$ gives~$\vert\cZ_i\vert\cdot\ell_1\geq L_i-\Big(\vert V(\mathcal{Z}_i)\vert+\sum_{S\in\cS\cap\cZ_i}\sigma(S)\Big)$.
		Hence, we have~$\vert\cZ_i\vert\cdot\ell_0\leq L_i-\Big(\vert V(\mathcal{Z}_i)\vert+\sum_{S\in\cS\cap\cZ_i}\sigma(S)\Big)\leq\vert\cZ_i\vert\cdot\ell_1$ which yields the existence of the integers~$\lambda_i^j$.
	}
	that there are integers~$\lambda_i^1,\dots,\lambda_i^{\vert \cZ_i\vert}\in[\ell_0,\ell_1]$ with $$\sum_{j\in[\vert\cZ_i\vert]}\lambda_i^j=L_i-\Big(\vert V(\mathcal{Z}_i)\vert+\sum_{S\in\cS\cap\cZ_i}\sigma(S)\Big)\,.$$
	
	Next, we aim to connect the paths in~$\mathcal{Z}_i$ by means of Lemma~\ref{lem:probconn} to a cycle~$Z_i$ for every~$i\in[h]$, using paths with~$\lambda_i^1,\dots,\lambda_i^{\vert \cZ_i\vert}$ inner vertices.
	To this end, list the paths in~$\cZ_i$ arbitrarily as~$A_i^1,\dots, A_i^{\vert\cZ_i\vert}$ for all~$i\in[h]$, and for all~$i\in[h]$ and~$g\in[\vert\cZ_i\vert]$, let~$\orde{s}_i^g$ and~$\orde{t}_i^g$ be the starting and ending edge of~$A_i^g$, respectively.
	Further, write~$p_{\star}=\sum_{i\in[h]}\vert\cZ_i\vert$ and for~$p\in[p_{\star}]$, let~$\iota_1(p),\iota_2(p)$ be such that~$(\iota_1(p),\iota_2(p))$ is the~$p$-th element in the lexicographic ordering of the tuples in~$\{(i,g):i\in[h],g\in[\vert\cZ_i\vert]\}$.
	Now, for every~$p\in[p_{\star}]$, set~$\orde{s}_p=\orde{s}_{\iota_1(p)}^{\iota_2(p)+1}$ (where we view the upper index modulo~$\iota_2(p)$) and~$\orde{t}_p=\orde{t}_{\iota_1(p)}^{\iota_2(p)}$.
	Note that when for every~$p\in[p_{\star}]$, we connect~$\orde{s}_p$ and~$\orde{t}_p$ by a path with~$\lambda_{\iota_1(p)}^{\iota_2(p)}$ inner vertices, we obtain cycles~$Z_1,\dots,Z_h$ such that~$\vert V(Z_i)\vert=L_i-\sum_{S\in\cS\cap\cZ_i}\sigma(S)$ for all~$i\in[h]$.
	
	Set~$\cQ=\{\orde{s}_1,\orde{t}_1,\dots,\orde{s}_{p_{\star}},\orde{t}_{p_{\star}}\}$ and note that $$\vert\cQ\vert\leq 2\big(m+\vert\widehat{\cW}\vert+\vert\cS\vert\big)\leq \frac{2n}{L}+\frac{2\vert V_1\vert}{L'}+\frac{2\vartheta^2\vert V_2\vert}{L'}\leq\frac{\vert V_1\vert}{L'^{1/2}}\,.$$
	Further, by~\ref{item: app random absorption absorption probabilities}, by~\eqref{eq:probinoneuv} (together with the union bound), and by~\eqref{eq:layerprobcoverend} we know that for~$e\in E(F)$,
	\begin{align}\label{eq:layerprobtuples}
		\mathds{P}[\orde{e}\in\mathcal{Q}\mid \cP']\leq \frac{\vartheta}{L'\vert V_2\vert^{k-1}}+\frac{1}{L^{1/2}\vert V_1\vert^{k-1}}+\frac{1}{L'^{2/3}\vert V_1\vert^{k-1}}\leq\frac{1}{L'^{1/2}\vert V_1\vert^{k-1}}\,.
	\end{align}
	
	For~$i\in[p_{\star}]$, let~$\cP_{i}$ be the set of all~$\orde{s}_{i}$--$\orde{t}_{i}$-paths~$P$ with~$\vert V(P^{\circ})\vert=\lambda_{\iota_1(i)}^{\iota_2(i)}$ and~$V(P^{\circ})\subseteq R$.
	Recall that the properties of~$R$ guarantee that~$\vert\cP_{i}\vert\geq\frac{\beta}{2}\vert R\vert^{\lambda_{\iota_1(i)}^{\iota_2(i)}}\geq\beta^{\ell_1+2}\vert V_1\vert^{\lambda_{\iota_1(i)}^{\iota_2(i)}}$.
	
	Now we apply Lemma~\ref{lem:probconn} with~$F[V_1]$,~$\cQ$,~$\big(\lambda_{\iota_1(i)}^{\iota_2(i)}\big)_{i\in[p_{\star}]}$,~$(\cP_i)_{i\in[p_{\star}]}$,~$\frac{1}{L'^{1/2}}$,~$\beta^{\ell_1+2}$, $\ell_1$, $\ell_0$ here instead of~$H$,~$\cQ$,~$(\lambda_i)_{i\in[m]}$,~$(\cP_i)_{i\in[m]}$,~$\zeta$,~$\beta$,~$\ell_1$,~$\ell_0$ there.
	
	Lemma~\ref{lem:probconn} then yields a probabilistic construction of a collection of paths~$\cW'\subseteq\bigcup_{i\in[p_{\star}]}\cP_i$ with~$\vert\cW'\cap\cP_i\vert=1$ for all~$i\in[p_{\star}]$ such that for every~$e\in E(F)$, we have
	\begin{align}\label{eq:layerprobconn}
		\mathcal{P}[e\in E(\cW')\mid \cP']\leq\frac{1}{L'^{1/4}\vert V_1\vert^{k-1}}\,.
	\end{align}
	This leaves us with cycles~$Z_1,\dots,Z_h$ such that for all~$i\in[h]$, the cycle~$Z_i$ contains the paths in~$\cZ_i$ as subpaths and~$\vert V(Z_i)\vert=L_i-\sum_{S\in\mathcal{S}\cap\cZ_i}\sigma(S)$.
	Observe that in particular, every element of~$\mathcal{S}$ is a subpath in one of the cycles~$Z_1,\dots,Z_h$.

	We aim to absorb the set~$X=V_1\setminus \bigcup_{i\in[h]} V(Z_i)$ of not yet covered vertices into the paths in~$\mathcal{S}$.
	For this, note that since~$\vert V(Z_i)\vert=L_i-\sum_{S\in\cS\cap\mathcal{Z}_i}\sigma(S)$ for all~$i\in[h]$, we have~$\vert X\vert=\sum_{S\in\cS}\sigma(S)=c$ because~$\sum_{i\in[h]}L_i=n$.
	So property~\ref{item: app random absorption absorption} indeed allows us to absorb~$X$ into the cycles~$Z_i$.
	More precisely, there is a probabilistic construction of a set~$\cS'$ of vertex-disjoint paths in~$F[V_1]$ such that
	\begin{itemize}
		\item\label{item: random absorption absorption structure} there is a bijection~$\varphi:\cS\to\cS'$ such that for all~$S\in\cS$, the path~$\varphi(S)$ is an~$(L'+\sigma(S))$-path with~$V(\varphi(S))\subseteq V(S)\cup X$ that has the same end-tuples as~$S$;
		\item\label{item: random absorption absorption probabilities} for all~$e\in E(F)$, we have
		\begin{align}\label{eq:probinS'}
			\pr[e\in E(\cS')\mid \cP']\leq \frac{1}{\vartheta^4 \vert V_2\vert^{k-1}}\qand\pr[e\in E(\boundary{\cS'})\mid \cP']\leq \frac{1}{L'\vert V_2\vert^{k-1}}.
		\end{align}
	\end{itemize}
	Due to the properties of~$\cS'$, replacing the every path~$S\in\cS\cap\cZ_i$ in the cycle~$Z_i$ by the path~$\varphi(S)$ for all~$i\in[h]$ leaves us with vertex-disjoint cycles~$C_1',\dots ,C_h'$ with~$\vert C_i'\vert=L_i$, that is,~$C_i'$ is a copy of~$C_i$ for all~$i\in[h]$.
	Thus we have constructed a copy of the cycle factor~$\cC$.

\medskip

	Finally, let us collect the upper bounds for the probabilities for the edges of different types to occur in the constructed cycle factor~$\cC'=\bigcup_{i\in[h]}C_i'$.
	First note that for~$i'\in[m]$ and~$e\in E(F)$ which is not ending (with respect to~$\mathcal{P}$), we have that~$e\in E(\orde{u}_{i'}\orde{a}_{i'})$ can only happen for some~${i'}\in[m]$ if~$e=u_{i'}$ and the analogous remark holds for~$\orde{b}_{i'}\orde{v}_{i'}$.
	Therefore, given~$e\in E(F)$ which is not ending,~\eqref{eq:probinoneuv} (together with the union bound) yields
	\begin{align}\label{eq:layeruvfinal}
		\mathds{P}\Big[e\in\bigcup_{{i'}\in[m]}E(\orde{u}_{i'}\orde{a}_{i'})\cup E(\orde{b}_{i'}\orde{v}_{i'})\mid \cP'\Big]\leq\frac{1}{L^{1/2}\vert V_1\vert^{k-1}}\,.
	\end{align}
	For~$j\in[k-1]$ and a~$j$-ending edge~$e\in E(F)$, we consider all partitions~$\{\unord{x}_1,\unord{x}_2\}$ of~$e$ for which there exists a~$P\in\cP$ such that~$\unord{x}_1$ is an end-set of~$P$.
	The event~$e\in E(\orde{u}_{i'}\orde{a}_{i'})$ (respectively~$e\in E(\orde{b}_{i'}\orde{v}_{i'})$) can only happen for some~${i'}\in[m]$, if for one of these partitions,~$\unord{x}_1$ is an end-set of~$P_{i'}\in\cP'$ and~$\unord{x}_2\subseteq u_{i'}$ (respectively~$\unord{x}_2\subseteq v_{i'}$).
	Note that for a fixed realisation of~$\cP'$, this can happen for at most one partition and one~${i'}\in[m]$.
	Further, since~$e$ is~$j$-ending, for each such partition, we have~$\vert\unord{x}_1\vert\leq j$.
	Thus,~\eqref{eq:probinoneuv} implies that
	\begin{align}\label{eq:layeruvfinal2}
		\mathds{P}\Big[e\in\bigcup_{{i'}\in[m]}E(\orde{u}_{i'}\orde{a}_{i'})\cup E(\orde{b}_{i'}\orde{v}_{i'})\mid \cP'\Big]\leq\frac{\ell_0}{\vert V_1
		\vert^{k-j}}\,.
	\end{align}
	
	Note that the probabilities analysed after~\eqref{eq:layerprobpathcont} were the probabilities for edges of~$F$ to occur in some subpath conditioned on the choice of~$\cP'$.
	While we do not need to make use of this for most types, we will do so for~$1$-concentrated edges.
	First note that for~$e\in E(F)$, if~$e\in E(F\cap \cC')$, then
	\begin{align}\label{eq:layerpossocc}
		e\in E(\widehat{\cW})\cup E(\cW')\cup E(\cS')\cup \bigcup_{{i'}\in[m]}E(\orde{u}_{i'}\orde{a}_{i'})\cup E(\orde{b}_{i'}\orde{v}_{i'})\,,
	\end{align}
	and if~$e$ is not ending, then in addition~$e\subseteq V_1$ holds.
	If~$e\in E(F)$ is~$1$-concentrated with respect to~$\cP$, for the event~$e\subseteq V_1$ to occur,~$V(P)\subseteq V_1$ has to hold for each of the~$k$ paths in~$\cP$ which contain a vertex of~$e$. 
	Therefore, using~\eqref{eq:layerpossocc} together with the bounds on the respective probabilities in~\eqref{eq:layerprobcover},~\eqref{eq:layerprobconn},~\eqref{eq:probinS'},~\eqref{eq:layeruvfinal}, and~\eqref{eq:layerprobpathcont}, entails
	\begin{align*}
		\mathds{P}[e\in E(F\cap\cC')]=\mathds{P}\big[e\in E(F\cap\cC')\mid e\subseteq V_1\big]\mathds{P}[e\subseteq V_1]\leq\frac{L'}{\vert V_1\vert^{k-1}}\frac{100\delta^{k}}{99}\leq\frac{\delta^{1/2}}{n^{k-1}}\,.
	\end{align*}
	
	If~$e\in E(F)$ is~$j$-ending (with respect to~$\cP$) for a~$j\in[k-1]$,~\eqref{eq:layerpossocc} together with the bounds in~\eqref{eq:layerprobcover},~\eqref{eq:layerprobconn},~\eqref{eq:probinS'}, and~\eqref{eq:layeruvfinal2} give that 
	\begin{align*}
		\mathds{P}[e\in E(F\cap\cC')]
		&\leq \mathds{P}\big[e\in E(\widehat{\cW})\cup E(\cW')\cup E(\cS')\big]+\mathds{P}\Big[e\in\bigcup_{{i'}\in[m]}E(\orde{u}_{i'}\orde{a}_{i'})\cup E(\orde{b}_{i'}\orde{v}_{i'})\Big]\\
		&\leq\frac{L'}{\vert V_1\vert^{k-1}}+\frac{\ell_0}{\vert V_1\vert^{k-j}}\leq\frac{1}{\delta^{k}n^{k-j}}\,.
	\end{align*}
	Lastly, if~$e\in E(F)$ is neither~$1$-concentrated nor ending, \eqref{eq:layerpossocc} together with~\eqref{eq:layerprobcover},~\eqref{eq:layerprobconn},~\eqref{eq:probinS'}, and~\eqref{eq:layeruvfinal} entail that $$\mathds{P}[e\in E(F\cap\cC')]\leq\frac{L'}{\vert V_1\vert^{k-1}}\leq\frac{1}{\delta^{k} n^{k-1}}\,.$$
	
	Summarised, we provided a probabilistic construction of a cycle factor~$\cC'$ which is a copy of the cycle factor~$\cC$ such that for every~$e\in E(F)$, the probability~$\mathds{P}[e\in E(F\cap\cC')]$ is bounded as claimed in the statement.
\end{proof}

The next proposition says that given an approximate decomposition of the edge set into approximate partitions of the vertex set into long paths, that is, given the setup after applying Proposition~\ref{prop: simultaneous path cover}, we can indeed obtain an approximate decomposition of the edge set into given cycle factors of not too small girth.

\begin{prop}\label{prop:cycdecomptoHCdecomp}
	Suppose~$1/n\ll 1/L\ll\mu\ll\rho\ll \eta,1/k$. 
	Suppose~$H$ is a~$k$-graph on~$n$ vertices with vertex set~$V$ and~$F$ is an~$\eta$-intersecting~$\rho$-almost regular spanning subgraph of~$H$.
	Suppose~$\mathcal{P}_1,\dots ,\mathcal{P}_{r}$ are edge-disjoint collections of~$L$-paths in~$H-F$ with~$\vert V(\cP_i)\vert\geq (1-\mu) n$ for all~$i\in[r]$ such that for all~$\unord{e}\in \unordsubs{V}{k}$, we have
	\begin{enumerate}[label=\textup{(\roman*)}]
		\item\label{it:leftover} $\vert\mathcal{I}_{\textup{lo}}(\unord{e})\vert\leq\mu r$;
		\item\label{it:jcon} $\vert\mathcal{I}_{j\textup{-con}}(\unord{e})\vert\leq 2^{L^2}n^{k-j}$ for all~$j\in[k-1]$ and~$\vert\mathcal{I}_{k\textup{-con}}(\unord{e})\vert\leq 2^{L^2}n$;
		\item\label{it:j-end} $\vert\mathcal{I}_{j\textup{-end}}(\unord{e})\vert\leq\frac{n^{k-j}}{L^{1/2}}$ for all~$j\in[k-1]$.
	\end{enumerate}
	Then, for all cycle factors~$\cC_1,\dots,\cC_r$ on~$n$ vertices of girth at least~$L^3$, there are edge-disjoint copies of~$\mathcal{C}_1,\dots,\mathcal{C}_r$ in~$H$.
\end{prop}

\begin{proof}
	In the following, we will analyse a random process in which we utilise Lemma~\ref{lem:layer} to transform each collection of paths~$\mathcal{P}_i$ in a fairly uniform way into a cycle factor~$\cC_i'$ that is a copy of~$\cC_i$.
	
	Suppose~$\mu\ll\delta\ll \rho$ and define a random process inductively as follows.
	Suppose that for some~$i\in[r]$ and all~$j\in[i-1]$ and~$\unord{x}\in\unordsubs{V}{k-1}$, we have defined a cycle factor~$\cC'_{j}$ in~$H$ and a~$\{0,1\}$-valued random variable~$Y_j^{\unord{x}}$.
	Further, set~$\cC'_{<i}=\bigcup_{j\in[i-1]}\cC'_j$.
	
	Let~$\cE_i$ be the event that for all~$\unord{x}\in\unordsubs{V}{k-1}$, we have that~$ d_{F\cap\cC'_{<i}}(\unord{x})\leq\delta^{1/4}n$.
	If~$\cE_i$ does not occur, we set~$\cC'_i=\emptyset$ and choose~${Y_i^{\unord{x}}}_{\vert\cE_i^c}$ such that~$\mathds{P}[Y_i^{\unord{x}}=1\mid (Y^{\unord{x}}_j)_{j\in[i-1]}, \cE_i^c]=\sum_{\unord{e}\in\unordsubs{V}{k}\colon\unord{x}\subseteq\unord{e}}p_{\tau(\unord{e},i)}$.
	In the end, we will show that with high probability,~$\cE_i$ occurs for all~$i\in[r]$.
	
	If~$\cE_i$ occurs, it follows that for all~$\unord{x},\unord{y}\in\unordsubs{V}{k-1}$, we have
	$$\vert N_{F-\cC'_{<i}}(\unord{x})\cap N_{F-\cC'_{<i}}(\unord{y})\vert\geq\vert N_{F}(\unord{x})\cap N_{F}(\unord{y})\vert-2\delta^{1/4} n\geq\frac{\eta} {2}n$$ and that~$F-\cC_{<i}'$ is~$2\rho$-almost regular.
	This means that if~$\cE_i$ occurs, we may apply Lemma~\ref{lem:layer}, with~$H$,~$F-\cC'_{<i}$,~$L$,~$\cC_i$,~$\cP_i$,~$\eta/2$,~$2\rho$,~$\delta$,~$\mu$ here instead of~$H$,~$F$,~$L$,~$\cC$,~$\cP$,~$\eta$,~$\rho$,~$\delta$,~$\mu$ there, to obtain a probabilistic construction of a cycle factor~$\cC'_i$ in~$H$ such that
	\begin{itemize}
		\item  $\cC'_i$ is a copy of the cycle factor~$\mathcal{C}_i$;
		\item $E(\cC'_i)\subseteq E(\cP_i)\cup E(F-\cC'_{<i})$ (in particular,~$\cC'_i$ is edge-disjoint from every cycle factor~$\cC'_{j}$ for all~$j\in[i-1]$);
		\item $\mathds{P}[\unord{e}\in E(F\cap\cC'_i)\mid (Y_j^{\unord{x}})_{j\in[i-1]},\cE_i]\leq p_{\tau (\unord{e},i)}$ for every~$\unord{e}\in \unordsubs{V}{k}$ where~$p_{1\textup{-con}}=\frac{\delta^{1/2}}{n^{k-1}}$, $p_{j\textup{-end}}=\frac{1}{\delta^{k} n^{k-j}}$ for all~$j\in[k-1]$, $p_{\typeend[$k$]}=0$, and~$p_{j
			\textup{-con}}=p_{\textup{lo}}=\frac{1}{\delta ^{k}n^{k-1}}$ for all~$j\in[k]\setminus\{1\}$ (here~$p_{\typeend[$k$]}=0$ holds since if~$\unord{e}\in\unordsubs{V}{k}$ is~$k$-ending with respect to~$\cP_i$, then~$\unord{e}\notin E(F)$).
	\end{itemize}
	Denoting the indicator variable of the event~$d_{F\cap\cC'_i}(\unord{x})\geq 1$ by~$I^{\unord{x}}_i$ for every~$\unord{x}\in\unordsubs{V}{k-1}$, this definition of~$\cC'_i$  implies~$\mathds{P}[I^{\unord{x}}_{i}=1\mid (Y_j^{\unord{x}})_{j\in[i-1]},\mathcal{E}_i]\leq \sum_{\unord{e}\in\unordsubs{V}{k}\colon\unord{x}\subseteq\unord{e}} p_{\tau (\unord{e},i)}$ and we set~${Y^{\unord{x}}_i}_{\vert\cE_i}={I_i^{\unord{x}}}_{\vert\cE_i}$.
	Thus, for all~$\unord{x}\in\unordsubs{V}{k-1}$, we have
	\begin{align}\label{eq:numofocc}
		\big\vert\{i\in[r]\colon d_{F\cap\cC'_{i}}(\unord{x})\geq 1 \}\big\vert\leq\sum_{i\in[r]}Y_i^{\unord{x}}\,.
	\end{align}
	
	By definition, we have that~$\mathds{P}[Y^{\unord{x}}_{i}=1\mid (Y^{\unord{x}}_j)_{j\in[i-1]}]\leq \sum_{\unord{e}\in\unordsubs{V}{k}\colon\unord{x}\subseteq\unord{e}}p_{\tau (\unord{e},i)}$ holds for all~$i\in[r]$ and this entails that for all~$\unord{x}\in\unordsubs{V}{k-1}$, we have
	\begin{align*}
	\MoveEqLeft[4]
	\sum_{i\in[r]}\EE\big[Y_i^{\unord{x}}=1\mid (Y_j^{\unord{x}})_{j\in[i-1]}\big]\leq \sum_{\unord{e}\in\unordsubs{V}{k}\colon\unord{x}\subseteq\unord{e}}\sum_{\tau\in\cT}p_{\tau}\vert\mathcal{I}_{\tau}(\unord{e})\vert\\
	\overset{\text{\ref{it:leftover}--\ref{it:j-end}}}{\leq}&n\Big(\mu r\cdot\frac{1}{\delta^kn^{k-1}}+r\cdot\frac{\delta^{1/2}}{n^{k-1}}+2\sum_{j\in[k-1]\setminus\{1\}}2^{L^2}n^{k-j}\frac{1}{\delta^kn^{k-1}}+\sum_{j\in[k-1]}\frac{n^{k-j}}{L^{1/2}}\frac{1}{\delta^kn^{k-j}}\Big)\leq\frac{\delta^{1/3} n}{2}\,.
	\end{align*}
	\color{black}
	Thus, we obtain by Freedman's inequality (Lemma~\ref{lem:freedman}) that
	\begin{align}\label{eq: app freedman}
	\mathds{P}\Big[\sum_{i\in[r]}Y^{\unord{x}}_i\geq \delta^{1/3} n\Big]\leq \exp(-n^{1/2})\,.
	\end{align}
	Suppose now that~$\sum_{i\in[r]}Y_i^{\unord{x}}\leq\delta^{1/3}n$ holds for all~$\unord{x}\in\unordsubs{V}{k-1}$.
	Then by~\eqref{eq:numofocc}, we conclude that~$d_{F\cap\cC'_{<i}}(\unord{x})\leq\delta^{1/4}n$ holds for all~$i\in[r]$ and~$\unord{x}\in\unordsubs{V}{k-1}$, meaning that the event~$\cE_i$ occurs for all~$i\in[r]$.
	Consequently, by~\eqref{eq: app freedman} and the union bound we conclude that $$\mathds{P}\Big[\bigcap_{i\in[r]}\cE_i\Big]\geq\mathds{P}\biggr[\sum_{i\in[r]}Y^{\unord{x}}_i\leq \delta^{1/3}n\text{ for all }\unord{x}\in \unordsubs{V}{k-1}\biggr]>0\,.$$
	Thus, with positive probability it happens that~$\cE_i$ occurs for all~$i\in[r]$, meaning that in each step~$i$ of this random process we construct a copy of~$\cC_i$ that is edge-disjoint from all previously constructed cycle factors.
	This yields edge-disjoint copies of~$\cC_1,\dots,\cC_r$, as desired.
\end{proof}

\section{Proof of Theorem~\ref{thm: main}}\label{sec: comb}

We now prove our main theorem by combining Propositions~\ref{prop: simultaneous path cover} and~\ref{prop:cycdecomptoHCdecomp}.

\begin{proof}[Proof of Theorem~\ref{thm: main}]
	Suppose~$1/n\ll \rho, 1/L\ll \mu \ll \rho_F \ll \eps \ll \eta, 1/k$.
	Suppose~$H$ is an~$\eta$-intersecting~$\rho$-almost~$r$-regular~$k$-graph on~$n$ vertices. Let~$r'\defeq (1-\eps)r/k$ and suppose~$\cC_1,\ldots,\cC_{r'}$ are cycle factors, whose girth is at least~$L$.
	We show that there is a suitable spanning subgraph~$F$ of~$H$ such that Proposition~\ref{prop: simultaneous path cover} guarantees the existence of collections~$\cP_1,\ldots,\cP_{r'}$ of paths in~$H-F$ that we can connect to cycles forming a copy of~$\cC_i$ for all~$i\in[r']$ via Proposition~\ref{prop:cycdecomptoHCdecomp}.

	Let~$V\defeq V(H)$ and~$E\defeq E(H)$. Let~$F$ be a random spanning subgraph of~$H$ for which every edge~$e\in E$ is included in~$E(F)$ independently at random with probability~$p\defeq \eps-2\rho$.
	Then Chernoff's inequality (Lemma~\ref{lemma: chernoff}) and the union bound show that with positive probability~$F$ is an~$\eps^2\eta/2$-intersecting~$2\rho$-almost~$pr$-regular spanning subgraph of~$H$ such that~$H-F$ is~$\eta/2$-intersecting.
	From now on, let~$F$ denote such a subgraph.

	For all~$v\in V$, we have\COMMENT{%
		use~$(1-\rho)-(1+2\rho)(\eps-2\rho)=1-\rho-\eps+2\rho-2\eps\rho+4\rho^2\geq 1-\rho-\eps+2\rho-2\eps\rho=1-\eps+\rho(1-2\eps)\geq 1-\eps$
	}
	\begin{equation*}
		d_{H'}(v)\geq ((1-\rho)-(1+2\rho)(\eps-2\rho))r\geq kr'
	\end{equation*}
	and\COMMENT{%
		use~$(1+\rho)-(1-2\rho)(\eps-2\rho)=1+\rho-\eps+2\rho+2\eps\rho-4\rho^2\leq 1+3\rho-\eps+2\rho\eps\leq 1+3\rho-\eps+2\rho\eps+\rho-6\rho\eps=(1+4\rho)(1-\eps)$
	}
	\begin{equation*}
		d_{H'}(v)\leq ((1+\rho)-(1-2\rho)(\eps-2\rho))r\leq (1+4\rho)kr'.
	\end{equation*}
	Thus with~$4\rho$,~$L^{1/3}$,~$\eta/2$,~$\mu$,~$H'$ playing the roles of~$\rho$,~$L$,~$\eta$,~$\mu$,~$H$, Proposition~\ref{prop: simultaneous path cover} yields edge-disjoint collections~$\cP_1,\ldots,\cP_{r'}$ of~$L^{1/3}$-paths in~$H'$ with~$\abs{V(\cP_i)}\geq (1-\mu)n$ for all~$i\in[r']$ such that the following holds for all~$\unord{e}\in\unordsubs{V}{k}$.
	\begin{itemize}
		\item $\abs{\cI_{\typelo}(\unord{e})}\leq\mu r'$;
		\item $\abs{\cI_{\typecon[$j$]}(\unord{e})}\leq n^{k-j}/(\eta/2)^{2L^{1/3}}$ for all~$j\in[k-1]$ and~$\abs{\cI_{\typecon[$k$]}(\unord{e})}\leq n/(\eta/2)^{2L^{1/3}}$;
		\item $\abs{\cI_{\typeend[$j$]}(\unord{e})}\leq n^{k-j}/L^{1/6}$ for all~$j\in[k-1]$.
	\end{itemize}
	Since~$F$ is~$2\rho$-almost regular,~$F$ is in particular~$\rho_F$-almost regular. Consequently, with~$L^{1/3}$,~$\mu$,~$\rho_F$, $\eps^2\eta/2$,~$H$,~$F$ playing the roles of~$L$,~$\mu$,~$\rho$,~$\eta$,~$H$,~$F$, Proposition~\ref{prop:cycdecomptoHCdecomp} yields copies of the given cycle factors as desired.
\end{proof}

\section{Concluding remarks}
In this paper, we prove a strong generalization of both the well-known Dirac-type result for the existence of one Hamilton cycle in large $k$-graphs due to R\"odl, Ruci\'nski, and Szemer\'edi~\cite{RRS:06,RRS:11,RRS:08} and a result concerning asymptotically optimal packings of Hamilton cycles in graphs by Ferber, Krivelevich, and Sudakov~\cite{FKS:17}.
In fact, our result even applies to cycle factors of large girth.

It was recently proved independently by Lang and Sanhueza-Matamala~\cite{LS:20} and by Polcyn, Reiher, R\"odl, and Sch\"ulke~\cite{PRRS:20} that every large $k$-graph on $n$ vertices with $\delta_{k-2}(H)\geq (5/9+o(1))\binom{n}{2}$ contains a Hamilton cycle.
We wonder whether such $k$-graphs actually contain $(1-o(1))\reg_k(H)/k$ edge-disjoint Hamilton cycles.

There are yet other sufficient conditions for Hamiltonicity in hypergraphs, see for instance~\cite{S:19}.
It would be interesting to know which other sufficient conditions imply a packing result similar to Theorem~\ref{thm: simple min_deg}.
Another tempting question in this direction is as follows.
Call a~$k$-graph~$H$ \emph{robustly Hamiltonian} if we can delete $o(n)$ edges incident to each $(k-1)$-set and~$H$ still contains a Hamilton cycle.
Does every large robustly Hamiltonian~$k$-graph~$H$ contain $(1-o(1))\reg_k(H)/k$ edge-disjoint Hamilton cycles?

Lastly, it would of course be desirable to obtain a real decomposition of a~$k$-graph into Hamilton cycles, or even stronger, 
to show that any~$k$-graph satisfying certain conditions contains~$\reg_k(H)/k$ edge-disjoint Hamilton cycles.
However, even decompositions of cliques into Hamilton cycles are in general not known to exist.
Further, it was recently shown by Piga and Sanhueza-Matamala~\cite{PS:21} that there are arbitrarily large~$3$-graphs~$H$ with~$\delta_{2}(H)\geq(\frac{2}{3}-o(1))\vert V(H)\vert$ which do not contain~$\reg_3(H)/3$ edge-disjoint Hamilton cycles.
Therefore, our results cannot be improved to exact decompositions without increasing the lower bound on the minimum degree significantly.

\bibliographystyle{amsplain}
\bibliography{ReferencesLocal}

\end{document}